\documentclass[a4paper, oneside, 10pt]{amsart}

\usepackage[T1]{fontenc}              
\usepackage[latin1]{inputenc}         
\usepackage{geometry}                 
\usepackage{mathtools}
\usepackage{todonotes}

\usepackage{color}                      
\definecolor{karin}{RGB}{0,100,200}
\definecolor{laertis}{RGB}{200,0,100}
\definecolor{mads}{RGB}{11,113,147}

\usepackage{amsmath,amsfonts,amssymb} 
\usepackage{bbm}
\usepackage{amsthm}                   
\usepackage{graphicx}                 
\usepackage{subfig}                   
\usepackage{listings}                 
\usepackage{enumerate}
\usepackage{tikz-cd}
\usepackage{rotating}
\usepackage{soul} 
\usepackage{hyperref} 
\usepackage[capitalize]{cleveref}
\usetikzlibrary{shapes.geometric}

\newtheorem{lem}{Lemma}[section]
\newtheorem{prop}[lem]{Proposition}
\newtheorem{cor}[lem]{Corollary}
\newtheorem{thm}[lem]{Theorem}

\theoremstyle{definition}
\newtheorem{defin}[lem]{Definition}
\newtheorem{remark}[lem]{Remark}

\newtheorem{setup}{Setup}

\newtheorem{question}{Question}
\newtheorem{example}[lem]{Example}
\numberwithin{equation}{section}

\newcommand{\isom}{\simeq}
\newcommand{\comp}{\circ}
\newcommand{\graph}[1]{\overline{#1}}

\DeclareMathOperator{\charac}{char}

\DeclareMathOperator{\CP}{\chi}
\DeclareMathOperator{\rad}{rad}
\DeclareMathOperator{\soc}{soc}

\DeclareMathOperator{\Thick}{Thick}

\DeclareMathOperator{\gldim}{gldim}
\DeclareMathOperator{\pdim}{pdim}
\DeclareMathOperator{\idim}{idim}
\DeclareMathOperator{\Ext}{Ext}

\DeclareMathOperator{\gr}{gr}
\DeclareMathOperator{\level}{lvl}
\DeclareMathOperator{\maxlevel}{max.lvl}
\DeclareMathOperator{\minlevel}{min.lvl}

\DeclareMathOperator{\add}{add}
\DeclareMathOperator{\I}{I}
\DeclareMathOperator{\J}{J}
\DeclareMathOperator{\Htpycat}{K}
\DeclareMathOperator{\Hom}{Hom}
 
\DeclareMathOperator{\End}{End}

\DeclareMathOperator{\modu}{mod}

\DeclareMathOperator{\db}{D^b}
\DeclareMathOperator{\derived}{\db(\modu \Lambda)}

\DeclareMathOperator{\topu}{top}

\DeclareMathOperator{\op}{op}

\DeclareMathOperator{\D}{D}

\DeclareMathOperator{\Img}{Im}

\DeclareMathOperator{\dimv}{\textbf{dim}}

\DeclareMathOperator{\K}{\mathbf{k}}

\DeclareMathOperator{\CC}{\mathbb{C}}
\DeclareMathOperator{\ZZ}{\mathbb{Z}}
\DeclareMathOperator{\QQ}{\mathbb{Q}}
\DeclareMathOperator{\RR}{\mathbb{R}}
\DeclareMathOperator{\proj}{proj}

\DeclarePairedDelimiter\abs{\lvert}{\rvert}%

\usepackage[foot]{amsaddr}

\begin{document}

\title[\resizebox{5.5in}{!}{Higher homological algebra for one-point extensions of bipartite hereditary algebras and spectral graph theory}]{Higher homological algebra for one-point extensions of bipartite hereditary algebras and spectral graph theory}

\author[K. M. Jacobsen]{Karin M. Jacobsen}
\address{Department of mathematics, Aarhus University, Aarhus, Denmark}
\email{karin.jacobsen@math.au.dk}
\author[M. H. Sand\o y]{Mads Hustad Sand\o y}
\address{Department of mathematical sciences, NTNU, Trondheim, Norway}
\email{mads.sandoy@ntnu.no}
\author[L. Vaso]{Laertis Vaso}
\address{Department of mathematical sciences, NTNU, Trondheim, Norway}
\email{laertis.vaso@ntnu.no}

\begin{abstract}
    In this article we study higher homological properties of $n$-levelled algebras and connect them to properties of the underlying graphs. Notably, to each $2$-representation-finite quadratic monomial algebra $\Lambda$ we associate a bipartite graph $\graph{B_{\Lambda}}$ and we classify all such algebras $\Lambda$ for which $\graph{B_{\Lambda}}$ is regular or edge-transitive. We also show that if $\graph{B_{\Lambda}}$ is semi-regular, then it is a reflexive graph.
\end{abstract}

\maketitle

\setcounter{section}{0}

\section{Introduction}
\label{Section:Introduction}
\subsubsection*{Background} A celebrated result of Gabriel classifies hereditary representation-finite algebras over an algebraically closed field $\K$ by the Dynkin diagrams \cite{GAB72}. The representation theory of these algebras is completely understood and can be described using Auslander--Reiten theory, see for example \cite{ARS97, ASS06}. 

In 2004 Iyama introduced an $n$-dimensional version of Auslander--Reiten theory where $n$ is a positive integer \cite{Iya07b, ARtheoryrevisited}. By applying this theory one obtains the class of $n$-representation-finite algebras that can be seen as certain generalisations of hereditary representation-finite algebras. Examples of $n$-representation-finite algebras include higher Auslander algebras of type $A$ \cite{Iya11} and algebras of the form $\Lambda_1\otimes\cdots\otimes\Lambda_k$ where each $\Lambda_i$ is an $\ell$-homogeneous $n_i$-representation-finite algebra over a perfect field $\K$ \cite[Corollary 1.5]{HerschendIyama10}.

Although higher dimensional Auslander--Reiten theory has found applications in fields as diverse as algebraic geometry \cite{IW, IW2, IW3, HIMO, JKM}, combinatorics \cite{OT,HJ, Wil2022}, higher category theory and algebraic K-theory \cite{DJW} and symplectic geometry \cite{DJL, Ded23}, and although many classical results from representation theory have been generalised in this setting \cite{Iya07a,HIO14, Jor2014}, it seems that an explicit classification of $n$-representation-finite algebras for $n>1$ is currently out of reach. However, some more specialised versions of this question have been adressed by imposing additional constrains, e.g. \cite{DI20, VASradsquarezero, HJS22}. In particular, in \cite{HI11} Herschend and Iyama show that all basic $2$-representation-finite algebras are described using cuts of selfinjective quivers with potential and examples of such nature are given in \cite{Pas20, Pet19}. 

In a similar manner the authors of \cite{ST24} tried to classify $n$-representation-finite quadratic monomial bound quiver algebras. They gave a complete classification when $n>2$ and in the case $n=2$ they showed that if a quadratic monomial bound quiver algebra $\Lambda=\K Q/\I$ is $2$-representation-finite, then the quiver $Q$ is a \emph{star quiver $S_{(r,s)}$} given by
\begin{equation}\label{eq:star quiver}
\begin{tikzpicture}[scale=0.8, transform shape, baseline={(current bounding box.center)}]]
    \node (v1) at (-0.5,1.75) {$x_1$};
    \node (vr) at (-0.5,0.25) {$x_r$};    
    \node (u1) at (2.5,1.75) {$y_1$};  
    \node (us) at (2.5,0.25) {$y_s$.};               \node (z) at (1,1) {$z$};
    
    \draw[loosely dotted] (-0.5,0.7) to (-0.5,1.3);  \draw[loosely dotted]  (2.5,0.7) to (2.5,1.3);
    \draw[->] (v1) to node[above] {$a_1$}  (z);
    \draw[->] (vr) to node[below] {$a_r$}  (z);
    \draw[->] (z) to node[above] {$b_1$}  (u1);
    \draw[->] (z) to node[below] {$b_s$}  (us);
\end{tikzpicture}
\end{equation}

\subsubsection*{Main results} One aim of this article is to continue the classification in \cite{ST24} by studying which of these algebras are indeed $2$-representation-finite. Note that the data defining a quadratic monomial bound quiver algebra $\Lambda=\K S_{(r,s)}/\I$ is equivalent to that of a quiver $B_{\Lambda}$ with vertices $\{x_1,\ldots,x_r,y_1,\ldots,y_s\}$ and arrows $\beta_{ji}:y_j\to x_i$ whenever $a_i b_j\not\in\I$. Note that the underlying graph $\overline{B_{\Lambda}}$ is bipartite with partition given by $\{x_1,\ldots,x_r\}$ and $\{y_1,\ldots,y_s\}$; if all of the vertices $x_1,\ldots,x_r$ have the same degree $\Sigma_1$ and all of the vertices $y_1,\ldots,y_s$ have the same degree $\Sigma_2$, then $\overline{B_{\Lambda}}$ his called \emph{semi-regular of bidegree $(\Sigma_1,\Sigma_2)$}. Our main result is the following.
    
\begin{thm}[Corollary \ref{cor:2-RF semi-regular star algebra}]\label{thm:2-rep-fin semiregular}
    Let $\Lambda=\K S_{(r,s)}$ be a $2$-representation-finite $(r,s)$-star algebra. If the underlying graph $\graph{B_{\Lambda}}$ of $B_{\Lambda}$ is semi-regular of bidegree $(\Sigma_1,\Sigma_2)$, then 
    \begin{enumerate}
        \item[(a)] $r\Sigma_1 = s\Sigma_2 $, and
        \item[(b)] $\Sigma_1\Sigma_2+r+s-r\Sigma_1 \in\{1,2,3,4\}$.
    \end{enumerate}
\end{thm}

By solving the system of linear Diophantine equations given by parts (a) and (b) of Theorem \ref{thm:2-rep-fin semiregular} and using the existing literature on graph theory, we can give the following classification result. 

\begin{thm}[Proposition \ref{prop: classification in the regular case} and Proposition \ref{prop: classification in the edge-transitive case}]\label{thm:classification results}
Assume that $\charac(\K)=0$. Let $\Lambda=\K S_{(r,s)}/\I$ be a $2$-representation-finite $(r,s)$-star algebra. 
\begin{enumerate}
    \item[(i)] If $B_{\Lambda}$ is regular, then the underlying graph $\graph{B_{\Lambda}}$ is one of the following:
\begin{enumerate}[(a)]
    \item the bipartite complement of the path graph $P_2$;  
    \item the cycle graph $C_8$; 
    \item the Heawood graph (G-5 in the Encyclopedia of graphs \cite{EoG}) or its bipartite complement.
\end{enumerate}    
    \item[(ii)] If $B_{\Lambda}$ is edge-transitive, then the underlying graph $\graph{B_{\Lambda}}$ is one of graphs appearing in (a)---(c) or one of the following:
\begin{enumerate}
    \item[(d)] the graph G-9P730 in the Encyclopedia of Graphs or its bipartite complement;
    \item[(e)] the graph G-9P731 in the Encyclopedia of Graphs or its bipartite complement.
\end{enumerate}
\end{enumerate}
\end{thm}

We refer the interested reader to (\ref{eq:the A3 example})--(\ref{eq:the bipartite complement of the graph G-9P731 example}) for a detailed description of these $2$-representation-finite algebras. We note that in all examples we know of a quadratic monomial algebra $\Lambda$ being $2$-representation-finite, the graph $\graph{B_{\Lambda}}$ is both edge-transitive and semi-regular. We thus pose the following question.

\begin{question}\label{question:are they all semi-regular edge-transitive}
    Does there exist a $2$-representation-finite quadratic monomial algebra $\Lambda=\K S_{(r,s)}/\I$ such that $\graph{B_{\Lambda}}$ is not semi-regular or edge-transitive? 
\end{question}

If the answer to Question \ref{question:are they all semi-regular edge-transitive} is negative, then Theorem \ref{thm:classification results} provides a classification of all $2$-representation-finite quadratic monomial algebras when $\charac(\K)=0$.
    
\subsubsection*{Applications} Theorem \ref{thm:2-rep-fin semiregular} provides an unexpected connection between representation theory, graph theory and linear Diophantine equations. Before we mention one application of this connection, let us recall that a graph whose second largest eigenvalue is less than or equal to $2$ is called a \emph{reflexive} graph \cite{KS14}; if furthermore the graph is bipartite and the largest eigenvalue is strictly greater than $2$ then the graph is called a \emph{Salem} graph \cite{McKS05}. Salem graphs can be thought of as a generalization of the Dynkin and extended Dynkin diagrams.  
\begin{thm}[Corollary \ref{cor:the graph B_Lambda is reflexive and almost always Salem}]\label{thm:2-rep-fin gives Salem graph}
    Assume that $\charac(\K)=0$. Let $\Lambda=\K S_{(r,s)}/\I$ be a $2$-representation-finite $(r,s)$-star algebra. Assume that $\Lambda$ is semi-regular. Then $\graph{B_{\Lambda}}$ is a reflexive graph. Moreover, $\graph{B_{\Lambda}}$ is a Salem graph if and only if $(r,s)\neq (1,1)$ and $(r,s)\neq (4,4)$.
\end{thm}

Most of our work is done in a more general setting where $S_{(r,s)}$ is replaced by an $n$-levelled quiver and $2$-representation-finiteness is replaced by being twisted fractionally Calabi--Yau; the results in this general setting may be of independent interest.

\subsubsection*{Outline of the paper} In Section \ref{Section:Background} we recall the necessary background for this article. In Section \ref{Section:n-levelled algebras} we study $n$-levelled algebras from the point of view of higher homological algebra. In Section \ref{Section:One-point extensions of bipartite hereditary algebras} we study one-point extensions of bipartite hereditary algebras and explain how they are related to star quivers. In Section \ref{Section:2-representation-finite quadratic monomial algebras} we focus on $2$-representation-finite quadratic monomial algebras and prove the main results of this article. In Appendix \ref{Section:Classification in regular and edge-transitive case} we investigate which solutions of the system of linear Diophantine equations found in Appendix \ref{Section:Solutions to the system of linear Diophantine equations} give rise to $2$-representation-finite algebras in the regular and edge-transitive case.  

\section*{Acknowledgements}
The authors would like to thank Darius Dramburg, Martin Herschend and \O yvind Solberg for useful comments and discussions.

The first named author is grateful for support from the  Danish National Research Foundation (DNRF156) National Research Foundation (DNRF156), and the Norwegian Research Council via the project "Higher homological algebra and tilting theory" (301046).

The second named author is grateful to have been supported by Norwegian Research Council project 301375, ``Applications of reduction techniques and computations in representation theory''.

The third named author gratefully acknowledge the support and resources provided by the Centre for Advanced Studies at the Norwegian Academy of Science for their kind hospitality during the preparation of part of this paper. 

\section{Background}
\label{Section:Background}
\subsection{Conventions and notation}
For a positive integer $d\geq 1$, we denote by $I_d$ or simply $I$ the $d\times d$ identity matrix and by $0_d$ or simply $0$ the $d\times d$ zero matrix. We write the composition of morphisms 
\[
X \xrightarrow{f} Y \xrightarrow{g} Z
\]
as $g \circ f$. 

We fix an algebraically closed field $\K$. Throughout this paper, by an algebra we mean a finite-dimensional associative $\K$-algebra and by a $\Lambda$-module we mean a finitely generated right $\Lambda$-module. All quivers in this article are connected unless stated otherwise. For standard results on representation theory of algebras and quivers we refer the reader to \cite{ARS97,ASS06}.

For a finite-dimensional algebra $\Lambda$, we let $\modu \Lambda$ be the category of finite-dimensional $\Lambda$-modules. Similarly, $\gr \Lambda$ denotes the category of finite-dimensional graded $\Lambda$-modules with morphisms consisting of homogeneous homomorphisms of degree zero. Given a graded $\Lambda$-module $M= \oplus_{i\in\ZZ}M_i$, we define the \emph{$j$-th graded shift} of $M$ to be the graded $\Lambda$-module $M\langle j \rangle$ with $M\langle j \rangle_i = M_{i-j}$. We denote by $D$ the duality $D(-) = \Hom_{\K}(-,\K)$ between $\modu \Lambda$ and $\modu \Lambda^{\text{op}}$. For a $\Lambda$-module $M$ we denote by $\pdim(M)$ respectively $\idim(M)$ the projective, respectively injective, dimension of $M$. 

Let $Q$ be a quiver. We write the composition of arrows
\[
\begin{tikzcd}
i \xrightarrow{\alpha}j \xrightarrow{\beta} k
\end{tikzcd}
\]
in $Q$ as $\alpha\beta$. The underlying graph of the quiver $Q$ will be denoted by $\graph{Q}$. 

Given a bound quiver algebra $\Lambda=\K Q/\I$ and $M\in\modu\Lambda$, the \emph{dimension vector of $M$} is defined to be the vector $\dimv(M)$ given by
\[
\dimv(M) = (\dim_{\K}(Me_1),\dots,\dim_{\K}(Me_n))^T, 
\]
where $e_1,\dots,e_n$ are the idempotents of $\Lambda$ corresponding to the vertices of $Q$. We denote the \emph{Loewy length} of $M$ by $\ell\ell_{\Lambda}(M)$ or simply $\ell\ell(M)$ when the algebra $\Lambda$ is clear from context. We use the convention that the \emph{Cartan matrix} $C_{\Lambda}$ of $\Lambda$ is given by the matrix consisting of the dimension vectors of the indecomposable projective $\Lambda$-modules arranged in order from top to bottom as row vectors. 
Given this, the \emph{Coxeter matrix} of $\Lambda$ is given by $-C_{\Lambda}^{T} C_{\Lambda}^{-1}$. The \emph{Coxeter polynomial} $\CP_{\Lambda}(t)$ of $\Lambda$ is the characteristic polynomial of the Coxeter matrix of $\Lambda$, that is
\[
\CP_{\Lambda}(t) = \det(tI-(-C_{\Lambda}^TC_{\Lambda}^{-1})).
\]
Note that the Coxeter polynomial of an algebra $\Lambda$ is a monic integer polynomial that is self-reciprocal, i.e.\ its coefficients form a palindromic sequence \cite{Lukas}, or, equivalently, it is a polynomial $\CP_{\Lambda}(t)$ of degree $n$ satisfying $$\CP_{\Lambda}(t) = t^n \CP_{\Lambda}(t^{-1}).$$ 

\subsection{Spectral graph theory}

By a \emph{graph} we mean a simple undirected graph $G = (V, E)$ where $V$ is the set of vertices and $E$ is the set of edges. All graphs in this article are \emph{finite}, that is $V$ and $E$ are finite sets. To any graph $G$ one can associate its \emph{adjacency matrix} $A(G)$ by choosing some ordering on the vertices $V$ and setting $(A(G))_{i,j} = 1$ when there is an edge between $i$ and $j$ and letting $(A(G))_{i,j} = 0$ otherwise. Clearly, the adjacency matrix of a graph is symmetric, and since it is also real, its spectrum is real as well. Hence, its eigenvalues are ordered, say as $\lambda_{n} \leq \lambda_{n-1} \leq \ldots \leq \lambda_1$.

Recall that a graph $G$ is called \emph{bipartite} if there is a partition $V = X \cup Y$ such that every edge has one endpoint in $X$ and the other in $Y$. We call such a partition a \emph{coloring} of $G$. It is easy to see that a graph $G$ is bipartite if and only if the adjacency matrix of $G$ can be given in the form
\[
A(G) = 
\begin{bmatrix} 
0_r & R \\
R^T & 0_s
\end{bmatrix}
\]
with $R$ some $r\times s$-matrix with zeroes and ones. The spectrum of a bipartite graph is symmetric about the origin, and this property in fact characterizes the class of bipartite graphs. 

The \emph{degree} of a vertex of a graph is the number of edges incident with it. 
A graph $G$ is \emph{regular} if every vertex has the same degree. 
A bipartite graph $G = (X \cup Y, E)$ is \emph{semi-regular} if every vertex in $X$ has degree $m \in \mathbb{N}$ and every vertex in $Y$ has degree $n \in \mathbb{N}$, and in this case we say that $G$ has \emph{bidegree} $(m,n)$. Note that in this case we have $m\abs{X}=n\abs{Y}$ by counting the number of edges in the graph.

The graphs whose largest eigenvalue is less than or equal to $2$ are called \emph{Smith} or \emph{cyclotomic} graphs and have been classified: they coincide with the class of Dynkin and extended Dynkin diagrams \cite{SMI}. The case where the bound is sharp corresponds to the Dynkin case, whereas equality corresponds to the extended Dynkin case. All but one of the classes of Smith graphs are bipartite. The graphs whose second largest eigenvalue is less than or equal to $2$ form another class of much studied graphs called \emph{reflexive} graphs \cite{KS14}. Furthermore, bipartite reflexive graphs whose largest eigenvalue is strictly greater than $2$ are called \emph{Salem} graphs \cite{McKS05}. These notions have connections to hyperbolic geometry \cite{Discrete-hyperbolic-geometry}. 

Given a bipartite graph $G = (X \cup Y, E)$ one can consider \emph{its bipartite complement}, namely the graph obtained by taking the same vertex set as $G$ but having edge set consisting of every edge going between $X$ and $Y$ that is \emph{not} in $E$. For an example, see figure \ref{fig:bipartite complement}.

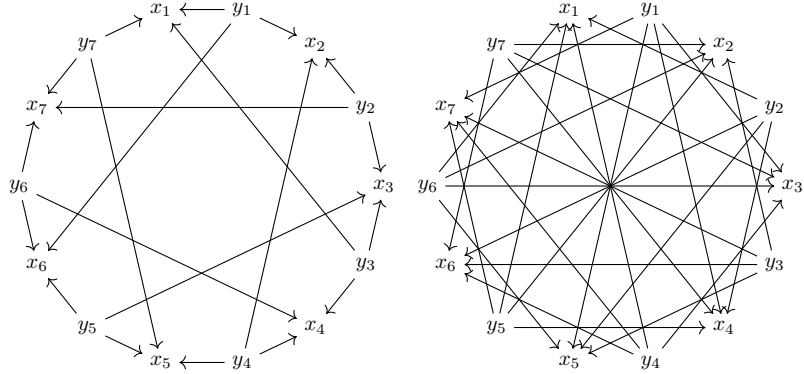
\begin{figure}
    \centering
    \begin{tikzpicture}[scale=0.8, transform shape, baseline={(current bounding box.center)}]

\def\n{14}

\def\radius{3}

    \node (x3) at ({360/\n * (1 - 1)}:\radius) {$x_3$};
    \node (y2) at ({360/\n * (2 - 1)}:\radius) {$y_2$};
    \node (x2) at ({360/\n * (3 - 1)}:\radius) {$x_2$};
    \node (y1) at ({360/\n * (4 - 1)}:\radius) {$y_1$};
    \node (x1) at ({360/\n * (5 - 1)}:\radius) {$x_1$};
    \node (y7) at ({360/\n * (6 - 1)}:\radius) {$y_7$};
    \node (x7) at ({360/\n * (7 - 1)}:\radius) {$x_7$};
    \node (y6) at ({360/\n * (8 - 1)}:\radius) {$y_6$};
    \node (x6) at ({360/\n * (9 - 1)}:\radius) {$x_6$};
    \node (y5) at ({360/\n * (10 - 1)}:\radius) {$y_5$};
    \node (x5) at ({360/\n * (11 - 1)}:\radius) {$x_5$};
    \node (y4) at ({360/\n * (12 - 1)}:\radius) {$y_4$};
    \node (x4) at ({360/\n * (13 - 1)}:\radius) {$x_4$};
    \node (y3) at ({360/\n * (14 - 1)}:\radius) {$y_3$};

    \draw[->] (y1) to (x1);
    \draw[->] (y1) to (x2);
    \draw[->] (y1) to (x6);
    \draw[->] (y2) to (x2);
    \draw[->] (y2) to (x3);
    \draw[->] (y2) to (x7);
    \draw[->] (y3) to (x3);
    \draw[->] (y3) to (x4);
    \draw[->] (y3) to (x1);
    \draw[->] (y4) to (x4);
    \draw[->] (y4) to (x5);
    \draw[->] (y4) to (x2);
    \draw[->] (y5) to (x5);
    \draw[->] (y5) to (x6);
    \draw[->] (y5) to (x3);
    \draw[->] (y6) to (x6);
    \draw[->] (y6) to (x7);
    \draw[->] (y6) to (x4);
    \draw[->] (y7) to (x1);
    \draw[->] (y7) to (x7);
    \draw[->] (y7) to (x5);
\end{tikzpicture}
\begin{tikzpicture}[scale=0.8, transform shape, baseline={(current bounding box.center)}]

\def\n{14}

\def\radius{3}

    \node (x3) at ({360/\n * (1 - 1)}:\radius) {$x_3$};
    \node (y2) at ({360/\n * (2 - 1)}:\radius) {$y_2$};
    \node (x2) at ({360/\n * (3 - 1)}:\radius) {$x_2$};
    \node (y1) at ({360/\n * (4 - 1)}:\radius) {$y_1$};
    \node (x1) at ({360/\n * (5 - 1)}:\radius) {$x_1$};
    \node (y7) at ({360/\n * (6 - 1)}:\radius) {$y_7$};
    \node (x7) at ({360/\n * (7 - 1)}:\radius) {$x_7$};
    \node (y6) at ({360/\n * (8 - 1)}:\radius) {$y_6$};
    \node (x6) at ({360/\n * (9 - 1)}:\radius) {$x_6$};
    \node (y5) at ({360/\n * (10 - 1)}:\radius) {$y_5$};
    \node (x5) at ({360/\n * (11 - 1)}:\radius) {$x_5$};
    \node (y4) at ({360/\n * (12 - 1)}:\radius) {$y_4$};
    \node (x4) at ({360/\n * (13 - 1)}:\radius) {$x_4$};
    \node (y3) at ({360/\n * (14 - 1)}:\radius) {$y_3$};

    \draw[->] (y1) to (x3);
    \draw[->] (y1) to (x4);
    \draw[->] (y1) to (x5);
    \draw[->] (y1) to (x7);
    \draw[->] (y2) to (x1);
    \draw[->] (y2) to (x4);
    \draw[->] (y2) to (x5);
    \draw[->] (y2) to (x6);
    \draw[->] (y3) to (x2);
    \draw[->] (y3) to (x5);
    \draw[->] (y3) to (x6);
    \draw[->] (y3) to (x7);
    \draw[->] (y4) to (x1);
    \draw[->] (y4) to (x3);
    \draw[->] (y4) to (x6);
    \draw[->] (y4) to (x7);
    \draw[->] (y5) to (x1);
    \draw[->] (y5) to (x2);
    \draw[->] (y5) to (x4);
    \draw[->] (y5) to (x7);
    \draw[->] (y6) to (x1);
    \draw[->] (y6) to (x2);
    \draw[->] (y6) to (x3);
    \draw[->] (y6) to (x5);
    \draw[->] (y7) to (x2);
    \draw[->] (y7) to (x3);
    \draw[->] (y7) to (x4);
    \draw[->] (y7) to (x6);
\end{tikzpicture}
\caption{On the left, a bipartite graph known as the \emph{Heawood graph}. On the right, its bipartite complement.}
    \label{fig:bipartite complement}
\end{figure}

A \emph{loopless multigraph} $G$ is a graph with possibly multiple parallel edges, but with no edges starting and ending in the same vertex. Equivalently, it is a graph $G = (V,E)$ with a weight function $w \colon E \rightarrow \mathbb{Z}_{>0}$. The degree of a vertex $v$ is still defined to be the number of edges adjacent to $v$ (or, equivalently, the sum $\sum w(e)$ where $e$ ranges among all edges adjacent to $v$).
Consequently, it makes sense to speak of bipartite loopless multigraphs as well as of regular and semi-regular versions of those. Note however that there is no natural definition of a bipartite complement for loopless multigraphs. 
The adjacency matrix $A(G)$ of such a graph $G$ has $(A(G))_{i,j}$ equal to the number of edges between vertices $i$ and $j$. 
Of course, this is a symmetric integer matrix, one can consider its spectrum, its Perron-Frobenius eigenvector, and so on as usual.
Note also that the facts listed above for the spectrum of a bipartite graph also hold in this generality. 

\subsection{Bi-eigenvectors and bi-eigenvalues}

Let $G=(X\cup Y, E)$ be a semi-regular graph of bidegree $(m,n)$. Assume that $\abs{X}=r$ and $\abs{Y}=s$. Let $A=A(G)=\begin{bsmallmatrix}
    0 & R \\ R^T & 0
\end{bsmallmatrix}$ be the adjacency matrix of $A$. Then we have
\[
\begin{bmatrix}
    0 & R \\ R^T & 0
\end{bmatrix} \begin{bmatrix}
    1 \\ \vdots \\ 1 
\end{bmatrix}_{r+s} = \begin{bmatrix}
    m \\ \vdots \\ m \\ n \\ \vdots \\ n
\end{bmatrix} = \begin{bmatrix}
    m \begin{bmatrix}
        1 \\ \vdots \\ 1 
    \end{bmatrix}_r \\ n \;\begin{bmatrix}
        1 \\ \vdots \\ 1 
    \end{bmatrix}_s
\end{bmatrix}.
\]
We want to study this eigenvector-like behaviour of the all-ones vector with respect to the adjacency matrix of a semi-regular graph. Motivated by this, we introduce the following definition.

\begin{defin}\label{def:bieigenvector}
Let 
\[
A=
\begin{bmatrix}
0_r & R \\
R^T & 0_s
\end{bmatrix}
\]
be a square $\CC$-matrix, let $d_r$ be an $r\times 1$ $\CC$-vector, let $d_s$ be an $s\times 1$ $\CC$-vector and let $\Sigma_1,\Sigma_2\in\CC$.
We say that $d=\begin{bsmallmatrix} d_r \\ d_s \end{bsmallmatrix}$ is a \emph{bi-eigenvector of $A$ with bi-eigenvalue $(\Sigma_1,\Sigma_2)$} if $d$ is nonzero and both $Rd_s=\Sigma_1 d_r$ and $R^T d_r=\Sigma_2 d_s$ hold, or, equivalently, if it is nonzero and
\[
\begin{bmatrix}
0_r & R \\
R^T & 0_s
\end{bmatrix} 
\begin{bmatrix}
d_r \\
d_s
\end{bmatrix}
= 
\begin{bmatrix}
\Sigma_1 d_r \\
\Sigma_2 d_s
\end{bmatrix}.
\]
\end{defin}

\begin{example}\label{ex:eigenvector is bi-eigenvector}
Let 
\[
A=
\begin{bmatrix}
0_r & R \\
R^T & 0_s
\end{bmatrix}.
\]
Then if $v=(v_1,\dots,v_{r+s})^T$ is an eigenvector of $A$ with eigenvalue $\lambda$, we have that $((v_1,\dots,v_r)^T,(v_{r+1},\dots,v_{r+s})^T)$ is a bi-eigenvector of $A$ with bi-eigenvalue $(\lambda,\lambda)$.
\end{example}

\begin{lem}\label{lem:basic properties of bi-eigenvectors}
Let
\[
A=
\begin{bmatrix}
0_r & R \\
R^T & 0_s
\end{bmatrix}
\]
be a square $\CC$-matrix. Let $(d_r,d_s)$ be a bi-eigenvector of $A$ with bi-eigenvalue $(\Sigma_1,\Sigma_2)$. Then the following hold.
\begin{enumerate}[(a)]
    \item We have $RR^T d_r=\Sigma_1\Sigma_2 d_r$. In particular, if $d_r\neq 0$, then $d_r$ is an eigenvector of $RR^T$ with eigenvalue $\Sigma_1\Sigma_2$.
    \item We have $R^TR d_s=\Sigma_1\Sigma_2 d_s$. In particular, if $d_s\neq 0$, then $d_s$ is an eigenvector of $R^TR$ with eigenvalue $\Sigma_1\Sigma_2$.
    \item We have $\Sigma_1 \lvert d_r\rvert^2 = \Sigma_2 \lvert d_s \rvert^2$.
\end{enumerate}
\end{lem}

\begin{proof}
Parts (a) and (b) follow immediately from $Rd_s=\Sigma_1 d_r$ and $R^T d_r=\Sigma_2 d_s$. For part (c) we have
\[
\Sigma_1 \lvert d_r \rvert^2 = d_r^T(\Sigma_1 d_r) = d_r^T (R d_s) = (R^T d_r)^T d_s = \Sigma_2 d_s^T d_s = \Sigma_2 \lvert d_s\rvert^2.  \qedhere
\] 
\end{proof}

We are especially interested in nonnegative integer bi-eigenvectors of adjacency matrices of bipartite loopless multigraphs.

\begin{lem}\label{lem:bi-eigenvectors and nonnegative integer matrices}
Let $G$ be a bipartite loopless multigraph with adjacency matrix
\[
A= A(G)=
\begin{bmatrix}
0_r & R \\
R^T & 0_s
\end{bmatrix},
\]
and assume that $G$ has at least one edge. Let $(d_r,d_s)$ be a bi-eigenvector of $A$ with nonnegative integer entries and with bi-eigenvalue $(\Sigma_1,\Sigma_2)$. Then the following hold.
\begin{enumerate}[(a)]
    \item $\Sigma_1,\Sigma_2\in\QQ$, $\Sigma_1\Sigma_2\in\ZZ$ and $\Sigma_1\Sigma_2>0$.
    \item $d_r\neq 0$ and $d_s\neq 0$. In particular, $d_r$ is an eigenvector of $RR^T$ with eigenvalue $\Sigma_1\Sigma_2$ and $d_s$ is an eigenvector of $R^TR$ with eigenvalue $\Sigma_1\Sigma_2$.
    \item $\Sigma_1\lvert d_r\rvert^2 = \Sigma_2 \lvert d_s \rvert^2$.
    \item $\begin{bmatrix}
    \sqrt{\Sigma_1} d_r \\ 
    \sqrt{\Sigma_2} d_s
    \end{bmatrix}$ is an eigenvector of $A$ with eigenvalue $\sqrt{\Sigma_1\Sigma_2}$.
    \item The largest eigenvalue of $A$ is $\sqrt{\Sigma_1\Sigma_2}$.
\end{enumerate}
\end{lem}

\begin{proof}
\begin{enumerate}[(a)]
    \item Since $R$, $d_r$ and $d_s$ have integer entries, it follows that $\Sigma_1,\Sigma_2\in\QQ$. Next, by Lemma \ref{lem:basic properties of bi-eigenvectors}(a) and (b) we have
    \[
    \text{$RR^T d_r = \Sigma_1\Sigma_2 d_r$ \;\;and\;\; $R^TR d_s = \Sigma_1\Sigma_2 d_s$.}
    \]
    Since the entries of $R$ are nonnegative integers, it follows that $RR^T$ and $R^TR$ both have nonnegative integer entries. Since $G$ has at least one edge, we also have that $RR^T\neq 0$ and $R^TR\neq 0$. Since both $d_r$ and $d_s$ have nonnegative integer entries, but by definition $d_r$ and $d_s$ cannot both be zero vectors, it follows that $\Sigma_1\Sigma_2$ is a positive integer.
    \item By definition we have $Rd_s=\Sigma_1 d_r$ and $R^T d_r = \Sigma_2 d_s$. Since by part (a) we have $\Sigma_1\neq 0$ and $\Sigma_2\neq 0$, and since by definition at least one of $d_r$ and $d_s$ is nonzero, it follows that the other one also is nonzero. The claim about eigenvectors and eigenvalues follows by Lemma \ref{lem:basic properties of bi-eigenvectors}(a) and (b).
    \item Follows by Lemma \ref{lem:basic properties of bi-eigenvectors}(c).
    \item By part (a) we have that $\begin{bmatrix}
    \sqrt{\Sigma_1} d_r \\ 
    \sqrt{\Sigma_2} d_s
    \end{bmatrix}$ is nonzero. Then we compute
    \[
    \begin{bmatrix}
    0_r & R \\
    R^T & 0_s
    \end{bmatrix}\begin{bmatrix}
    \sqrt{\Sigma_1} d_r \\ 
    \sqrt{\Sigma_2} d_s
    \end{bmatrix} =
    \begin{bmatrix}
    R\sqrt{\Sigma_2}d_s \\
    R^T\sqrt{\Sigma_1}d_r
    \end{bmatrix}=
    \begin{bmatrix}
    \Sigma_1\sqrt{\Sigma_2}d_r \\
    \Sigma_2\sqrt{\Sigma_1}d_s
    \end{bmatrix}=
    \sqrt{\Sigma_1\Sigma_2}
    \begin{bmatrix}
    \sqrt{\Sigma_1}d_r\\
    \sqrt{\Sigma_2}d_s
    \end{bmatrix}.
    \]
    \item By (d) we have that $\begin{bmatrix}
    \sqrt{\Sigma_1} d_r \\ 
    \sqrt{\Sigma_2} d_s
    \end{bmatrix}$ has $\sqrt{\Sigma_1\Sigma_2}$ as an eigenvalue. Since $\begin{bmatrix}
    \sqrt{\Sigma_1} d_r \\ 
    \sqrt{\Sigma_2} d_s
    \end{bmatrix}$ has nonnegative entries, it follows by the Perron--Frobenius theorem that $\sqrt{\Sigma_1\Sigma_2}$ is largest in absolute value among the eigenvalues of $A$. Since all eigenvalues of $A$ are real and $\sqrt{\Sigma_1\Sigma_2}$ is positive, the claim follows. \qedhere 
\end{enumerate}
\end{proof}

\subsection{Koszul algebras}\label{subsec:Koszul algebras}
Recall that a graded algebra $\Lambda = \bigoplus_{i \geq 0} \Lambda_i$ with semisimple degree $0$ part is a \emph{Koszul algebra} if $\Lambda_0$ satisfies 
$$\Ext^{i}_{\gr \Lambda}(\Lambda_0, \Lambda_0 \langle j \rangle) = 0$$ 
whenever $i \neq j$ \cite[Proposition 2.1.3]{Beilinson-Ginzburg-Soergel}. In that case, we call
\[
\Lambda^! = \bigoplus_{i \geq 0}\Ext_{\gr \Lambda}^{i}(\Lambda_0,\Lambda_0\langle i \rangle)
\]
the\emph{ Koszul dual of }$\Lambda$. If $\Lambda=\K Q/\I$ is a bound quiver algebra, then the quiver of $\Lambda^!$ is given by the opposite quiver of $Q$; see e.g. \cite[Definition 2.8.1 and Proposition 2.10.1]{Beilinson-Ginzburg-Soergel}.

\begin{example}\label{ex:quadratic relations Koszul} 
Let $\Lambda=\K Q/\I$ be a bound quiver algebra, where the ideal $\I$ is generated by quadratic relations and $\gldim\Lambda\leq 2$. 
Then $\Lambda$ is Koszul; see e.g.\ \cite[Proposition 2.19]{Martinez-Villa-survey}. 
\end{example}

\subsection{Higher-dimensional homological algebra}

We recall some notions and results in higher-dimensional homological algebra from \cite{ARtheoryrevisited, HerschendIyama10, HIO14}. Let $\Lambda$ be a finite-dimensional algebra. We denote by
\[
\tau_n\coloneqq \tau\Omega^{n-1}:\modu\Lambda \to \modu \Lambda \text{ and } \tau_n^{-}\coloneqq \tau^{-}\Omega^{-(n-1)}:\modu\Lambda \to \modu \Lambda
\]
the \emph{$n$-Auslander--Reiten translations}. 

\begin{defin}\label{def:n-cluster tilting module}
Let $M\in\modu\Lambda$ and $n\geq 1$ be an integer. We say that $M$ is an \emph{$n$-cluster tilting} module if
\begin{align*}
    \add(M) &= \{N\in\modu \Lambda \mid \Ext^{i}_{\Lambda}(M,N)=0 \text{ for all $0<i<n$} \} \\
    &= \{N\in\modu \Lambda \mid \Ext^{i}_{\Lambda}(N,M)=0 \text{ for all $0<i<n$} \}.
\end{align*}
\end{defin}

Note that if $M\in\modu\Lambda$ is an $n$-cluster tilting module and $X\in\add(M)$ is non-projective, then $\pdim(X)\geq n$; similarly if $X\in\add(M)$ is non-injective, then $\idim(X)\geq n$. Now assume that $\Lambda$ has finite global dimension. We denote by 
\[
\nu = \nu_{\Lambda}\coloneqq (D\Lambda)\overset{\mathbf{L}}{\otimes}_{\Lambda}- : \db(\modu\Lambda)\to \db(\modu\Lambda)
\]
the \emph{(derived) Nakayama functor} and by
\[
\nu^{-} = \nu^{-}_{\Lambda}\coloneqq \mathbf{R}\Hom_{\Lambda}(D\Lambda,-): \db(\modu\Lambda)\to \db(\modu\Lambda)
\]
its quasi-inverse; for more details see \cite{Hap88}. 
It is well-known that $\nu$ is a Serre functor of $\db(\modu \Lambda)$ whenever $\Lambda$ has finite global dimension. 
We denote $\nu_{n}=\nu\comp[-n]$. 
Then we have the formulas
\begin{equation}\label{eq:connect taun and nakayama}
\tau_n = H^{0}(\nu_n(-)):\modu\Lambda \to \modu\Lambda \text{ and } \tau_n^{-}=H^{0}(\nu_n^{-}(-)):\modu\Lambda \to \modu\Lambda,
\end{equation}
see \cite{HIO14}. We also use the following result from \cite{HIO14}.

\begin{prop}\cite[Proposition 2.3]{HIO14} \label{prop:nun remains in mod Lambda} Let $M\in\modu\Lambda$. 
\begin{enumerate}
    \item[(a)] If $\nu_n^{i}(M)\in\modu\Lambda$ for some $i>0$, then $\nu_n^{j}(M)\in\modu\Lambda$ for all $0\leq j\leq i$.
    \item[(b)] If $\nu_n^{-i}(M)\in\modu\Lambda$ for some $i>0$, then $\nu_n^{-j}(M)\in\modu\Lambda$ for all $0\leq j\leq i$.
\end{enumerate}
\end{prop}

With this setup in mind, we recall the following definitions.

\begin{defin}\label{def:n-RF RI hereditary} Let $n\geq 1$ be a positive integer and $\Lambda$ a finite-dimensional algebra.
\begin{enumerate}
    \item[(a)]\cite[Proposition 2.6]{HIO14} We say that $\Lambda$ is \emph{$n$-representation-finite} if for every indecomposable projective $\Lambda$-module $P\in\modu\Lambda$ there exists some $i\geq 0$ such that $\nu_n^{-i}(P)$ is an indecomposable injective $\Lambda$-module. 
    \item[(b)]\cite[Definition 2.7]{HIO14} We say that $\Lambda$ is \emph{$n$-representation-infinite} if $\gldim\Lambda\leq n$ and for every indecomposable projective module $P\in\modu\Lambda$ and for any $i\geq 0$ we have that $\nu_n^{-i}(P)\in\modu\Lambda$.
    \item[(c)]\cite[Theorem 3.4]{HIO14} We say that $\Lambda$ is \emph{$n$-hereditary} if $\Lambda$ is either $n$-representation-finite or $n$-representation-infinite.
\end{enumerate}
\end{defin}

We collect some important results for $n$-hereditary algebras in the following remark.

\begin{remark}\label{rem:n-representation-finite}
\begin{enumerate}
    \item[(a)] By \cite[Proposition 2.6] {HIO14} we have that $\Lambda$ is $n$-representation-finite if and only if $M=\bigoplus_{i\geq 0}\tau_n^{-i}(\Lambda)$ is an $n$-cluster tilting module and $\gldim(\Lambda)\leq n$. Let $X\in\add(M)$ be non-projective. Since $\pdim(X)\geq n$, we conclude that $\pdim(X)=n$. Similarly, if $X\in\add(M)$ is non-injective, then $\idim(X)=n$.
    \item[(b)] Assume that $\Lambda$ is $n$-representation-finite. Then the functor $\tau_n$ gives a bijection from the isomorphism classes of indecomposable non-projective modules in $\add(M)$ to the isomorphism classes of indecomposable non-injective modules in $\add(M)$, with inverse given by $\tau_n^{-}$, see \cite[Theorem 2.8]{ARtheoryrevisited}. In this case we obtain
    \[
    M = \bigoplus_{i\geq 0} \tau_n^{-i}(\Lambda) = \bigoplus_{i\geq 0} \tau_n^{i}(D\Lambda).
    \]
    In particular, for each indecomposable projective module $P\in\modu\Lambda$, there exists a unique  $\ell_P>0$ such that $\tau_n^{-(\ell_P-1)}(P)$ is indecomposable injective and $\tau_n^{-\ell_P}(P)=0$. If for all indecomposable projective modules $P,P'\in\modu\Lambda$ we have $\ell_P=\ell_{P'}=:\ell$, then we say that $\Lambda$ is \emph{$\ell$-homogeneous} see \cite[Definition 1.2]{HerschendIyama10}. Thus, in general, if $\Lambda=\K Q/\I$ is a bound quiver algebra, then we obtain
    \begin{itemize}
        \item[$\bullet$] a bijection $\sigma:Q_0\to Q_0$, and
        \item[$\bullet$] for each vertex $v\in Q_0$ a positive integer $\ell_v$
    \end{itemize}
    such that $\tau_n^{-(\ell_v-1)}(e_v\Lambda)\isom D(\Lambda e_{\sigma(v)})$ and $\tau_n^{\ell_v-1}(D\Lambda e_v) \isom e_{\sigma^{-1}(v)}\Lambda$.
    \item[(c)] By part (a) and the definition of $n$-hereditary algebras, it follows that if $\Lambda$ is $n$-hereditary, then $\gldim(\Lambda)\leq n$.
    \item[(d)] Let $M\in\modu\Lambda$ be a module and assume that $\nu_n^{-i}(M)\in\modu\Lambda$. Then by (\ref{eq:connect taun and nakayama}) we have that 
    \begin{equation}\label{eq:taun and nun coincide}
    \tau_n^{-i}(M) = H^{0}(\nu_n^{-i}(M)) = \nu_n^{-i}(M).
    \end{equation}
    Assume that $\Lambda$ is $n$-representation-finite. Let $P\in\modu\Lambda$ be indecomposable projective. Then by (\ref{eq:taun and nun coincide}) and Proposition \ref{prop:nun remains in mod Lambda} we obtain that
    \begin{equation}\label{eq:where taun and nun are the same}
    \tau_n^{-j}(P)=\nu_n^{-j}(P)
    \end{equation}
    for $0\leq j\leq \ell_P-1$. 
    Similarly, if $\Lambda$ is $n$-representation-infinite we have that (\ref{eq:where taun and nun are the same}) holds for all $j\geq 0$. Corresponding statements hold by replacing projective with injective, $\tau_n^{-}$ with $\tau_n$ and $\nu_n^{-}$ with $\nu_n$.
\end{enumerate}
\end{remark}

We now recall two more definitions.

\begin{defin}\cite{HerschendIyama10}\label{def:higher preprojective}
Let $\Lambda$ be an $n$-representation-finite algebra. We denote by
\[
\Pi = \Pi(\Lambda) \coloneqq \bigoplus_{i\in\ZZ} \Hom_{\derived}(\Lambda,\nu_{n}^{-i}(\Lambda))
\]
the \emph{$(n+1)$-preprojective algebra}.
\end{defin}

\begin{defin}\cite[Def 0.3]{HerschendIyama10}\label{def:fractionally Calabi--Yau}
Let $\Lambda$ be a finite-dimensional algebra. We say that $\Lambda$ is \emph{twisted fractionally Calabi--Yau of dimension $\tfrac{m}{l}$} if the functor $\nu^{l}$ is isomorphic to the functor $(-)_{\phi}\circ[m]$ for some integers $m$ and $l\neq 0$ and $\phi$ an automorphism of $\Lambda$.
\end{defin}

The following general statements concerning twisted fractionally Calabi--Yau algebras are important for us.

\begin{thm}\cite[Theorem 1.1 and Theorem 1.3]{HerschendIyama10}\label{thm:twisted fractionally CY and n-RF}
Let $\Lambda$ be a finite-dimensional connected algebra, and let $\ell$ and $n$ be positive integers.  
\begin{enumerate}
    \item[(a)] If $\Lambda$ is $n$-representation-finite, then $\Lambda$ is twisted fractionally Calabi--Yau.
    \item[(b)] $\Lambda$ is $\ell$-homogeneous $n$-representation-finite if and only if $\Lambda$ is twisted $\tfrac{n(\ell-1)}{\ell}$-Calabi--Yau and $\gldim(\Lambda)\leq n$.
\end{enumerate}
\end{thm}

\begin{prop}\label{prop:fractionally Calabi--Yau implies eigenvalues of Coxeter matrix are in unit circle}
Let $\Lambda$ be a twisted fractionally-Calabi--Yau algebra of finite global dimension. Then all the eigenvalues of the Coxeter matrix of $\Lambda$ are on the unit circle.
\end{prop}

\begin{proof}
We follow \cite[Remark 1]{Han2022}. Let the Calabi--Yau dimension of $\Lambda$ be $\tfrac{m}{l}$. 
By \cite[Proposition 4.3(a)]{HerschendIyama10} we have that $\nu^l(\Lambda)\isom \Lambda[m]$. 
Hence $\nu^l\comp[-m]$ induces a group automorphism of the Grothendieck group of the derived category of $\Lambda$ with basis given by the indecomposable projective $\Lambda$-modules.
Moreover, the corresponding matrix $[\nu^{l}\comp[-m]]$ is a permutation matrix up to a sign and so has eigenvalues on the unit circle. 
The eigenvalues of $[\nu^{l}\comp[-m]]$ are, up to sign, equal to the eigenvalues of $[\nu^{l}]$, which are $l$-th powers of the eigenvalues of $[\nu]$, see \cite[Proposition 4.4.5]{MatrixMath}. It follows that the eigenvalues of $[\nu\circ [-1]]$ are on the unit circle. Since the Coxeter matrix corresponds to the induced action of the functor $\nu \circ [-1]$ on the Grothendieck group of the derived category of $\Lambda$ (see e.g.\ point (iii) at the start of the first section in \cite{Lenzing-delaPena'08}), the claim follows.
\end{proof}

\subsection{One-point extensions}

We recall the definition of a one-point extension of an algebra.

\begin{defin}\label{def:one-point extension}
Let $\Lambda$ be a finite-dimensional algebra and $M\in\modu \Lambda$. The \emph{one-point extension $\Lambda[M]$ of $\Lambda$ by $M$} is the algebra
\[
\Lambda[M] = 
\begin{bmatrix}
\K & M\\
0 & \Lambda
\end{bmatrix}.
\]
\end{defin}

Notice that if $\Lambda[M]$ is the one-point extension of $\Lambda$ by $M$, and if $\{e_1,\dots,e_k\}$ is a complete set of orthogonal primitive idempotents of $\Lambda$, then the Cartan matrix of $\Lambda[M]$ is given by 
\begin{equation}\label{eq:cartan matrix of one-point extension}
C_{\Lambda[M]} = 
\begin{bmatrix}
1 & \dimv(M)^T\\
0 & C_{\Lambda}
\end{bmatrix}.
\end{equation}

\begin{lem}\label{lem:quiver of one-point extension}
Let $\Lambda=\K Q/\I$ be a finite-dimensional algebra. Let $M\in\modu\Lambda$ be a module. Then the one-point extension $\Lambda[M]$ of $\Lambda$ by $M$ has a presentation $\Lambda[M]=\K Q[\omega]/\I'$ with
\begin{align*}
    Q[\omega]_0 &= Q_0\cup\{\omega\}, \\
    Q[\omega]_1 &= Q_1\cup\bigcup_{v\in Q_0}A_v, 
\end{align*}
where $A_v$ is a set of arrows $\omega\to v$ which is in bijection with a basis of $(\topu M)e_v$.
\end{lem}

\begin{proof}
Follows by the definition of a quiver of a finite-dimensional algebra, see for example \cite[Section II.3]{ASS06}.
\end{proof}

\section{\texorpdfstring{$n$}{n}-levelled algebras}
\label{Section:n-levelled algebras}
\subsection{Definition and basic properties}

We recall the following notion.

\begin{defin}\cite[Definition 3.1]{Hille'94}\label{def:n-levelled}
Let $Q$ be a quiver. 
\begin{enumerate}
    \item[(a)] We say that $Q$ is \emph{$n$-levelled} if there is a surjective function $\level \colon Q_0 \rightarrow \{0,1, \ldots, n\}$ such that if $\alpha \colon v \rightarrow u$ is an arrow in $Q$ then $\level(v) = \level(u) + 1$. For a vertex $v\in Q_0$ we call $\level(v)$ the \emph{level} of $v$. 
    \item[(b)] Let $\Lambda=\K Q/\I$ be a bound quiver algebra. We say that $\Lambda$ is \emph{$n$-levelled} if $Q$ is $n$-levelled. In this case, if $M\in\modu \Lambda$ is a module, then we say that \emph{$M$ is supported at level $i$} if $M$ is supported at a vertex $v$ with $\level(v)=i$. We also denote
\begin{align*}
\maxlevel(M) &= \max\{ i\in \{0,1,\ldots,n\}\mid \text{$M$ is supported at level i}\},\\
\minlevel(M) &= \min\{ i\in \{0,1,\ldots,n\} \mid \text{$M$ is supported at level i}\}.
\end{align*}
\end{enumerate}
\end{defin}

Note that the definition in \cite{Hille'94} has some additional assumptions which do not matter for us since all our quivers are connected.

\begin{example}\label{ex:bipartite quiver is 1-levelled}
Let $G$ be a bipartite loopless multigraph with a coloring $V=X \cup Y$. Let $Q$ be the quiver with underlying graph $\graph{Q}=G$ and bipartite orientation, that is the arrows in $Q$ are oriented from the vertices in $X$ to the vertices in $Y$. Define 
\[
\level(v)=\begin{cases} 1, &\mbox{if $v\in X$,} \\ 0, &\mbox{if $v\in Y$.}\end{cases}
\]
Then $Q$ is a $1$-levelled quiver and $\K Q$ is a $1$-levelled algebra.
\end{example}

We have the following easy observations.

\begin{lem}\label{lem:n-levelled basic properties}
Let $\Lambda=\K Q/\I$ be an $n$-levelled algebra. Then the following hold.
\begin{enumerate}
    \item[(a)] The quiver $Q$ is acyclic and  $\Lambda$ is finite-dimensional.
    \item[(b)] Let $M\in\modu\Lambda$. Then 
    \begin{enumerate}
        \item[(b1)] $\ell\ell(M)\leq \maxlevel(M)-\minlevel(M)+1$. Moreover, if there exists a path $p$ from a vertex at level $\maxlevel(M)$ to a vertex at level $\minlevel(M)$ such that $Mp\neq 0$, then equality holds.
        \item[(b2)] $\pdim(M)\leq \maxlevel(M)$. Moreover, if $\pdim(M)=n$, then $\maxlevel(M)=n$.
        \item[(b3)] $\idim(M)\leq n-\minlevel(M)$. Moreover, if $\idim(M)=n$, then $\minlevel(M)=0$.
    \end{enumerate}
    \item[(c)] $\gldim(\Lambda)\leq n$.
\end{enumerate}
\end{lem}

\begin{proof}
That $Q$ is acyclic follows immediately by the definition of an $n$-levelled quiver. Hence $\Lambda$ is finite-dimensional and part (a) holds. Part (c) follows immediately from part (b). It remains to show part (b). Since $\Lambda$ is $n$-levelled, it readily follows that 
\[
\minlevel(M) \leq \minlevel(\rad M) \leq \maxlevel(\rad M) \leq \maxlevel(M)-1.
\]
Hence the inequality in (b1) follows by induction on $\maxlevel(M)-\minlevel(M)$ and using $\ell\ell(M)=\ell\ell(\rad M)+1$. Furthermore, if such a path $p$ as in (b1) exists, then 
\[
\rad^{\maxlevel(M)-\minlevel(M)}(M) = M \rad^{\maxlevel(M)-\minlevel(M)}(\Lambda) \supseteq Mp \neq 0,
\]
and equality in (b1) holds.

Next, since $\Omega M\subseteq \rad M$, we also have that 
\[
\maxlevel(\Omega M)\leq \maxlevel(M)-1.
\]
Hence $\pdim(M)\leq \maxlevel(M)$ follows by induction on $\maxlevel(M)$ and using $\pdim(M)=\pdim(\Omega M)+1$. Notice that since $\maxlevel(M)\leq n$, it follows that $\pdim M=n$ implies $\maxlevel(M)=n$. This shows part (b2). Part (b3) follows similarly.
\end{proof}

The following statement establishes some connections between $n$-levelled algebras and the Koszul property. 

\begin{prop}\label{prop:Koszul and n-levelled}
Let $\Lambda=\K Q/\I$ be an $n$-levelled algebra. Then the following hold.  
\begin{enumerate}
    \item[(a)] Let $v,u\in Q_0$. If $\Ext^{i}_{\Lambda}(e_u\Lambda_0, e_v\Lambda_0) \neq 0$ implies that $i=\level(v)-\level(u)$, then $\Lambda$ is Koszul.
\end{enumerate}
Assume moreover that $\Lambda$ is Koszul.
\begin{enumerate}
    \item[(b)] The Koszul dual $\Lambda^!$ is $n$-levelled and we have a triangulated equivalence $\db (\modu \Lambda)\isom \db (\modu \Lambda^!)$ given by a tilting complex $\widetilde{\Lambda_0} := \bigoplus_{u \in Q_0} e_u \Lambda_0 [n-\level(u)]$.
    \item[(c)] Let $\Lambda$ be twisted $\tfrac{m}{l}$-Calabi--Yau with twisting automorphism $\phi$. 
    For a vertex $v\in Q_0$, let $\phi(v)\in Q_0$ be the unique vertex such that $(e_v\Lambda)_{\phi}\isom e_{\phi(v)}\Lambda$. If  $\level (\phi(v)) = \level (v)$ for all $v \in Q_0$, then the Koszul dual $\Lambda^{!}$ is twisted $\tfrac{m}{l}$-Calabi--Yau as well.
\end{enumerate}
\end{prop}

\begin{proof}
\begin{enumerate}
    \item[(a)] 
    This follows by \cite[Lemma 3.1]{Hille'94}.
    \item[(b)] That $\Lambda^!$ is $n$-levelled follows by the fact that its quiver is the opposite of the quiver of $\Lambda$. 
    For the second claim, we note that it can be found (implicitly) in \cite{Hille'94}, but we also include a proof for the convenience of the reader: 
    Observe that our assumption that $\Lambda$ is Koszul combined with part (a) above implies that 
    $\Hom_{\db (\modu \Lambda)}(\widetilde{\Lambda_0}, \widetilde{\Lambda_0} [i]) = 0$ for $i \neq 0$. 
    Since $\Lambda$ is finite-dimensional we know that $\Thick (\Lambda_0) = \db (\modu \Lambda)$. It is straightforward to check that $\Thick (\Lambda_0) = \Thick (\widetilde{\Lambda_0})$, and this then proves the second claim as $\Lambda$ is of finite global dimension. 
    \item[(c)] Let $\epsilon_i = \sum_{\level(j) = i} e_j$. Since $\level(\phi(v))=\level(v)$ for all $v\in Q_0$, we have that $\phi$ acts as a bijection between the vertices of $Q_0$ which are on the same level. Hence $(\epsilon_i \Lambda_0)_{\phi}\isom \epsilon_i\Lambda_0$. Then since $\Lambda$ is twisted $\tfrac{m}{l}$-Calabi--Yau with twisting automorphism $\phi$ we obtain
    \[
    \nu^{l}_{\Lambda}(\epsilon_i \Lambda_0) \isom (\epsilon_i \Lambda_0)_{\phi}[m]\isom \epsilon_{i} \Lambda_0 [m].
    \]

    Now, recall that we let $\widetilde{\Lambda_0} = \bigoplus_{u \in Q_0} e_u \Lambda_0 [n-\level(u)]$, and observe that $\epsilon_i \Lambda^{!} \simeq \mathbf{R}\Hom_{\Lambda}(\widetilde{\Lambda_0}, \epsilon_i \widetilde{\Lambda_0})$ 
    where $\epsilon_i \widetilde{\Lambda_0} = \bigoplus_{\level(j) = i} e_j \Lambda_0 [n - i]$. 
    Consequently, $\epsilon_i \Lambda^{!}$  is sent to (a cohomological shift of) $\epsilon_i \Lambda_0$ by the quasi-inverse of $\mathbf{R}\Hom_{\Lambda}(\widetilde{\Lambda_0}, -)$.
    By the uniqueness of the Serre functor, we thus get that $\nu_{\Lambda^!}^l (\Lambda^!) \simeq \Lambda^![m]$, and so the claim follows by \cite[Proposition 4.3(a)]{HerschendIyama10}. \qedhere 
\end{enumerate}
\end{proof}

\begin{example}\label{ex:2-levelled is Koszul}
Let $\Lambda=\K Q/\I$ be a $2$-levelled algebra. Then $\I$ is generated by quadratic relations and $\gldim\Lambda\leq 2$ by Lemma \ref{lem:n-levelled basic properties}. Hence $\Lambda$ is Koszul by Example \ref{ex:quadratic relations Koszul}.
\end{example}

\subsection{\texorpdfstring{$n$}{n}-levelled algebras and higher homological algebra}

In this section we investigate some properties of algebras which are both $n$-representation-finite and $n$-levelled. We start with some useful properties of indecomposable summands of the $n$-cluster tilting module of such an algebra.

\begin{prop}\label{prop:Loewy length and support for n-RF n-levelled}
Let $\Lambda=\K Q/\I$ be an $n$-representation-finite and $n$-levelled algebra for $n \geq 1$. Let $M=\bigoplus_{i\geq 0}\tau_n^{-i}(\Lambda)=\bigoplus_{i\geq 0}\tau_n^{i}(D\Lambda)$. Let $X$ be an indecomposable summand of $M$, Then the following hold: 
\begin{enumerate}
    \item[(a)] If $X$ is not projective, then $\maxlevel(X)=n$.
    \item[(b)] If $X$ is not injective, then $\minlevel(X)=0$.
    \item[(c)] If $X$ is neither projective nor injective, then $\ell\ell(X)=n+1$.
\end{enumerate}
\end{prop}

\begin{proof}
Assume that $X$ is not projective. Then $\pdim(X)=n$ by Remark \ref{rem:n-representation-finite}(a). Then part (a) follows by Lemma \ref{lem:n-levelled basic properties}(b2). Part (b) follows similarly. For part (c), we have by parts (a) and (b) that $X$ is supported at level $n$ and at level $0$. We write $X=(X_v,f_{\alpha})$ as a representation of $Q$ bound by $\I$. By Lemma \ref{lem:n-levelled basic properties}(b1) it is enough to show that there exists a path $p$ starting at a vertex at level $n$ and terminating at a vertex at level $0$ such that $f_p$ is a nonzero linear map. For a vertex $v\in Q_0$ let $U_i(v)$ denote the set of all paths in $Q$ starting at a vertex at level $i$ and terminating at the vertex $v$. Consider the representation $Y=(Y_v,g_{\alpha})$ given by
\[
Y_v = \sum_{p\in U_n(v)} \Img(f_p)
\]
and $g_{\alpha}$ given by the restriction of $f_{\alpha}$ to $Y_{s(\alpha)}$. It is easy to see that $Y$ is isomorphic to a direct summand of $X$. Since $X$ is indecomposable, it follows that $X\isom Y$ and so $Y$ is supported at levels $n$ and $0$. By definition of $Y$ there exists a path $p$ starting at a vertex at level $n$ and terminating at a vertex at level $0$ such that $g_p$ is a nonzero linear map, as required. 
\end{proof}

We continue with an important property of the indecomposable projective modules of the higher preprojective algebra in our setting.

\begin{prop}\label{prop:projectives have the same length}
Let $\Lambda$ be an algebra which is $n$-levelled, $n$-representation-finite and Koszul. Then the following hold.
\begin{enumerate}
    \item[(a)] All the indecomposable projective modules of the $(n+1)$-preprojective algebra $\Pi\coloneqq \Pi(\Lambda)$ of $\Lambda$ have the same Loewy length.
    \item[(b)] Let $v,u\in Q_0$ be two vertices. If $\level(v)=\level(u)$, then $\ell\ell_{\Lambda}(e_v\Lambda)=\ell\ell_{\Lambda}(e_u\Lambda)$.
\end{enumerate}
\end{prop}

\begin{proof}
\begin{enumerate}
    \item[(a)] By \cite[Corollary 4.13]{GrantIyama2020} we have that the algebra $\Pi$ is selfinjective. By \cite[Theorem 4.21]{GrantIyama2020} we have that $\Pi$ is almost Koszul in the sense of \cite{BBK}. In particular, $\Pi$ has only quadratic relations and thus necessarily has homogeneous relations with respect to the grading putting each arrow in degree $1$. Hence $\Pi$ satisfies the requirements of \cite[Theorem 3.3]{MartinezVilla99} from which the claim follows.
    \item[(b)] Let $w\in Q_0$ be a vertex and $S=S(w)$ be the simple $\Lambda$-module corresponding to $w$. By Lemma \ref{lem:n-levelled basic properties}(b2) we have that $\pdim(S)\leq n$ and if $\pdim(S)=n$, then $\level(w)=n$. In this case it can be easily seen that if $e_x\Lambda$ appears as a direct summand in the last nonzero term of the projective resolution of $S$, then $\level(x)=0$. It follows by \cite[Section 3]{GrantIyama2020} that the quiver of $\Pi$ is given by the quiver $Q$ with possibly some more arrows from vertices at level $n$ to vertices at level $0$ (note that the referenced article uses left modules while we use right modules). Let $l_v^{\Pi}$ be the length of a longest path $p_v^{\Pi}$ in $\Pi$ starting at the vertex $v$; similarly define $l_u^{\Pi}$ and $p_u^{\Pi}$. Then $l_v^{\Pi}=\ell\ell_{\Pi}(e_v\Pi)$ and $l_u^{\Pi}=\ell\ell_{\Pi}(e_u\Pi)$ and so by part (a) we obtain $l_v^{\Pi}=l_u^{\Pi}$. Since $\Lambda$ is $n$-levelled, and by the description of the arrows in $\Pi$, the path $p_v^{\Pi}$ is of the form $p_v^{\Pi}=\alpha_i\alpha_{i+1}\cdots \alpha_{i+k}$ where $\alpha_j$ is an arrow from a vertex in level $j$ to a vertex in level $j-1$ or from a vertex in level $0$ to a vertex in level $n$ and similarly for $p_u^{\Pi}$. Consider the subpath $p_v^{\Lambda}$ of $p_v^{\Pi}$ given by considering only the arrows in the quiver $Q$ and similarly define $p_u^{\Lambda}$. By their construction and since $\level(v)=\level(u)$, it follows that $p_v^{\Lambda}$ and $p_u^{\Lambda}$ have the same length. 
    Note that a path in $\Lambda$ is nonzero if and only if it is nonzero as a path in $\Pi$ since $\Lambda$ can be considered as the degree $0$ part of $\Pi$ with respect to some grading. 
    Hence, it follows that $p_v^{\Lambda}$ and $p_u^{\Lambda}$ are paths of maximal length in $\Lambda$ starting at $v$ and $u$ respectively. Since they have the same length, it follows that $\ell\ell_{\Lambda}(e_v\Lambda)=\ell\ell_{\Lambda}(e_u\Lambda)$. \qedhere
\end{enumerate}
\end{proof}

We are especially interested in the following special cases of $2$-levelled quivers.

\begin{defin}\label{def:light 2-levelled}
A $2$-levelled quiver $Q$ is called \emph{$i$-light} if the level $i$ is the only level which consists of exactly one vertex. A $2$-levelled bound quiver algebra $\Lambda=\K Q/\I$ is called \emph{$i$-light} if $Q$ is a $2$-levelled $i$-light quiver.
\end{defin}

A crucial example for us is the following.

\begin{defin}\label{def:star algebra}
    Let $r,s$ be integers $\geq 1$. The \emph{$(r,s)$-star quiver} $S_{(r,s)}$ is the quiver 
    
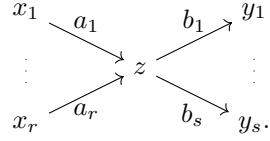
\[
        \begin{tikzpicture}[scale=1, transform shape, baseline={(current bounding box.center)}]]
        \node (v1) at (-0.5,1.75) {$x_1$};
        \node (vr) at (-0.5,0.25) {$x_r$};
        
        \node (u1) at (2.5,1.75) {$y_1$};
        \node (us) at (2.5,0.25) {$y_s$.};                
        \node (z) at (1,1) {$z$};

        \draw[loosely dotted] (-0.5,0.8) to (-0.5,1.2);         
        \draw[loosely dotted]  (2.5,0.8) to (2.5,1.2);
        \draw[->] (v1) to node[above] {$a_1$}  (z);
        \draw[->] (vr) to node[below] {$a_r$}  (z);
        \draw[->] (z) to node[above] {$b_1$}  (u1);
        \draw[->] (z) to node[below] {$b_s$}  (us);
        \end{tikzpicture}
        \]
    We say that a bound quiver algebra $\Lambda=\K Q/\I$ is an \emph{$(r,s)$-star algebra} if $Q=S_{(r,s)}$.
\end{defin}

In particular a quiver $Q$ is $2$-levelled and $1$-light if and only if $Q=S_{(r,s)}$ for some $r,s\geq 2$. Note also that for a $2$-levelled quiver $Q$ we have that $Q$ is $i$-light if and only if $Q^{\text{op}}$ is $(2-i)$-light, and similarly for a $2$-levelled bound quiver algebra $\Lambda$. In the special case of $i$-light $n$-levelled algebras and under the assumption of $2$-representation-finiteness of $\Lambda$ we know the Loewy length of indecomposable projective and indecomposable injective $\Lambda$-modules.

\begin{lem}\label{lem:Loewy length of projectives light}
Let $\Lambda=\K Q/\I$ be a $2$-levelled $i$-light and $2$-representation-finite finite-dimensional algebra. Let $v\in Q_0$ be a vertex. Then the following hold.
\begin{enumerate}
    \item[(a)] $\ell\ell_{\Lambda}(e_v\Lambda)=\level(v)+1$.
    \item[(b)] $\ell\ell_{\Lambda}(D\Lambda e_v)=3-\level(v)$
\end{enumerate}
\end{lem}

\begin{proof}
\begin{enumerate}
    \item[(a)] If $\level(v)=0$, then $v$ is a sink and the claim follows immediately. If $\level(v)=1$, then it is enough to show that there exists an arrow $\alpha\in Q_1$ starting at $v$. Assume to a contradiction that $\level(v)=1$ and there is no arrow from $v$ to the unique vertex in level $0$. Then $\ell\ell_{\Lambda}(e_v\Lambda)=1$. It follows by Proposition \ref{prop:projectives have the same length}(b) that the indecomposable projective at every vertex at level $1$ has Loewy length $1$. In particular, there exists no arrow from any vertex at level $1$ to the vertex at level $0$, contradicting $Q$ being connected and $2$-levelled.

    Now assume that $\level(v)=2$. If $i=1$, then $Q=S_{(r,s)}$ is a star quiver with $r,s\geq 2$. Let $\alpha:v\to z$ be the unique arrow starting at the vertex $v$. By \cite[Lemma 4.24]{ST24} we have that there exist at least $2$ arrows $\beta$, $\beta'$ starting at $z$ such that $\alpha\beta$ and $\alpha\beta'$ are nonzero in $\Lambda$. It follows that $\ell\ell_{\Lambda}(e_v\Lambda)$ is at least $3$. Clearly it can be at most $3$ which shows the claim in this case.

    Now assume that $i\in\{0,2\}$. Since $\Lambda$ is $0$-light if and only if $\Lambda^{\text{op}}$ is $2$-light and using the duality $D$ between $\modu\Lambda$ and $\modu\Lambda^{\text{op}}$, it is enough to show the result for $i=0$. Hence assume that $i=0$. Clearly there exists an arrow starting at $v$ since $Q$ is $2$-levelled and connected. Hence $\ell\ell_{\Lambda}(e_v\Lambda)\geq 2$. By Lemma \ref{lem:n-levelled basic properties}(b1) it is enough to show that there exists a nonzero path from $v$ to the unique vertex at level $0$. Assume to a contradiction that no such path exists. Then $\ell\ell_{\Lambda}(e_v\Lambda)=2$ and by Proposition \ref{prop:projectives have the same length}(b) we conclude that the same is true for all vertices at level $0$. Hence $\Lambda$ is a radical square zero algebra since $\Lambda$ is $2$-levelled. Then the unique vertex $v$ with $\level(v)=0$ has at least $2$ incoming arrows, since $\Lambda$ is $2$-levelled $0$-light. But this contradicts $\Lambda$ being $2$-representation-finite via \cite[Lemma 2.3]{VASradsquarezero}.
    
    \item[(b)] Notice that $\Lambda^{\text{op}}$ is also an $n$-levelled algebra with $\level_{\Lambda^{\text{op}}}(v)=2-\level_{\Lambda}(v)$ for a vertex $v\in Q_0$. Then using part (a) we have
    \[
    \ell\ell_{\Lambda}(D\Lambda e_v) = \ell\ell_{\Lambda^{\text{op}}}(e_v \Lambda^{\text{op}}) = \level_{\Lambda^{\text{op}}}(v)+1 = 2-\level_{\Lambda}(v)+1 = 3-\level_{\Lambda}(v)
    \]
    as required. \qedhere
\end{enumerate}
\end{proof}

The following is the main result for this section.

\begin{prop}\label{prop:2-levelled light and 2-RF is twisted fractionally Calabi--Yau}
Let $\Lambda=\K Q/\I$ be a $2$-levelled $i$-light and $2$-representation-finite finite-dimensional algebra. Let $\sigma:Q_0\to Q_0$ be as in Remark \ref{rem:n-representation-finite}(b). Let $v\in Q_0$ be a vertex.
\begin{enumerate}
    \item[(a)] We have $\level(v)=\level(\sigma(v))$ and $\ell\ell_{\Lambda}(e_v\Lambda)+\ell\ell_{\Lambda}(D\Lambda e_{\sigma(v)})=4$. 
    \item[(b)] There exists some positive integer $\ell$ such that $\Lambda$ is $\ell$-homogeneous. In particular, $\Lambda$ is twisted $\frac{2(\ell-1)}{\ell}$-Calabi--Yau.
    \item[(c)] The algebra $\Lambda$ is Koszul and the Koszul dual $\Lambda^!$ is $2$-levelled and $2$-representation-finite.
\end{enumerate}
\end{prop}

\begin{proof} Let $\ell_v>0$ be as in Remark \ref{rem:n-representation-finite}(b). Let $\Pi=\Pi(\Lambda)$ be the $3$-preprojective algebra of $\Lambda$. By \cite[Proposition 2.12]{GrantIyama2020} we have that the Loewy length of $e_v\Pi$ satisfies
\begin{equation}\label{eq:computation of Loewy length 1} 
    \ell\ell_{\Pi}(e_v\Pi) = \sum_{j=0}^{l_v-1}\ell\ell_{\Lambda}\left(\nu_2^{-j}(e_v\Lambda)\right).
\end{equation}
By Remark \ref{rem:n-representation-finite}(d) we have that $\nu_2^{-j}(e_v\Lambda)=\tau_2^{-j}(e_v\Lambda)$ for every $0\leq j\leq \ell_v-1$. If $0<j<\ell_v-1$, then $\tau_2^{-j}(e_v\Lambda)$ is neither projective nor injective and so $\ell\ell_{\Lambda}(\tau_2^{-j}(e_v\Lambda))=3$ by Proposition \ref{prop:Loewy length and support for n-RF n-levelled}(c). Hence (\ref{eq:computation of Loewy length 1}) becomes
\begin{equation*}
    \ell\ell_{\Pi}(e_v\Pi)=\ell\ell_{\Lambda}(e_v\Lambda) + 3(\ell_v-2)+\ell\ell_{\Lambda}(D\Lambda e_{\sigma(v)}).
\end{equation*}
Using Lemma \ref{lem:Loewy length of projectives light}, we can rewrite the above as
\begin{equation}\label{eq:computation of Loewy length 2}
    \ell\ell_{\Pi}(e_v\Pi)=4+\level(v) + 3(\ell_v-2)-\level(\sigma(v)).
\end{equation}
Moreover, by Proposition \ref{prop:projectives have the same length}(a) we have for any vertex $u\in Q_0$ that 
\begin{equation}\label{eq:equal Loewy length}
    \ell\ell_{\Pi}(e_v\Pi) = \ell\ell_{\Pi}(e_{u}\Pi).
\end{equation}

\begin{enumerate}
\item[(a)] 
For $j\in\{0,1,2\}$, let $Q_0^{j} = \{ x\in Q_0 \mid \level(x)=j\}$ be the set of vertices in $Q$ at level $j$. 
 Replacing (\ref{eq:computation of Loewy length 2}) in (\ref{eq:equal Loewy length}) and viewing the equality modulo $3$ we obtain
\begin{equation}\label{eq:computation of Loewy length 3}
\level(v)-\level(\sigma(v))\equiv \level(u)-\level(\sigma(u)) \mod{3}.
\end{equation}
In particular, we have that $\level(v)=\level(u)$ implies that $\level(\sigma(v))=\level(\sigma(u))$. Hence for any $j\in \{0,1,2\}$ there exists $\sigma(j)\in\{0,1,2\}$ such that $\sigma(Q_0^{j}) \subseteq Q_0^{\sigma(j)}$. Since $Q_0$ is the disjoint union of $Q_0^0$, $Q_0^1$ and $Q_0^2$ and $\sigma:Q_0\to Q_0$ is a bijection, we conclude that $\abs{Q_0^{j}}=\abs{Q_0^{\sigma(j)}}$. Let $x\in Q_0$ be the unique vertex with $\level(x)=i$. Then
\[
\abs{Q_0^{\sigma(i)}} = \abs{Q_0^i} = \abs{\{x\}} = 1.
\]
Since $Q$ is an $2$-levelled $i$-light quiver, we obtain that $\sigma(i)=i$ and so $\sigma(x)=x$. Since (\ref{eq:computation of Loewy length 3}) holds for any two vertices in $Q_0$, for $u=x$ we obtain that 
\[
\level(v)-\level(\sigma(v)) \equiv 0 \mod{3}.
\]
Since $\level(v),\level(\sigma(v))\in\{0,1,2\}$, we conclude that $\level(v)=\level(\sigma(v))$ for any $v\in Q_0$. Using this and Lemma \ref{lem:Loewy length of projectives light} we obtain
\[
\ell\ell_{\Lambda}(e_v\Lambda)+\ell\ell_{\Lambda}(D\Lambda e_{\sigma(v)}) =  \level(v)+1+3-\level(\sigma(v))=4,
\]
which shows part (a).

\item[(b)] By part (a) we have that (\ref{eq:computation of Loewy length 2}) becomes
\[
\ell\ell_{\Pi}(e_v\Pi)=4+3(\ell_v-2),
\]
for every $v\in Q_0$. Replacing this in (\ref{eq:equal Loewy length}) we obtain that $\ell_v=\ell_u$ for any $v,u\in Q_0$. Setting $\ell\coloneqq \ell_v$ we obtain that $\Lambda$ is $\ell$-homogeneous. Finally, $\Lambda$ is twisted $\tfrac{2(\ell-1)}{\ell}$-Calabi--Yau by Theorem \ref{thm:twisted fractionally CY and n-RF}.

\item[(c)] We have that $\Lambda$ is Koszul by Example \ref{ex:2-levelled is Koszul}. Hence the Koszul dual $\Lambda^!$ is Koszul by Proposition \ref{prop:Koszul and n-levelled}(b).

Next let $v\in Q_0$ be a vertex. By part (b) we have that $\Lambda$ is $\ell$-homogeneous. Hence Remark \ref{rem:n-representation-finite}(b) and (d) give that
\begin{equation}\label{eq:first use 2-rep-fin}
\nu_2^{\ell-1}(D \Lambda e_v) \isom \tau_2^{\ell-1}(D\Lambda e_v) \isom e_{\sigma^{-1}(v)} \Lambda.
\end{equation}
On the other hand, by part (b) we have that $\Lambda$ is twisted $\tfrac{2(\ell-1)}{\ell}$-Calabi--Yau. Using this and $\nu_2=\nu\circ[-2]$ we compute
\begin{align}\label{eq:then use twisted CY}
\begin{split}
\nu_2^{\ell-1}(D \Lambda e_v) &= (\nu\circ [-2])^{\ell-1} (D \Lambda e_v) \isom \nu^{-} \nu^{\ell}(D \Lambda e_v)[-2(\ell-1)] \\
&\isom \nu^{-} (D \Lambda e_{\phi(v)})[-2(\ell-1)][2(\ell-1)] \isom e_{\phi(v)}\Lambda.
\end{split}
\end{align}
Comparing (\ref{eq:first use 2-rep-fin}) and (\ref{eq:then use twisted CY}) we see that $e_{\sigma^{-1}(v)}\Lambda\isom e_{\phi(v)}\Lambda$ from which we conclude that the bijections $\sigma: Q_0\to Q_0$ and $\phi:Q_0\to Q_0$ are inverse to each other. By part (a) we conclude that the conditions of part (c) of Proposition \ref{prop:Koszul and n-levelled} are satisfied and so $\Lambda^!$ is twisted fractionally Calabi--Yau of dimension $\tfrac{2(\ell-1)}{\ell}$. 
Since $\gldim(\Lambda^!) = 2$, the claim follows by Theorem \ref{thm:twisted fractionally CY and n-RF}(b). \qedhere
\end{enumerate}
\end{proof}

\section{One-point extensions of bipartite hereditary algebras}
\label{Section:One-point extensions of bipartite hereditary algebras}
Our aim in this section is to derive a formula for the Coxeter polynomial of a one-point extension of a $1$-levelled algebra which satisfies certain conditions. Note that the Coxeter polynomial of a one-point extension algebra has been computed in \cite{Hap09}.

\subsection{Setup}

For the rest of the section we work under the following setup.

\begin{setup}\label{setup:graph, quiver, one-point extension}
We consider a quadruple $(Q,M,\Sigma_1,\Sigma_2)$ where $Q$ is a quiver with at least one arrow, $M\in\modu\K Q$ is a nonzero module and $\Sigma_1,\Sigma_2\in\RR$ such that:
\begin{enumerate}[(i)]
    \item The underlying graph $\graph{Q}$ of $Q$ is a bipartite loopless multigraph with coloring $X\cup Y$ and $Q$ has bipartite orientation as in Example \ref{ex:bipartite quiver is 1-levelled}. In particular $Q$ is a $1$-levelled quiver. Moreover, we denote the vertices in $X$ by $\{x_1,\dots,x_r\}$ and the vertices in $Y$ by $\{y_1,\dots,y_s\}$ and we order them by 
    \[
    x_1 < x_2 < \cdots < x_r < y_1 < \cdots <y_{s-1} < y_s.
    \]
    \item $\topu(M)$ is supported only in level $1$. We set $\Gamma = \K Q[M]$ and $d_M=\dimv(M)^T$. Furthermore, we set 
    \[
    d_x = (\dim_{\K}(Me_{x_1}),\dots,\dim_{\K}(Me_{x_r}))^T \text{ and } d_y = (\dim_{\K}(Me_{y_1}),\dots,\dim_{\K}(Me_{y_s}))^T,
    \]
    where $e_{x_1},\dots,e_{x_r},e_{y_1},\dots,e_{y_s}$ are the idempotents of $\K Q$ corresponding to the vertices $x_1,\dots,x_r,y_1,\dots,y_s$.
    \item $(d_x,d_y)$ is a bi-eigenvector of $A(\graph{Q})$ with bi-eigenvalue $(\Sigma_1,\Sigma_2)$.
\end{enumerate}
\end{setup}

By part (i) in Setup \ref{setup:graph, quiver, one-point extension} we have that the adjacency matrix $A(\graph{Q})$ has the form
\[
A(\graph{Q}) = 
\begin{bmatrix} 
0_r & R \\
R^T & 0_s
\end{bmatrix}
\]
where $R$ is an $r\times s$-matrix with nonnegative integer entries. 
By part (ii) in Setup \ref{setup:graph, quiver, one-point extension} we have $
d_M= \begin{bmatrix}
d_x \\
d_y
\end{bmatrix}$. We collect some properties that we have in this setting.

\begin{prop}\label{prop:main properties of setup}
\begin{enumerate}[(a)]
    \item The Cartan matrix of $\K Q$ respectively $\Gamma$ is
    \[
    C_{\K Q} =
    \begin{bmatrix}
    I_r & R \\
    0_s & I_s
    \end{bmatrix}
    \text{ respectively }
    C_{\Gamma} = 
    \begin{bmatrix}
    1 & d_x^T & d_y^T \\ 
    0_r & I_r & R \\
    0_s & 0_s & I_s
    \end{bmatrix}.
    \]
    The Coxeter polynomial of $\K Q$ respectively $\Gamma$ is
    \[
    \CP_{\K Q}(t) = \det(tC_{\K Q}+C_{\K Q}^T)
    \text{ respectively }
    \CP_{\Gamma}(t) = \det(tC_{\Gamma}+C_{\Gamma}^T).
    \]
    \item The algebra $\Gamma$ is a $2$-levelled $0$-light Koszul algebra.
    \item $\Sigma_1,\Sigma_2\in \QQ$, $\Sigma_1\Sigma_2\in \ZZ$ and $\Sigma_1\Sigma_2>0$. 
    \item $d_x\neq 0$ and $d_y\neq 0$. In particular, $d_x$ is an eigenvector of $RR^T$ with eigenvalue $\Sigma_1\Sigma_2$ and $d_y$ is an eigenvector of $R^TR$ with eigenvalue $\Sigma_1\Sigma_2$.
    \item $\Sigma_1\lvert d_x\rvert^2 = \Sigma_2 \lvert d_y \rvert^2$.
    \item The largest eigenvalue of $A(\graph{Q})$ is $\sqrt{\Sigma_1\Sigma_2}$.
\end{enumerate}
\end{prop}

\begin{proof}
That the Cartan matrix of $\K Q$ has the form 
\begin{equation}\label{eq:cartan matrix of kQ}
C_{\K Q} = 
\begin{bmatrix}
I_r & R \\ 
0_s & I_s
\end{bmatrix}
\end{equation}
follows immediately by part (i) of Setup \ref{setup:graph, quiver, one-point extension}. Using (\ref{eq:cartan matrix of one-point extension}) and (\ref{eq:cartan matrix of kQ}), we compute the Cartan matrix of $\Gamma$ to be
\begin{equation}\label{eq:cartan matrix of one-point extension of bipartite quiver}
C_{\Gamma} = 
\begin{bmatrix}
1 & d_x^T & d_y^T \\ 
0_r & I_r & R \\
0_s & 0_s & I_s
\end{bmatrix}.
\end{equation}
Note that in particular we have $\det(C_{\K Q})=1$ and $\det(C_{\Gamma})=1$. Then we have
\begin{align*}
    \CP_{\Gamma}(t) &= \det(tI -(-C_{\Gamma}^T C_{\Gamma}^{-1})) \\
    &= \det(tC_{\Gamma}+C_{\Gamma}^T)\det(C_{\Gamma}^{-1}) \\
    &=\det(tC_{\Gamma}+C_{\Gamma}^T),
\end{align*}
and similarly for $\CP_{\K Q}(t)$. This shows part (a). That $\Gamma$ is $2$-levelled $0$-light follows from Lemma \ref{lem:quiver of one-point extension} and using the assumption that $\topu(M)$ is supported only in level $1$. Then $\Gamma$ is Koszul by Example \ref{ex:2-levelled is Koszul}. This shows part (b). Parts (c), (d), (e) and (f) follow by Lemma \ref{lem:bi-eigenvectors and nonnegative integer matrices}.
\end{proof}

\subsection{The Coxeter polynomial of \texorpdfstring{$\Gamma$}{G}}

In this section we continue working under Setup \ref{setup:graph, quiver, one-point extension} and we compute the Coxeter polynomial $\CP_{\Gamma}(t)$ of $\Gamma$. We need the following standard result about determinants of block matrices for our computations.

\begin{prop}\cite[Propositions 2.8.3 and 2.8.4]{MatrixMath}\label{prop:determinant of block matrix}
Consider the block matrix
$\begin{bmatrix}
A & B\\
C & D
\end{bmatrix}$.
\begin{enumerate}[(a)]
    \item If $A$ is invertible, then 
$
\det\begin{bmatrix}
A & B\\
C & D
\end{bmatrix} 
= 
\det(A) \det(D - CA^{-1}B)$.
\item If $D$ is invertible, then
$
\det\begin{bmatrix}
A & B\\
C & D
\end{bmatrix} 
= 
\det(D) \det(A - CD^{-1}B)$.
\end{enumerate}
\end{prop}

In the sequel, we set
\begin{equation}\label{eq:the w polynomial}
    w(t) \coloneqq (t+1)^2-t\Sigma_1\Sigma_2 = t^2 - (\Sigma_1\Sigma_2-2)t+1
\end{equation}

We have the following trivial observation.

\begin{lem}\label{lem:roots of w(t)}
Let 
\[
f_{\pm}(t) = \frac{(t^2-2)\pm t\sqrt{t^2-4}}{2}.
\]
Then the roots of $w(t)$ are $f_{+}(\sqrt{\Sigma_1\Sigma_2})$ and $f_{-}(\sqrt{\Sigma_1\Sigma_2})$. In particular, $-1$ is not a root of $w(t)$.
\end{lem}

\begin{proof}
That these are the roots follows by a direct computation. Also, evaluating $w(t)$ at $-1$ we obtain $w(-1)=\Sigma_1\Sigma_2$ which is nonzero by Proposition \ref{prop:main properties of setup}(c).
\end{proof}

\begin{lem}\label{lem:w(t) is a factor of characteristic polynomial of kQ}
The polynomial $w(t)$ is a factor of $\CP_{\K Q}(t)$.
\end{lem}

\begin{proof}
Using Proposition \ref{prop:main properties of setup}(a) we have
\begin{align*}
    \CP_{\K Q}(t) =\det(tC_{\K Q}+C_{\K Q}^T) 
    =\det\begin{bmatrix}
    (t+1)I_r & tR \\
    R^T & (t+1)I_s
    \end{bmatrix}.
\end{align*}
For $t\neq -1$ and using Proposition \ref{prop:determinant of block matrix} we obtain
\begin{align*}
    \CP_{\K Q}(t) 
    & = \det((t+1)I_r) \det((t+1)I_s - R^T((t+1)I_r)^{-1} tR)\\
    & = (t+1)^r \det((t+1)I_s - (t+1)^{-1}tR^T R)\\
    & = (t+1)^r \det((t+1)^{-1}((t+1)^2I_s-tR^T R))\\
    & = (t+1)^{r-s} \det((t+1)^2 I_s - tR^T R)
\end{align*}
Now notice that $R^T R d_y = \Sigma_1 \Sigma_2 d_y$ implies
\[
((t+1)^2I_s - tR^TR)d_y = w(t)  d_y.
\]
By Proposition \ref{prop:main properties of setup}(c) we have that $d_y\neq 0$. Hence if $\lambda$ is a root of $w(t)$, we obtain that the matrix $(\lambda+1)^2 I_s-\lambda R^T R$ has $w(\lambda)=0$ as an eigenvalue, and so
\[
\det((\lambda+1)^2 I_s - \lambda R^TR)=0.
\]
It follows that all roots of $w(t)$ are roots of $\det((t+1)^2 I_s - t R^TR)$. Since by Lemma \ref{lem:roots of w(t)} we have that $-1$ is not a root of $w(t)$, and since $w(t)$ is a monic polynomial we conclude that $w(t)$ is a factor of $\CP_{\K Q}(t)$.
\end{proof}

\begin{lem}\label{lem:middle computation}
If $z$ is not a root of $\CP_{\K Q}(t)$, then we have
\begin{align*}
    w(z)
    \begin{bmatrix}
    d_x^T & d_y^T
    \end{bmatrix}
    \begin{bmatrix}
    (z + 1)I_r & zR\\
     R^T & (z+1)I_s
    \end{bmatrix}^{-1}
    \begin{bmatrix}
    d_x\\
    d_y
    \end{bmatrix}
    & = 
    \lvert d_x \rvert^2 (z+1 - z\Sigma_1) + \lvert d_y \rvert^2 (z+1 - \Sigma_2).
\end{align*}
\end{lem}

\begin{proof}
If $z$ is not a root of $\CP_{\K Q}(t)$, then 
\[
0\neq \CP_{\K Q}(z) = \det\begin{bmatrix}
    (z+1)I_r & zR \\
    R^T & (z+1)I_s
    \end{bmatrix}
\]
and so the matrix
\[
\begin{bmatrix}
    (z + 1)I_r & zR\\
     R^T & (z+1)I_s
    \end{bmatrix}
\]
is invertible. First, using $R d_y=\Sigma_1 d_x$ and $R^{T}d_x=\Sigma_2 d_y$, we compute 
\begin{align}\label{eq:middle computation first step}
    \begin{bmatrix}
    (z + 1)I_r & zR\\
     R^T & (z+1)I_s
    \end{bmatrix}
    \begin{bmatrix}
    (z+1 - z\Sigma_1) d_x\\
    (z+1 - \Sigma_2) d_y
    \end{bmatrix}
    &=
    w(z)
    \begin{bmatrix}
    d_x \\
    d_y
    \end{bmatrix}.
\end{align}
By multiplying both sides of (\ref{eq:middle computation first step}) by 
\[
\begin{bmatrix}
    d_x^T & d_y^T
    \end{bmatrix}
\begin{bmatrix}
    (z + 1)I_r & zR\\
     R^T & (z+1)I_s
    \end{bmatrix}^{-1}
\]
we obtain
\begin{align*}
    w(z) \begin{bmatrix}
    d_x^T & d_y^T
    \end{bmatrix}
    \begin{bmatrix}
    (z + 1)I_r & zR\\
     R^T & (z+1)I_s
    \end{bmatrix}^{-1}
    \begin{bmatrix}
    d_x \\
    d_y
    \end{bmatrix}
    &=
    \begin{bmatrix}
    d_x^T & d_y^T
    \end{bmatrix}
    \begin{bmatrix}
    (z+1 - z\Sigma_1) d_x\\
    (z+1 - \Sigma_2) d_y
    \end{bmatrix}\\
    &=\lvert d_x \rvert^2 (z+1 - z\Sigma_1) + \lvert d_y \rvert^2 (z+1 - \Sigma_2),
\end{align*}
as claimed.
\end{proof}

\begin{lem}\label{lem:the characteristic polynomial of Gamma}
We have
\[
\CP_{\Gamma}(t) = \frac{\CP_{\K Q}(t)}{w(t)}p(t),
\]
where 
\[
p(t) = (t + 1)(t^2 - (\Sigma_1\Sigma_2 + \lvert d_x \rvert^2 + \lvert d_y \rvert^2 - \lvert d_x \rvert^2\Sigma_1 - 2)t + 1).
\]
\end{lem}

\begin{proof}
Using Proposition \ref{prop:main properties of setup}(a), we have
\begin{align*}
    \CP_{\Gamma}(t) &= \det (tC_{\Lambda} + C_{\Lambda}^T) = 
    \det
    \begin{bmatrix}
    (t + 1) & t d_x^T & t d_y^T\\
    d_x & (t + 1)I_r & tR\\
    d_y & R^T & (t+1)I_s
    \end{bmatrix}.
\end{align*}
Assume that $z$ is not a root of $\CP_{\K Q}(t)$. Then the matrix
\[
\begin{bmatrix}
    (z + 1)I_r & zR\\
     R^T & (z+1)I_s
    \end{bmatrix}
\]
is invertible and so we may apply Proposition \ref{prop:determinant of block matrix}(b) to continue our computation
\begin{align*}
    \CP_{\Gamma}(z) & = \det\begin{bmatrix}
    (z + 1) & z d_x^T & z d_y^T\\
    d_x & (z + 1)I_r & zR\\
    d_y & R^T & (z+1)I_s
    \end{bmatrix}\\
    &=
    \det
    \begin{bmatrix}
    (z + 1)I_r & zR\\
    R^T & (z+1)I_s
    \end{bmatrix}
    \left((z+1) - 
    z
    \begin{bmatrix}
    d_x \\ d_y
    \end{bmatrix}
    \begin{bmatrix}
    (z + 1)I_r & zR\\
     R^T & (z+1)I_s
    \end{bmatrix}^{-1}
    \begin{bmatrix}
    d_x^T &
    d_y^T
    \end{bmatrix}\right)\\
    & = 
    \CP_{\K Q}(z)
    \left((z+1) - 
    z
    \begin{bmatrix}
    d_x^T & d_y^T
    \end{bmatrix}
    \begin{bmatrix}
    (z + 1)I_r & zR\\
     R^T & (z+1)I_s
    \end{bmatrix}^{-1}
    \begin{bmatrix}
    d_x\\
    d_y
    \end{bmatrix}\right).\\
\end{align*}
Assume that $z$ is not a root of $w(t)$ either. Using Lemma \ref{lem:middle computation}, we obtain
\begin{align*}
    w(z)\CP_{\Gamma}(z) &= \CP_{\K Q}(z)\left((z+1)w(z) - 
    zw(z)
    \begin{bmatrix}
    d_x^T & d_y^T
    \end{bmatrix}
    \begin{bmatrix}
    (z + 1)I_r & zR\\
     R^T & (z+1)I_s
    \end{bmatrix}^{-1}
    \begin{bmatrix}
    d_x\\
    d_y
    \end{bmatrix}\right) \\
    &= \CP_{\K Q}(z)\left((z+1)w(z) - 
    z\left(\lvert d_x \rvert^2 (z+1 - z\Sigma_1) + \lvert d_y \rvert^2 (z+1 - \Sigma_2)\right)\right).
\end{align*}
Set
\[
r(t) = (t+1)w(t) - 
    t\left(\lvert d_x \rvert^2 (t+1 - t\Sigma_1) + \lvert d_y \rvert^2 (t+1 - \Sigma_2)\right).
\]
We claim that $r(t)=p(t)$. Indeed, by expanding $r(t)$ we obtain
\[
r(t) = t^3 - (\Sigma_1\Sigma_2 + \lvert d_x \rvert^2 + \lvert d_y \rvert^2 - \Sigma_1\lvert d_x \rvert^2 - 3)t^2 - (\Sigma_1\Sigma_2 + \lvert d_x \rvert^2 + \lvert d_y \rvert^2 - \Sigma_2\lvert d_y \rvert^2 - 3)t + 1.
\]
Using $\Sigma_1\lvert d_x \rvert^2=\Sigma_2\lvert d_y \rvert^2$ which holds by Proposition \ref{prop:main properties of setup}(e), we obtain
\[
r(t) = (t + 1)(t^2 - (\Sigma_1\Sigma_2 + \lvert d_x \rvert^2 + \lvert d_y \rvert^2 - \lvert d_x \rvert^2\Sigma_1 - 2)t + 1) = p(t),
\]
as claimed. Hence we have shown that for all but finitely many $z\in\CC$ (namely the roots of $w(t)$ and $\CP_{\K Q}(t)$) we have 
\[
w(z)\CP_{\Gamma}(z) = \CP_{\K Q}(z) p(z).
\]
Since $w(t)$, $\CP_{\Gamma}(t)$ and $p(t)$ are all polynomials, we obtain an equality of polynomials
\[
w(t)\CP_{\Gamma}(t)=\CP_{\K Q}(t)p(t).
\]
Since by Lemma \ref{lem:w(t) is a factor of characteristic polynomial of kQ} we have that $w(t)$ is a factor of $\CP_{\K Q}(t)$, the claim follows by dividing both sides by $w(t)$.
\end{proof}

We end the section with the following result, giving a correspondence between the eigenvalues of the Coxeter matrix of $\K Q$ and the eigenvalues of the adjacency matrix of the associated bipartite graph $A = A(\graph{Q})$.

\begin{prop}\label{prop: eigenvalue correspondence}
Let $\CP_A(t)$ be the characteristic polynomial of $A=A(\graph{Q})$ and let $\CP_{\K Q}(t)$ be the Coxeter polynomial of $\K Q$. Then the roots of $\CP_A(t)$ come in pairs of the form $\{\lambda,\lambda^{-1}\}$, the roots of $\CP_{\K Q}(t)$ come in pairs of the form $\{z,z^{-1}\}$, and furthermore there is a correspondence between these pairs of roots. The correspondence is given by
\[
\{\lambda,-\lambda\} \mapsto \{z=f_{+}(\lambda), z^{-1}=f_{-}(\lambda)\},
\]
and
\[
\{z, z^{-1}\} \mapsto \{\lambda = \sqrt{z} + \frac{1}{\sqrt{z}},-\lambda\}, 
\]
where $f_{\pm}(t)$ is as in Lemma \ref{lem:roots of w(t)}. 
\end{prop}

This can be found stated in the beginning of Section 3 of \cite{McKS05}. 
For the reader's convenience, we have included a proof.

\begin{proof}
First of all it is straightforward to compute that the maps defined in the statement are inverse to each other. Also, since $\overline{Q}$ is a bipartite graph, the roots of $\CP_{A}(t)$ are symmetric around $0$.

Now, using Proposition \ref{prop:determinant of block matrix}(a) we compute
\begin{align*}
    \CP_{A}(\lambda) &= \det 
\begin{bmatrix}
\lambda I_r & -R \\
-R^T & \lambda I_s
\end{bmatrix} \\
&= \det(\lambda I_r)\det(\lambda I_s-(-R^T)(\lambda I_r)^{-1}(-R)) \\
&= \lambda^{r}\det(\lambda I_s - \lambda^{-1}R^T R).
\end{align*}

Making the substitution $\lambda = \sqrt{z} + \frac{1}{\sqrt{z}}$, we obtain
\begin{equation}\label{eq:q(sqrt(t)+1/sqrt(t)}
    \CP_{A}\left(\sqrt{z}+\frac{1}{\sqrt{z}}\right)= (z+1)^r z^{-\frac{r}{2}}\det((z+1)z^{-\frac{1}{2}}I_s - (z+1)^{-1}z^{\frac{1}{2}}R^TR).
\end{equation}
On the other hand, using Proposition \ref{prop:main properties of setup}(a) we have
\begin{align*}
    \CP_{\K Q}(z) =\det(zC_{\K Q}+C_{\K Q}^T) =\det\begin{bmatrix}
    (z+1)I_r & zR \\
    R^T & (z+1)I_s
    \end{bmatrix}.
\end{align*}
Using Proposition \ref{prop:determinant of block matrix}(a) we obtain
\begin{align*}
    \CP_{\K Q}(z) 
    & = \det((z+1)I_r) \det((z+1)I_s - R^T((z+1)I_r)^{-1} zR)\\
    & = (z+1)^r \det((z+1)I_s - (z+1)^{-1}zR^T R)\\
    & = (z+1)^r \det(z^{\frac{1}{2}}((z+1)z^{-\frac{1}{2}}I_s-(z+1)^{-1}z^{\frac{1}{2}}R^T R))\\
    & = (z+1)^{r} z^{\frac{s}{2}}\det((z+1)z^{-\frac{1}{2}}I_s - (z+1)^{-1}z^{\frac{1}{2}}R^T R)\\
    & \overset{(\ref{eq:q(sqrt(t)+1/sqrt(t)})}{=} (z+1)^{r} z^{\frac{s}{2}}\frac{\CP_{A}\left(\sqrt{z}+\frac{1}{\sqrt{z}}\right)}{(z+1)^r z^{-\frac{r}{2}}}  \\
    &=z^{\frac{r+s}{2}}\CP_{A}\left(\sqrt{z}+\frac{1}{\sqrt{z}}\right).
\end{align*}
Hence, we see that roots of $\CP_{\K Q}(t)$ indeed come in pairs $\{z,z^{-1}\}$ and that every pair of roots $\{z, z^{-1}\}$ of $\CP_{\K Q}(t)$ corresponds to a pair of roots $\{\lambda, -\lambda\}$ of $\CP_A(t)$ as claimed. 
\end{proof}
\subsection{One-point extensions and higher homological algebra}

In this section we continue working under Setup \ref{setup:graph, quiver, one-point extension}. Our aim is to use the Coxeter polynomial of $\Gamma$ to obtain necessary conditions that have to be satisfied if $\Gamma$ has certain properties coming from higher homological algebra. 
First we obtain a condition on the graph $\graph{Q}$.

\begin{prop}\label{prop:if Gamma is twisted fractionally Calabi-Yau then graph is reflexive}
Assume that $\Gamma$ is twisted fractionally Calabi--Yau. Then $\graph{Q}$ is a reflexive graph. Furthermore, $\graph{Q}$ is a Salem graph if and only if $\sqrt{\Sigma_1\Sigma_2}> 2$.
\end{prop}

\begin{proof}
Let
\begin{equation}\label{eq:characteristic polynomial of adjacency matrix}
\CP_A(t) = \det 
\begin{bmatrix}
tI_r & -R \\
-R^T & tI_s
\end{bmatrix}
\end{equation}
be the characteristic polynomial of $A=A(\graph{Q})$. Recall that by Proposition \ref{prop:main properties of setup}(f) we have that $\sqrt{\Sigma_1\Sigma_2}$ is the largest eigenvalue of $A$. Hence it is enough to show that all other eigenvalues belong to $[-2,2]$. Since eigenvalues of $A$ are symmetric around $0$, it is enough to show that if $\lambda>0$ is a root of $\CP_A(t)$ with $\lambda\neq \sqrt{\Sigma_1\Sigma_2}$, then $\lambda\leq 2$.

Assume towards a contradiction that there exists a root $\lambda$ of $\CP_A(t)$ with $\lambda > 2$ and $\lambda\neq\sqrt{\Sigma_1\Sigma_2}$, and recall the correspondence between the roots of $\CP_A(t)$ and $\CP_{\K Q}(t)$ described in Proposition \ref{prop: eigenvalue correspondence}.
In particular, $z=f_{+}(\lambda)$ where $f_{+}(t)$ is as in Lemma \ref{lem:roots of w(t)}, is a root of $\CP_{\K Q}(t)$. Since $\lambda\neq\sqrt{\Sigma_1\Sigma_2}$, Lemma \ref{lem:roots of w(t)} implies that $z$ is not a root of $w(t)$. We conclude by Lemma \ref{lem:the characteristic polynomial of Gamma} that $z$ is a root of $\CP_{\Gamma}(t)$. Since $\Gamma$ is twisted fractionally Calabi--Yau, we conclude that $\abs{z}=1$ by Proposition \ref{prop:fractionally Calabi--Yau implies eigenvalues of Coxeter matrix are in unit circle}. But then
\[
1 = \abs{z} = \abs{f_{+}(\lambda)}= \abs*{\frac{\left(\lambda^2-2\right)+\lambda\sqrt{\lambda^2-4}}{2}} > \frac{(4-2)+2\sqrt{4-4}}{2}=1, 
\]
is a contradiction.
\end{proof}

Next we obtain a condition on the data $(\lvert d_x \rvert, \lvert d_y \rvert, \Sigma_1, \Sigma_2)$. We start with the following easy lemma.

\begin{lem}\label{lem:monic polynomial of degree 2 with roots on unit circle}
Let $s(t)=t^2+at+b$ be a polynomial such that $a,b\in\QQ$ and all roots of $s(t)$ are on the unit circle. Then $s(t)\in\{t^2+1,t^2\pm t+1, t^2\pm 2t+1\}$.
\end{lem}

\begin{proof}
Let $z$ and $\overline{z}$ be the two roots of $s(t)$ in $\CC$. Since by assumption we have $\lvert z \rvert=1$ it follows that $z=\cos\theta + i \sin\theta$ where $0\leq \theta < 2\pi$. Then
\[
t^2 + at + b = s(t) = (t-z)(t-\overline{z}) = t^2-(z+\overline{z})t+z\overline{z} = t^2-2\cos\theta t+1.
\]
This shows that $2\cos\theta\in\QQ$ and that $b=1$. Hence $\cos\theta\in\QQ$ and so, by Niven's theorem \cite{Niven56}, we obtain that $\cos\theta\in\{0,\pm\tfrac{1}{2},\pm 1\}$. The claim follows.
\end{proof}

\begin{prop}\label{prop:properties when one-point extension is fractionally Calabi-Yau}
Assume that $\Gamma$ is twisted fractionally Calabi--Yau. Then the following hold.
\begin{enumerate}[(a)]
    \item $\Sigma_1,\Sigma_2\in \QQ$, $\Sigma_1\Sigma_2\in \ZZ$ and $\Sigma_1\Sigma_2>0$. 
    \item $d_x\neq 0$ and $d_y\neq 0$. In particular, $d_x$ is an eigenvector of $RR^T$ with eigenvalue $\Sigma_1\Sigma_2$ and $d_y$ is an eigenvector of $R^TR$ with eigenvalue $\Sigma_1\Sigma_2$.
    \item $\Sigma_1\lvert d_x\rvert^2 = \Sigma_2 \lvert d_y \rvert^2$.
    \item $\Sigma_1\Sigma_2 + \lvert d_x \rvert^2 + \lvert d_y \rvert^2 -\lvert d_x \rvert^2 \Sigma_1\in\{0,1,2,3,4\}$.
\end{enumerate}
\end{prop}

\begin{proof}
Parts (a), (b) and (c) are just parts (c), (d) and (e) of Proposition \ref{prop:main properties of setup}. It remains to show part (d). Notice that by Lemma \ref{lem:the characteristic polynomial of Gamma} we have
\begin{equation}\label{eq:char pol lambda}
w(t)\CP_{\Gamma}(t) =\CP_{\K Q}(t)(t+1)q(t),
\end{equation}
where
\[
q(t)\coloneqq t^2 - (\Sigma_1\Sigma_2 + \lvert d_x \rvert^2 + \lvert d_y \rvert^2 - \lvert d_x \rvert^2\Sigma_1 - 2)t + 1.
\]
Since by Lemma \ref{lem:w(t) is a factor of characteristic polynomial of kQ} we have that $w(t)$ divides $\CP_{\K Q}(t)$, we conclude that the roots of $q(t)$ are roots of $\CP_{\Gamma}(t)$. Since $\Gamma$ is twisted fractionally Calabi--Yau, it follows that the roots of $\CP_{\Gamma}(t)$ lie on the unit circle from Proposition \ref{prop:fractionally Calabi--Yau implies eigenvalues of Coxeter matrix are in unit circle}. Hence the roots of $q(t)$ lie on the unit circle as well. By part (a) we have that $q(t)$ is a rational polynomial. Then by Lemma \ref{lem:monic polynomial of degree 2 with roots on unit circle} we obtain that
\[
q(t)\in\{t^2+1,t^2\pm t+1, t^2\pm 2t+1\}.
\]
and the claim follows.
\end{proof}

\begin{cor}\label{cor:properties when one-point extension is 2-RF} 
Assume that $\Gamma$ is $2$-representation-finite. Then the statements (a)--(d) of Proposition \ref{prop:properties when one-point extension is fractionally Calabi-Yau} hold.
\end{cor}

\begin{proof}
By Proposition \ref{prop:main properties of setup}(b) we have that $\Gamma$ is a $2$-levelled $0$-light algebra. Since $\Gamma$ is also $2$-representation-finite, by Proposition \ref{prop:2-levelled light and 2-RF is twisted fractionally Calabi--Yau}(b) we have that $\Gamma$ is twisted fractionally Calabi--Yau, which is the requirement of Proposition \ref{prop:properties when one-point extension is fractionally Calabi-Yau}.
\end{proof}

\subsection{The order of the negative of the Coxeter matrix}

We finish this section with a small computation that will be useful later. As before we work under Setup \ref{setup:graph, quiver, one-point extension}. Then the Cartan matrix of $\Gamma$ is of the form
\[
C_{\Gamma} = 
\begin{bmatrix}
1 & d_x^T & d_y^T \\ 
0_{r,1} & I_r & R \\
0_{s,1} & 0_{s,r} & I_s
\end{bmatrix},
\]
where $R$ is the adjacency matrix of the underlying graph of the quiver $Q$ and where $(d_x^T, d_y^T)$ is the dimension vector of $M$. In the sequel, whenever we have an identity matrix $I_k$, we simply write $I$; the dimension $k$ is always clear from context. Our aim is to compute the order of the matrix $-\Phi_{\Gamma}$ under certain assumptions. 

Let us first set $C_+ := C_{\Gamma} - I$. Since $R$ is bipartite, we have that $C_+^3 = (C_{\Gamma} - I)^3 = 0$, and we thus have that 
\[
C_{\Gamma}^{-1} = C_+^{2} - C_{+} + I = \begin{bmatrix}
    0 & 0_{1,r} & d_x R^T \\ 0_{r,1} & 0_{r,r} & 0_{r,s} \\ 0_{s,1} & 0_{s,r} & 0_{s,s}   
\end{bmatrix} -
    \begin{bmatrix}
0 & d_x^T & d_y^T \\ 
0_{r,1} & 0_{r,r} & R \\
0_{s,1} & 0_{s,r} & 0_{s,s}
\end{bmatrix} + I.
\]
Recall that the Coxeter matrix of $\Gamma$ is given by $\Phi_{\Gamma} = -C_{\Gamma}^{T}C_{\Gamma}^{-1}$.
Hence, using that $R d_y=\Sigma_1 d_x$ and $R^{T} d_y=\Sigma_2 d_y$, we can compute 
\begin{align*}
C_{\Gamma}^{T}C_{\Gamma}^{-1}
 &= C_{\Gamma}^{T} C_{+}^2 -C_{\Gamma}^{T} C_{+} + C_{\Gamma} \\
 &=
\begin{bmatrix}
    1 & -d_x^T & (\Sigma_2 - 1)d_y^T \\ 
    d_x & -d_x d_x^T + I & (\Sigma_2 - 1)d_x d_y^T - R \\
    d_y & -d_y d_x^T + R^T & (\Sigma_2 - 1)d_yd_y^T - R^T R + I
    \end{bmatrix}. 
\end{align*}

From this, we get 
\[
C_{\Gamma}^{T}C_{\Gamma}^{-1} 
\begin{bmatrix}
a \\
bd_x\\
cd_y
\end{bmatrix}
= 
\begin{bmatrix}
a - b\lvert d_x \rvert^2 + c(\Sigma_2 -1)\lvert d_y\rvert^2 \\
d_x(a + b(1-\lvert d_x\rvert^2) + c((\Sigma_2 -1)\lvert d_y \rvert^2-\Sigma_1)\\
d_y(a + b(\Sigma_2 - \lvert d_x \rvert^2) + c((\Sigma_2 - 1)\lvert d_y \rvert^2 - \Sigma_1 \Sigma_2 + 1))
\end{bmatrix}.
\]
Let now 
\[
\begin{bmatrix}
f_1 \\
f_2 \\
f_3
\end{bmatrix}
:=
\begin{bmatrix}
(\Sigma_2 -1)\lvert d_y\rvert^2 \\
(\Sigma_2 -1)\lvert d_y \rvert^2 - \Sigma_1 \\
(\Sigma_2 - 1)\lvert d_y \rvert^2 - \Sigma_1 \Sigma_2 + 1
\end{bmatrix},
\]
so that
\begin{equation}\label{eq:order of -Phi_Gamma 1}
(C_{\Gamma}^{T}C_{\Gamma}^{-1})^{m}
\begin{bmatrix}
1 \\
d_x\\
d_y
\end{bmatrix}=
\begin{bmatrix}
1 \\
d_x\\
d_y
\end{bmatrix}
\end{equation}
holds if and only if 
\begin{equation}\label{eq:order of -Phi_Gamma 2}
\begin{bmatrix}
1 & -\lvert d_x \rvert^2 & f_1 \\
1 & 1 - \lvert d_x \rvert^2 & f_2 \\
1 & \Sigma_2 - \lvert d_x \rvert^2 & f_3 \\
\end{bmatrix}^{m}
\begin{bmatrix}
1 \\
1\\
1
\end{bmatrix}=
\begin{bmatrix}
1 \\
1\\
1
\end{bmatrix}
\end{equation}
holds. With this in mind we have the following.

\begin{prop}\label{prop: p and the order of the Coxeter matrix when bi-eigenvector} 
Assume that $\Sigma_1 \lvert d_x \rvert^2=\Sigma_2 \lvert d_y\rvert^2$ and that $\Sigma_1\Sigma_2 + \lvert d_x \rvert^2 + \lvert d_y \rvert^2 -\lvert d_x \rvert^2 \Sigma_1=p$ for some $p\in\{0,1,2,3,4\}$. Then the following hold:
\begin{enumerate}[(a)]
    \item If $p \neq 4$, then the negative of the Coxeter matrix $\Phi_{\Gamma}$ is of order described in the following table:
    \[
    \begin{tabular}{c|c}
         $p$ & Order of $-\Phi_\Gamma$  \\
         $0$ & $\infty$\\
         $1$ & $6$ \\
         $2$ & $4$ \\
         $3$ & $3$ \\
     \end{tabular}
     \]
    \item If $p = 4$, then $(r, \Sigma_1, s, \Sigma_2) = (4,2,4,2)$.
\end{enumerate}
\end{prop}

\begin{proof}
Using the equivalence of (\ref{eq:order of -Phi_Gamma 1}) and (\ref{eq:order of -Phi_Gamma 2}), it is enough to show that
\[
\begin{bmatrix}
1 & -\lvert d_x \rvert^2 & f_1 \\
1 & 1 - \lvert d_x \rvert^2 & f_2 \\
1 & \Sigma_2 - \lvert d_x \rvert^2 & f_3 \\
\end{bmatrix}^{m}
\begin{bmatrix}
1 \\
1\\
1
\end{bmatrix} = \begin{bmatrix}
1 \\
1\\
1
\end{bmatrix}
\]
for $m > 0$ never happens if $p = 0$, happens first for $m = 6$ for $p = 1$, for $m = 4$ for $p = 2$, for $m = 3$ for $p = 3$, and only ever happens if $(\lvert d_x\rvert^2, \Sigma_1, \lvert d_y\rvert^2, \Sigma_2) = (4,2,4,2)$ for $p = 4$, in which case it happens first for $m = 4$. This can be showed using $\Sigma_1 \lvert d_x \rvert^2=\Sigma_2 \lvert d_y\rvert^2$ and that $\Sigma_1\Sigma_2 + \lvert d_x \rvert^2 + \lvert d_y \rvert^2 -\lvert d_x \rvert^2 \Sigma_1=p$ and by noting that
\[
\begin{bmatrix}
1 & -\lvert d_x \rvert^2 & f_1 \\
1 & 1 - \lvert d_x \rvert^2 & f_2 \\
1 & \Sigma_2 - \lvert d_x \rvert^2 & f_3 \\
\end{bmatrix}
\begin{bmatrix}
0 \\
1\\
1
\end{bmatrix}
=
\begin{bmatrix}
0 \\
1\\
1
\end{bmatrix}
+
\begin{bmatrix}
\Sigma_1\Sigma_2\\
(\Sigma_2 - 1)\Sigma_1\\
\Sigma_2
\end{bmatrix}
-
\begin{bmatrix}
p \\
p\\
p
\end{bmatrix},
\]
that
\[
\begin{bmatrix}
1 & -\lvert d_x \rvert^2 & f_1 \\
1 & 1 - \lvert d_x \rvert^2 & f_2 \\
1 & \Sigma_2 - \lvert d_x \rvert^2 & f_3 \\
\end{bmatrix}
\begin{bmatrix}
1 \\
1\\
1
\end{bmatrix}
=
\begin{bmatrix}
0 \\
1\\
1
\end{bmatrix}
+
\begin{bmatrix}
1 \\
1\\
1
\end{bmatrix}
+
\begin{bmatrix}
\Sigma_1\Sigma_2\\
(\Sigma_2 - 1)\Sigma_1\\
\Sigma_2
\end{bmatrix}
-
\begin{bmatrix}
p \\
p\\
p
\end{bmatrix},
\]
and that
\[
\begin{bmatrix}
1 & -\lvert d_x \rvert^2 & f_1 \\
1 & 1 - \lvert d_x \rvert^2 & f_2 \\
1 & \Sigma_2 - \lvert d_x \rvert^2 & f_3 \\
\end{bmatrix}
\begin{bmatrix}
\Sigma_1\Sigma_2\\
(\Sigma_2 - 1)\Sigma_1\\
\Sigma_2
\end{bmatrix} 
=
\begin{bmatrix}
\Sigma_1\Sigma_2\\
(\Sigma_2 - 1)\Sigma_1\\
\Sigma_2
\end{bmatrix}.
\]
Indeed, if $p = 0$ and $M$ is the matrix 
\[
\begin{bmatrix}
1 & -\lvert d_x \rvert^2 & f_1 \\
1 & 1 - \lvert d_x \rvert^2 & f_2 \\
1 & \Sigma_2 - \lvert d_x \rvert^2 & f_3 \\
\end{bmatrix},
\]
an easy induction gives that 
\[
M^m \begin{bmatrix}
1 \\
1\\
1
\end{bmatrix} 
=
m\cdot
\begin{bmatrix}
0 \\
1\\
1
\end{bmatrix}
+ 
\begin{bmatrix}
1 \\
1\\
1
\end{bmatrix}
+ 
\frac{1}{2}m(m+1)\cdot 
\begin{bmatrix}
\Sigma_1\Sigma_2\\
(\Sigma_2 - 1)\Sigma_1\\
\Sigma_2
\end{bmatrix}.
\]
The rest of the cases are similar.
\end{proof}

The additional conditions in Proposition \ref{prop: p and the order of the Coxeter matrix when bi-eigenvector} are of course not arbitrary: they are always satisfied when $\Gamma$ is twisted fractionally Calabi--Yau by Proposition \ref{prop:properties when one-point extension is fractionally Calabi-Yau}. Hence we obtain the following.

\begin{cor}\label{cor:proposition about order of matrix holds for Gamma twisted fractionally Calabi-Yau}
Assume that $\Gamma$ is twisted fractionally Calabi--Yau. Then the assumptions and results of Proposition \ref{prop: p and the order of the Coxeter matrix when bi-eigenvector} hold. 
\end{cor}

\section{2-representation-finite quadratic monomial algebras}
\label{Section:2-representation-finite quadratic monomial algebras}
Our aim in this section is to classify certain quadratic monomial $2$-hereditary algebras. This is achieved by relating star algebras to one-point extensions of the form described in Setup \ref{setup:graph, quiver, one-point extension}.

\subsection{Trivial extensions of star algebras}

We start by introducing some notation about star algebras. Let $S_{(r,s)}$ be an $(r,s)$-star quiver as in (\ref{eq:star quiver}) and let $\Lambda=\K S_{(r,s)}/\I$ be an $(r,s)$-star algebra. In particular $\Lambda$ is $2$-levelled, and if $r,s\geq 2$ then $\Lambda$ is also $1$-light. For $i\in \{1,\ldots,r\}$ we define the set $\mathcal{Z}(x_i)=\{b_j \mid a_i b_j\in \I\}$ and for $j\in \{1,\ldots,s\}$ we define the set $\mathcal{Z}(y_j)=\{a_i \mid a_i b_j\in \I\}$. Note that the relations generating $\I$ are necessarily monomial and can be described by a subset of $\{1,\ldots,i\}\times \{1,\ldots,j\}$. We encode this information in a quiver $B_{\Lambda}$ as follows:
\begin{itemize}
    \item The set of vertices of $B_{\Lambda}$ is the set $\{x_1,\ldots,x_r,y_1,\ldots,y_s\}$.
    \item The arrows of $B_{\Lambda}$ are of the form $\beta_{ji}:y_j\to x_i$ whenever $a_i b_j\not\in \I$.
\end{itemize}
Notice that the underlying graph of $B_{\Lambda}$ is bipartite with partition $\{x_1,\ldots,x_r\}\cup\{y_1,\ldots,y_s\}$.

\begin{defin}\label{def:balanced algebra}
    We say that a star algebra $\Lambda=\K S_{(r,s)}/\I$ is \emph{balanced} if every arrow in $S_{(r,s)}$ appears in at least one nonzero path of length $2$ in $\Lambda$ and in at least one zero path of length $2$ in $\Lambda$.
\end{defin}

Note that a star algebra $\Lambda=\K S_{(r,s)}/\I$ is balanced if and only if $1\leq \mathcal{Z}(x_i)\leq s-1$ for all $i\in\{1,\ldots,r\}$ and $1\leq \mathcal{Z}(y_j)\leq r-1$ for all $j\in\{1,\ldots,r\}$.

Recall that the \emph{trivial extension} of a finite-dimensional algebra $A$ is the symmetric algebra $T(A)$ given by $A \oplus D(A)$ as an $A$-bimodule and endowed with the multiplication induced by letting $D(A)$ square to zero, i.e.\ by setting $(a, f)\cdot (a', g) := (aa', a g + f a')$ for $(a, f), (a', g) \in T(\Lambda)$. In the case of balanced star algebras, we can compute the quiver with relations of the trivial extension of an $(r,s)$-star algebra using the following proposition.

\begin{prop}\label{prop: quiver and relations for trivial extension of star quiver algebra}
Let $\Lambda=\K S_{(r,s)}/\I$ be a balanced $(r,s)$-star algebra where $\I=(\rho)$. Then $T(\Lambda)=Q_{T(\Lambda)}/(\rho_{T(\Lambda)})$ where 
\begin{enumerate}
    \item[(i)] $Q_{T(\Lambda)}$ can be obtained by adjoining to the quiver $S_{(r,s)}$ the arrows of the quiver $B_{\Lambda}$, i.e.\ by adjoining an arrow $\beta_{ji} \colon y_j \rightarrow x_i$ for every non-zero path $a_i b_j$ in $\Lambda$, and
    \item[(ii)] $\rho_{T(\Lambda)}$ can be obtained by adjoining to $\rho$ the sets of relations given by
    \[
    \{ \beta_{ji} a_i - \beta_{ji'} a_{i'} \, \lvert \, a_{i} b_{j}, a_{i'}b_j \not \in \rho \, \, \text{for} \, \, j \, \, \text{such that} \, \,  1 \leq j \leq s \},
    \]
    and
    \[
    \{ b_{j} \beta_{ji} - b_{j'}\beta_{j'i} \, \lvert \, a_{i} b_{j}, a_{i}b_{j'} \not \in \rho \, \, \text{for} \, \, i \, \, \text{such that} \, \,  1 \leq i \leq r \},
    \]
    and the set of all paths of length greater than $3$. 
\end{enumerate}
\end{prop}

\begin{proof}
The computation of the quiver and relations of the trivial extension of an algebra can be found in \cite{FP02}; see also the related works \cite{Schroer99} and \cite{FSTTV22}.
\end{proof}

In other words, if $\Lambda=\K S_{(r,s)}/\I$ is a balanced $(r,s)$-star algebra, then the quiver $Q_{T(\Lambda)}$ of $T(\Lambda)$ has the form
\begin{equation}\label{eq:quiver of T(Lambda)}
\begin{tikzpicture}[scale=1, transform shape, baseline={(current bounding box.center)}]]
    \node (v1) at (-0.5,1.75) {$x_1$};
    \node (vr) at (-0.5,0.25) {$x_r$};
    \node (u1) at (2.5,1.75) {$y_1$};
    \node (us) at (2.5,0.25) {$y_s$};            
    \node (w1) at (4,1.75) {$x_1$};
    \node (wr) at (4,0.25) {$x_r$,};          
    \node (z) at (1,1) {$z$};
    \node (BL) at (3.25,-0.25) {$B_{\Lambda}$};

    \draw[loosely dotted] (-0.5,0.8) to (-0.5,1.2);  \draw[loosely dotted]  (2.5,0.8) to (2.5,1.2);
    \draw[loosely dotted]  (4,0.8) to (4,1.2);
    \draw[loosely dotted]  (3,1.75) to (3.4,1.75);
    \draw[loosely dotted]  (3,0.25) to (3.4,0.25);
    \draw[loosely dotted]  (3,1) to (3.4,1);
        
    \draw[->] (v1) to node[above] {$a_1$}  (z);
    \draw[->] (vr) to node[below] {$a_r$}  (z);
    \draw[->] (z) to node[above] {$b_1$}  (u1);
    \draw[->] (z) to node[below] {$b_s$}  (us);

    \tikzset{
	contour/.style={
		double distance=3mm,
		line cap=round,
		rounded corners=1pt,thick
        }
    }
    \draw[rounded corners=6pt] (2.25,2)-- (4.25,2)--(4.25,0)--(2.25,0)--cycle;
\end{tikzpicture}
\end{equation}
where we identify the two copies of the vertices $x_i$ for $i\in\{1,\ldots,r\}$ and the boxed area is a copy of the quiver $B_{\Lambda}$.

In certain cases we may obtain two algebras with isomorphic trivial extensions as per the following result.

\begin{thm}\cite[Theorem A]{FernandezPlatzeck06}\label{thrm:isomorphic trivial extensions}
Let $A$ be a finite-dimensional algebra such that all cycles in the quiver of $A$ are zero in $A$, let $\J$ be an ideal of $T(A)$ generated by exactly one arrow from each non-zero cycle in $T(A)$, and let $A' = T(A)/\J$. Then $T(A) \simeq T(A')$.
\end{thm}

We can apply Theorem \ref{thrm:isomorphic trivial extensions} in the case of a balanced star algebra as in the following lemma.

\begin{lem}\label{lem:trivial extension without arrows is setup 1}
Let $\Lambda=\K S_{(r,s)}/\I$ be a balanced $(r,s)$-star algebra. Let $\J$ be the ideal of $T(\Lambda)$ generated by $\{a_i\mid 1\leq i\leq r\}$ and define the algebra
\begin{equation}
\Gamma_{\Lambda}= T(\Lambda)/\J.
\end{equation}
Then the following hold.
\begin{enumerate}
    \item[(a)] $\Gamma_{\Lambda}=\K B_{\Lambda}[M]$ where $M$ is the unique indecomposable $B_{\Lambda}$-module with dimension vector $(1,\ldots,1)^{T}$. 
    \item[(b)] $T(\Lambda)\isom T(\Gamma_{\Lambda})$.
\end{enumerate}
\end{lem}

\begin{proof}
\begin{enumerate}
    \item[(a)] By definition, the quiver $Q_{\Gamma_{\Lambda}}$ is the quiver of $T(\Lambda)$ with the arrows $a_1,\ldots,a_r$ removed. In other words, by (\ref{eq:quiver of T(Lambda)}) we conclude that $Q_{\Gamma_{\Lambda}}$ has the form
    \begin{equation*}
    \begin{tikzpicture}[scale=1, transform shape, baseline={(current bounding box.center)}]]
       
        \node (u1) at (2.5,1.75) {$y_1$};
        \node (us) at (2.5,0.25) {$y_s$};            
        \node (w1) at (4,1.75) {$x_1$};
        \node (wr) at (4,0.25) {$x_r$.};          
        \node (z) at (1,1) {$z$};
        \node (BL) at (3.25,-0.25) {$B_{\Lambda}$};
        
        \draw[loosely dotted]  (2.5,0.8) to (2.5,1.2);
        \draw[loosely dotted]  (4,0.8) to (4,1.2);
        \draw[loosely dotted]  (3,1.75) to (3.4,1.75);
        \draw[loosely dotted]  (3,0.25) to (3.4,0.25);
        \draw[loosely dotted]  (3,1) to (3.4,1);
        
        \draw[->] (z) to node[above] {$b_1$}  (u1);
        \draw[->] (z) to node[below] {$b_s$}  (us);

        \tikzset{
            contour/.style={
                double distance=3mm,
			line cap=round,
			rounded corners=1pt,thick
            }
        }

        \draw[rounded corners=6pt] (2.25,2)-- (4.25,2)--(4.25,0)--(2.25,0)--cycle;
    \end{tikzpicture}
    \end{equation*}
    Moreover, by Proposition \ref{prop: quiver and relations for trivial extension of star quiver algebra} we see that the only relations in  $\Gamma_{\Lambda}$ are commutativity relations. More specifically, any two distinct paths in $Q_{T(\Lambda)}$ starting at the vertex $z$ and terminating at a vertex $x_i$ for some $i\in\{1,\ldots,r\}$ are of the form $\beta_{ji}b_j, \beta_{j'i}b_{j'}$, and we have $\beta_{ji}b_j-\beta_{j'i}b_{j'}\in \rho_{T(\Lambda)}$. It follows that $\beta_{ji}b_j-\beta_{j'i}b_{j'}\in\rho_{\Gamma_{\Lambda}}$ as well. Then $\Gamma_{\Lambda}=\K B_{\Lambda}[M]$ follows by a direct computation.

    \item[(b)] In view of Theorem \ref{thrm:isomorphic trivial extensions} we need to show that the set $\{a_i\mid 1\leq i\leq r\}$ consists of exactly one arrow from each non-zero cycle in $T(\Lambda)$. Since all paths of length greater than $3$ in $T(\Lambda)$ are zero, the only nonzero cycles in $\Lambda$ are of the form $a_i b_j \beta_{ji}$ with $a_i b_j\not\in \I$. In particular, every nonzero cycle in $\Lambda$ has exactly one arrow $a_i$. Since $\Lambda$ is balanced, each $a_i$ belongs to at least one cycle, and the claim follows. \qedhere
\end{enumerate}    
\end{proof}

We will need a way of relating the spectrum of the Coxeter matrices of $\Lambda$ and $\Gamma_{\Lambda}$. 
It turns out that the algebras $\Lambda$ and $\Gamma = \Gamma_{\Lambda}$ are derived equivalent, and this is sufficient as derived equivalent algebras have the same Coxeter polynomials; see e.g.\ \cite[Lemma 4.1]{Hap09}. 

To show that $\Lambda$ and $\Gamma_{\Lambda}$ are derived equivalent, we use a tilting complex that is Koszul dual to a $2$-APR tilting complex in the sense of \cite[Definition 3.14]{IO11}. 
Nevertheless, it takes less work to simply show that it is a tilting complex from scratch than to use Koszul duality and the results in \cite{IO11}.
Recall that we let $\epsilon_i = \Sigma_{\level(j) = i} e_{j}$ for $e_j$ a primitive idempotent for $\Gamma$.

\begin{prop}\label{prop: Gamma and Lambda have the same Coxeter polynomial}
Let $\Lambda=\K S_{(r,s)}/\I$ be a balanced $(r,s)$-star algebra. Let $X := \nu(\epsilon_0 \Gamma_0) \oplus \epsilon_2 \Gamma_0 \oplus \epsilon_1 \Gamma_0 [1]$ be a complex in $\D^b(\modu \Gamma)$. Then the following hold.
\begin{enumerate}
    \item[(a)] The Koszul dual $\Lambda^!$ is the bound quiver algebra with quiver
\[
\begin{tikzpicture}[scale=1, transform shape, baseline={(current bounding box.center)}]]
    \node (v1) at (-0.5,1.75) {$y_1$};
    \node (vr) at (-0.5,0.25) {$y_s$};
    \node (u1) at (2.5,1.75) {$x_1$};
    \node (us) at (2.5,0.25) {$x_r$.};      
    \node (z) at (1,1) {$z$};

    \draw[loosely dotted] (-0.5,0.8) to (-0.5,1.2);  \draw[loosely dotted]  (2.5,0.8) to (2.5,1.2);
        
    \draw[->] (v1) to node[above] {$b_1^{\ast}$}  (z);
    \draw[->] (vr) to node[below] {$b_s^{\ast}$}  (z);
    \draw[->] (z) to node[above] {$a_1^{\ast}$}  (u1);
    \draw[->] (z) to node[below] {$a_r^{\ast}$}  (us);
\end{tikzpicture}
\]
and relations given by all paths $b_j^{\ast}a_i^{\ast}$ such that $a_i b_j\not\in I$. In particular, $B_{\Lambda^!}$ is the bipartite complement of $B_{\Lambda}$.
    \item[(b)] $X$ is a tilting complex; 
    \item[(c)] $\End_{\D^b(\modu \Gamma)} (X) \simeq \Lambda^{!}$; and
    \item[(d)] $\Gamma = \Gamma_{\Lambda}$ is derived equivalent to $\Lambda$.
\end{enumerate}
\end{prop}
\begin{proof}
\begin{enumerate}
    \item[(a)] As we have mentioned, in Section \ref{subsec:Koszul algebras}, the quiver of $\Lambda^{!}$ is given by the opposite quiver of $S_{(r,s)}$. This shows the claim about the quiver of $\Lambda!$. 
    The description of the relations follows by the description of the quadratic dual of $\Lambda$ as a quotient of a path algebra of a quiver with relations; see e.g.\ \cite[Definition 2.8.1]{Beilinson-Ginzburg-Soergel}.
    In particular, since there are only (quadratic) monomial relations, $\Lambda^!$ is also quadratic monomial, and its relations are exactly those quadratic monomials that are not relations for $\Lambda$.
    \item[(b)] Since $\epsilon_0 \Gamma_0$ is projective, we know $\nu(\epsilon_0 \Gamma_0)$ is injective and thus to check that it is rigid, it suffices to check that 
\[
\Hom_{\D^b (\modu \Gamma^)}(\nu(\epsilon_0 \Gamma_0), (\epsilon_2 \Gamma_0 \oplus \epsilon_1 \Gamma_0 [1])[i \neq 0]) = 0
\]
and that
\[
\Hom_{\D^b (\modu \Gamma^)}(\epsilon_2 \Gamma_0 \oplus \epsilon_1 \Gamma_0 [1], (\epsilon_2 \Gamma_0 \oplus \epsilon_1 \Gamma_0 [1])[i \neq 0]) = 0.
\]
The latter of these follows by the fact that $\Gamma$ is $2$-levelled Koszul, see also the proof of Proposition \ref{prop:Koszul and n-levelled}(b). For the former we consider the minimal projective resolution of an indecomposable summand of $\nu(\epsilon_0 \Gamma_0)$.
Such a summand must be of the form $\nu(e_x \Gamma) \simeq D(\Gamma e_x)$, hence we see that it must have a minimal projective resolution of the form 
\[
P^{-2} \rightarrow P^{-1} \rightarrow e_z \Gamma, 
\]
with $P^{-1}$ a direct sum including a copy of $e_y \Gamma$ for each vertex $y$ such that there is no arrow from $y$ to $x$ in the quiver of $\Gamma$. 
In particular, note that it does not contain any summand of the form $\epsilon_2 \Gamma = e_z \Gamma$.
Moreover, any summand of $P^{-2}$ is of the form $e_x \Gamma$ for $x$ some vertex in level $0$. 
Consequently, by using that $\Ext^{i}_{\Gamma}(M, S) \simeq \Hom_{\Gamma}(P^{-i}(M), S)$ whenever $P^{\bullet}(M)$ is minimal projective resolution of $M$ and $S$ is a simple module, it is clear that 
\[
\Hom_{\D^b (\modu \Gamma^)}(\nu(\epsilon_0 \Gamma_0), (\epsilon_2 \Gamma_0 \oplus \epsilon_1 \Gamma_0 [1])[i \neq 0]) = 0.
\]
Hence, the complex $X$ is rigid. 
Since $\Gamma$ is finite-dimensional and of finite global dimension, we can use the radical filtration of $\nu(\epsilon_0 \Gamma)$ to deduce that the thick subcategory generated by $X$ equals $\Htpycat^b(\proj \Gamma)$, which is itself equivalent to $\D^b(\modu \Gamma)$ since $\Gamma$ is of finite global dimension. 
Consequently, we can conclude that $X$ is a tilting complex.  
\item[(c)] Note that
\[
\Hom_{\D^b (\modu \Gamma^)}(D(\Gamma e_x), e_z \Gamma_0) \simeq k \simeq \Hom_{\D^b (\modu \Gamma^)}(e_z \Gamma_0, e_y \Gamma_0 [1])
\]
for any vertex $x$ in level $0$ and any vertex $y$ in level $1$. 
Hence, the quiver of $\End_{\D^b (\modu \Gamma^)} (X)$ is the same as that of $\Lambda^!$. 

Since the quiver is the same as that of $\Lambda^!$, any relation must be a quadratic monomial, i.e.\ a path of length $2$.
Hence, to determine the relations, it suffices to note that, by the form of the minimal projective resolution of $D(\Gamma e_x)$, we have
\[
\Hom_{\D^b (\modu \Gamma^)}(D(\Gamma e_x), e_y \Gamma_0 [1]) \simeq k
\]
if there is no arrow in the quiver of $\Gamma$ from a vertex $x$ in level $0$ to a vertex $y$ in level $1$, and
\[
\Hom_{\D^b (\modu \Gamma^)}(D(\Gamma e_x), e_y \Gamma_0 [1]) = 0
\]
otherwise. 
\item[(d)] Finally we note that parts (b) and (c) give that $\Gamma_{\Lambda}$ is derived equivalent to $\Lambda^{!}$. Since $\Lambda$ is derived equivalent to $\Lambda^{!}$ by Proposition \ref{prop:Koszul and n-levelled}(b) (as $\Lambda$ is $2$-levelled and Koszul), it follows that $\Gamma_{\Lambda}$ and $\Lambda$ are also derived equivalent.  \qedhere
\end{enumerate}
\end{proof}

\subsection{Star algebras and higher homological algebra}

The motivation for studying balanced $(r,s)$-star algebras comes from the following result.

\begin{prop}\label{prop:star algebras which are 2-hereditary}
Let $A=\K Q/\I$ be a quadratic monomial algebra.
\begin{enumerate}
    \item[(a)] \cite[Corollary 4.21]{ST24} Assume that $\Ext^{1}_{A}(D A, A) = 0$. Then $A$ is an $(r,s)$-star algebra. In particular, every basic $2$-hereditary quadratic monomial algebras is an $(r,s)$-star algebra.
    \item[(b)] \cite[Lemma 4.16 and Lemma 4.24]{ST24} Assume that $A$ is $2$-hereditary. Then one of the following holds.
    \begin{enumerate}
        \item[(i)] $(r,s)=(1,1)$, $\I=(a_1 b_1)$ and $\Lambda$ is $2$-representation-finite, or
        \item[(ii)] $r,s\geq 4$, for every $i\in\{1,\ldots,r\}$ we have $2\leq \abs{\mathcal{Z}(x_i)}\leq s-2$, and for every $j\in\{1,\ldots,s\}$ we have $2\leq \abs{\mathcal{Z}(y_j)}\leq r-2$. 
    \end{enumerate}
\end{enumerate}
\end{prop}

By Proposition \ref{prop:star algebras which are 2-hereditary} we have that every $2$-hereditary quadratic monomial algebra is either of linearly oriented type $A_3$ modulo the ideal generated by all paths of length two, or a balanced $(r,s)$-star algebra. The first case is trivial and hence we may focus on the second one. Our aim is to obtain more necessary conditions in this case. To do this we use Lemma \ref{lem:trivial extension without arrows is setup 1} and the results of Section \ref{Section:One-point extensions of bipartite hereditary algebras}. Our first observation is the following.

\begin{lem}\label{lem:twisted fractionally Calabi-Yau of trivial extension}
Let $\Lambda=\K S_{(r,s)}/\I$ be a balanced $(r,s)$-star algebra. Then $\Lambda$ is twisted fractionally Calabi--Yau if and only if $\Gamma_{\Lambda}$ is twisted fractionally Calabi--Yau. 
\end{lem}

\begin{proof}
Since $\K$ is algebraically closed, we may apply \cite[Theorem 1.3 (i) and (iv)]{CDIM20} to obtain that $\Lambda$ is twisted fractionally Calabi--Yau if and only if $T(\Lambda)$ is twisted periodic. Since by Lemma \ref{lem:trivial extension without arrows is setup 1} we have that $T(\Lambda)\simeq T(\Gamma_{\Lambda})$, this is equivalent to $\Gamma_{\Lambda}$ being twisted fractionally Calabi--Yau, as claimed.
\end{proof}

To be able to utilize Lemma \ref{lem:twisted fractionally Calabi-Yau of trivial extension}, we focus on a special class of balanced $(r,s)$-star algebras.

\begin{defin}\label{def:semi-regular star algebra}
We say that an $(r,s)$-star algebra $\Lambda=\K S_{(r,s)}/\I$ is \emph{semi-regular of bidegree $(\Sigma_1,\Sigma_2)$} if the underlying bipartite graph $\graph{B_{\Lambda}}$ of the quiver $B_{\Lambda}$ is semi-regular of bidegree $(\Sigma_1,\Sigma_2)$ with respect to the coloring $\{x_1,\ldots,x_r\}\cup\{y_1,\ldots,y_s\}$.
\end{defin}

Unraveling the definition we see that an $(r,s)$-star algebra $\Lambda=\K S_{(r,s)}/\I$ is semi-regular of bidegree $(\Sigma_1,\Sigma_2)$ if and only if for each vertex $x_i$ in the star quiver $S_{(r,s)}$, there exist exactly $\Sigma_1$ arrows $b_j$ such that $a_i b_j\not\in\I$ and that for each vertex $y_j$ in the star quiver $S_{(r,s)}$, there exist exactly $\Sigma_2$ arrows $a_i$ such that $a_ib_j\not\in I$. We remark that we know of no example of a $2$-representation-finite $(r,s)$-star algebra which is not semi-regular. 

Notice that if a semi-regular $(r,s)$-star algebra $\Lambda=\K S_{(r,s)}/\I$ is balanced, then $2\leq \Sigma_1\leq s-2$ and $2\leq \Sigma_2\leq r-2$. In particular, this holds if $\Lambda$ is $2$-representation-finite and $(r,s)\neq (1,1)$.

The following is the main result of this section.

\begin{thm}\label{thrm:twisted fractionally Calabi-Yau semi-regular star algebra}
Assume that $\Lambda=\K S_{(r,s)}/\I$ is a semi-regular $(r,s)$-star algebra of bidegree $(\Sigma_1,\Sigma_2)$. Let $M$ be the unique indecomposable $B_{\Lambda}$-module with dimension vector $(1,\ldots,1)^T$. Then the following hold. 
\begin{enumerate}
    \item[(a)] The quadruple $(B_{\Lambda},M,\Sigma_1,\Sigma_2)$ satisfies the conditions of Setup \ref{setup:graph, quiver, one-point extension} with $\lvert d_x \rvert^2=r$ and $\lvert d_y \rvert^2=s$.
    \item[(b)] $r\Sigma_1 = s\Sigma_2 $.
\end{enumerate}
If moreover $\Lambda$ is balanced and twisted fractionally Calabi--Yau, then the following also hold.
\begin{enumerate}
    \item[(c)] $\Gamma_{\Lambda}$ is twisted fractionally Calabi--Yau.
    \item[(d)] $\Sigma_1\Sigma_2+r+s-r\Sigma_1 \in\{0,1,2,3,4\}$.
    \item[(e)] $\graph{B_{\Lambda}}$ is a reflexive graph. Furthermore, $\graph{B_{\Lambda}}$ is a Salem graph if and only if $\sqrt{\Sigma_1\Sigma_2}>2$.
\end{enumerate}
\end{thm}

\begin{proof}
Part (b) follows by the semi-regularity assumption. We now show part (a). Notice that conditions (i) and (ii) of Setup \ref{setup:graph, quiver, one-point extension} are satisfied by definition and so we only need to show condition (iii). Set $d_{M}^{T}=\dimv(M)$, $d_x=(1,\ldots,1)^T\in\CC^{r}$ and $d_{y}=(1,\ldots,1)^T\in\CC^{s}$. Clearly we have that $\lvert d_x \rvert^2=r$ and $\lvert d_y \rvert^2=s$. We claim that $\begin{bsmallmatrix} d_x \\ d_y \end{bsmallmatrix}$ is a bi-eigenvector of $A(\graph{B_{\Lambda}})$ with bi-eigenvalue $(\Sigma_1,\Sigma_2)$. First notice that the matrix $A(\graph{B_{\Lambda}})$ has the form
\[
A(\graph{B_{\Lambda}}) =\begin{bmatrix}
0_r & R \\
R^T & 0_s
\end{bmatrix}
\]
where $R$ is an $r\times s$ matrix such that
\[
r_{ij} = \begin{cases}
    1, &\mbox{if $a_i b_j\not\in \I$,} \\
    0, &\mbox{if $a_i b_j\in \I$.}
\end{cases}
\]
Let $R_i$ be the $i$-th column of $R$ and $C_j$ be the $j$-th column of $R$. We count the numbers of $1$'s in $R_i$ and $C_j$ via 
\[
\abs{R_i} \coloneqq \sum_{j=1}^{s}r_{ij} \text{ and } \abs{C_j} \coloneqq \sum_{i=1}^{r} r_{ij}.
\]
Then
\[
\abs{R_i} = \abs{\{\text{$y_j\in Q_0$ such that $a_i b_j\neq 0$}\}} = \deg_{B_{\Lambda}}(x_i) = \Sigma_1,
\]
where the last equality follows since $\Lambda$ is semi-regular of bidegree $(\Sigma_1,\Sigma_2)$. Similarly we obtain $\abs{C_j}=\Sigma_2$. Then
\[
\begin{bmatrix}
0_r & R \\
R^T & 0_s
\end{bmatrix}
\begin{bmatrix}
    d_x \\ d_y 
\end{bmatrix}=
\begin{bmatrix}
0_r & R \\
R^T & 0_s
\end{bmatrix}\begin{bmatrix}
    1 \\ \vdots \\ 1 
\end{bmatrix} = \begin{bmatrix}
    \abs{R_1} \\ \vdots \\ \abs{R_r} \\ \abs{C_1} \\ \vdots \\ \abs{C_s} 
\end{bmatrix} = \begin{bmatrix}
    \Sigma_1 \\ \vdots \\ \Sigma_1 \\ \Sigma_2 \\ \vdots \\ \Sigma_2 
\end{bmatrix}  
\]
which shows our claim and proves part (a). 

Next assume that $\Lambda$ is twisted fractionally Calabi--Yau. By Lemma \ref{lem:twisted fractionally Calabi-Yau of trivial extension} we obtain that $\Gamma_{\Lambda}$ is also twisted fractionally Calabi--Yau, proving part (c). Hence the conditions of Proposition \ref{prop:properties when one-point extension is fractionally Calabi-Yau} hold for $\Gamma_{\Lambda}=\K B_{\Lambda}[M]$. In particular we obtain that 
\[
\Sigma_1\Sigma_2 + \lvert d_x \rvert^2 + \lvert d_y \rvert^2 - \lvert d_x \rvert^2 \Sigma_1\in \{0,1,2,3,4\}. 
\]
By using that $\lvert d_x \rvert^2=r$ and $\lvert d_y \rvert^2=s$, we obtain part (d).

Finally, by Proposition \ref{prop:star algebras which are 2-hereditary}(b) and since $\Lambda$ is balanced we have that $\Sigma_1,\Sigma_2\geq 2$, from which we obtain that $\sqrt{\Sigma_1\Sigma_2}\geq 2$. Part (e) follows immediately by Proposition \ref{prop:if Gamma is twisted fractionally Calabi-Yau then graph is reflexive}.
\end{proof}

In particular, we are especially interested in the case where a star algebra is $2$-representation-finite. Since being $2$-representation-finite implies being twisted fractionally Calabi--Yau, we obtain this corollary of Theorem \ref{thrm:twisted fractionally Calabi-Yau semi-regular star algebra}.

\begin{cor}\label{cor:2-RF semi-regular star algebra}
Let $\Lambda=\K S_{(r,s)}/\I$ be a $2$-representation-finite $(r,s)$-star algebra. Then the following holds.
\begin{enumerate}
    \item[(a)] $\graph{B_{\Lambda}}$ is a reflexive graph. Furthermore, $\graph{B_{\Lambda}}$ is a Salem graph if and only if $\sqrt{\Sigma_1\Sigma_2}>2$.
\end{enumerate}
Assume moreover that $\Lambda$ is balanced. Then the following holds.
\begin{enumerate}
    \item[(b)] $\Lambda$ is $\ell$-homogeneous. 
\end{enumerate}
Finally, assume moreover that $\Lambda$ is semi-regular of bidegree  $(\Sigma_1,\Sigma_2)$. Then the following also hold.
\begin{enumerate}
    \item[(c)] $r\Sigma_1 = s\Sigma_2 $.
    \item[(d)] $\Sigma_1\Sigma_2+r+s-r\Sigma_1 \in\{1,2,3,4\}$.
\end{enumerate}
\end{cor}

\begin{proof}
Recall that by Proposition \ref{prop:star algebras which are 2-hereditary}(b) we have that either $(r,s)=(1,1)$ and $\graph{B_{\Lambda}}$ is the graph with two vertices and no edge, or $\Lambda$ is balanced. In the first case the only eigenvalue of the adjacency matrix of $\graph{B_{\Lambda}}$ is $0$ and so $\graph{B_{\Lambda}}$ is clearly reflexive but not Salem, showing part (a) for this case. Hence for the rest of the proof we may assume that $\Lambda$ is balanced.

Since $\Lambda$ is $2$-representation-finite, we have that $\Lambda$ is twisted fractionally Calabi--Yau by Theorem \ref{thm:twisted fractionally CY and n-RF}(a). Hence part (a) is just part (e) of Theorem \ref{thrm:twisted fractionally Calabi-Yau semi-regular star algebra}. Moreover, since $\Lambda$ is balanced and $2$-representation-finite we have that $r,s\geq 4$ by Proposition \ref{prop:star algebras which are 2-hereditary}. Hence $\Lambda$ is $0$-light and so $\Lambda$ is $\ell$-homogeneous by Proposition \ref{prop:2-levelled light and 2-RF is twisted fractionally Calabi--Yau}(b), showing part (b). Assuming that $\Lambda$ is semi-regular of bidegree  $(\Sigma_1,\Sigma_2)$, we obtain part (c) by Theorem \ref{thrm:twisted fractionally Calabi-Yau semi-regular star algebra}(b) we obtain that $\Sigma_1\Sigma_2+r+s-r\Sigma_1\in\{0,1,2,3,4\}$ by by Theorem \ref{thrm:twisted fractionally Calabi-Yau semi-regular star algebra}(d). Hence it is only left to show that $\Sigma_1\Sigma_2+r+s-r\Sigma_1\neq 0$.  

Assume to a contradiction that $\Sigma_1\Sigma_2+r+s-r\Sigma_1=0$. By Theorem \ref{thrm:twisted fractionally Calabi-Yau semi-regular star algebra} we have that the quadruple $(B_{\Lambda},M,\Sigma_1,\Sigma_2)$ satisfies the conditions of Setup \ref{setup:graph, quiver, one-point extension} with $\lvert d_x \rvert^2=r$ and $\lvert d_y \rvert^2=s$ and the $\Gamma_{\Lambda}$ is twisted fractionally Calabi--Yau. By Corollary \ref{cor:proposition about order of matrix holds for Gamma twisted fractionally Calabi-Yau} we obtain that the order of the matrix $-\Phi_{\Gamma_{\Lambda}}$ is infinite. But this contradicts Proposition \ref{prop:fractionally Calabi--Yau implies eigenvalues of Coxeter matrix are in unit circle}.
\end{proof}

The following proposition shows another strong connection between the higher homological properties of a balanced star algebra $\Lambda$ and the algebra $\Gamma_{\Lambda}$.

\begin{prop}
Let $\Lambda=\K S_{(r,s)}/\I$ be a balanced $(r,s)$-star algebra. 
Then $\Lambda$ is $2$-representation-finite if and only if $\Gamma_{\Lambda}$ is $2$-representation-finite.
In particular, $\Gamma_{\Lambda}$ is $\ell$-homogeneous.
\end{prop}

\begin{proof}
Note that the second claim follows from the first by Proposition \ref{prop:2-levelled light and 2-RF is twisted fractionally Calabi--Yau}(b).

To show the first claim, we recall that $\Lambda$ is Koszul.
Thus, if $\Lambda$ is $2$-representation-finite, then $\Pi(\Lambda^!)$ is almost Koszul by \cite{GrantIyama2020}. 
By Proposition \ref{prop: quiver and relations for trivial extension of star quiver algebra}, we know that $T(\Lambda)$ is a quadratic algebra, i.e.\ all its relations are sums of paths of length $2$. 
It is easy to check that we have $\Lambda_2 = \soc_{\Lambda^{e}} \Lambda$, where $\Lambda_2$ is the degree $2$ part of $\Lambda$ and $\Lambda^{e} = \Lambda^{\op} \otimes \Lambda$ is the enveloping algebra of $\Lambda$. 
Hence, we can apply \cite[Theorem 5.2 (c)] {GrantIyama2020} and deduce that $\Pi(\Lambda^{!})^{!} \simeq T(\Lambda)$. 

Considering an explicit description of the quadratic dual of a quadratic algebra in terms of quivers with relations as in \cite[Definition 2.8.1.]{Beilinson-Ginzburg-Soergel}, we see that $(\Gamma_{\Lambda})^{!}$ can be obtained as a cut 
of $\Pi(\Lambda^{!})^{!}$ in the sense of \cite[Definition 3.1]{HI11}, and hence we have shown one implication in the first claim. 
Indeed, every cut of a selfinjective quiver with potential is $2$-representation-finite by \cite[Theorem 3.11]{HI11}, and the $(2+1)$-preprojective algebra of a $2$-representation-finite algebra can be given by a selfinjective quiver with potential by \cite[Proposition 3.9]{HI11}.
The other implication follows by essentially the same argument.
\end{proof}

\subsection{Classification results in the regular and edge-transitive case}\label{sec:classification results}

In this section we classify quadratic monomial $2$-representation-finite algebras $\Lambda=\K S_{(r,s)}$ based on certain graph theoretic properties of the graph $\graph{B_{\Lambda}}$ under the assumption that the field $\K$ has characteristic zero. The proofs are given using classification results from the literature of graph theory and are included in Appendix \ref{Section:Classification in regular and edge-transitive case} for the interested reader. For naming conventions on graphs we use the Encyclopedia of Graphs \cite{EoG}. We start with the case where $\graph{B_{\Lambda}}$ is regular.

\begin{prop}\label{prop: classification in the regular case}
Assume that $\charac(\K)=0$. Let $\Lambda=\K S_{(r,s)}/\I$ be a $2$-representation-finite $(r,s)$-star algebra. If $B_{\Lambda}$ is regular, then the underlying graph $\graph{B_{\Lambda}}$ is one of the following:
\begin{enumerate}[(a)]
    \item the bipartite complement of the path graph $P_2$;  
    \item the cycle graph $C_8$; 
    \item the Heawood graph (G-5 in the Encyclopedia of graphs) or its bipartite complement.
\end{enumerate}
\end{prop}

Based on this characterization we can show the following. 

\begin{cor}\label{cor:the graph B_Lambda is reflexive and almost always Salem}
Assume that $\charac(\K)=0$. Let $\Lambda=\K S_{(r,s)}/\I$ be a $2$-representation-finite $(r,s)$-star algebra. Assume that $\Lambda$ is semi-regular of bidegree $(\Sigma_1,\Sigma_2)$. Then $\graph{B_{\Lambda}}$ is a reflexive graph. Moreover, $\graph{B_{\Lambda}}$ is also a Salem graph if and only if $(r,s)\neq (1,1)$ and $(r,s)\neq (4,4)$.
\end{cor}

\begin{proof}
By Corollary \ref{cor:2-RF semi-regular star algebra}(a) it is enough to show that $\sqrt{\Sigma_1\Sigma_2}\leq 2$ if and only if $(r,s)= (1,1)$ or $(r,s)= (4,4)$. 

Since $\Lambda$ is $2$-representation-finite, by Proposition \ref{prop:star algebras which are 2-hereditary}(b) we have that $(r,s,\Sigma_1,\Sigma_2)=(1,1,1,1)$ or $\Lambda$ is balanced and $2\leq \Sigma_1\leq s-2$ and $2\leq \Sigma_2\leq r-2$. Hence $\sqrt{\Sigma_1\Sigma_2}\leq 2$ is equivalent to $(r,s)=(1,1)$ or $\Sigma_1=\Sigma_2=2$. But if $\Sigma_1=\Sigma_2=2$, then $\graph{B_{\Lambda}}$ is regular and so by Proposition \ref{prop: classification in the regular case} we obtain that $\Sigma_1=\Sigma_2=2$ if and only if $(r,s)=(4,4)$, as required.
\end{proof}

Next, recall that a graph is called \emph{edge-transitive} if for any pair of edges $e_1$, $e_2$ in the graph there is an automorphism of the graph mapping $e_1$ to $e_2$.
All graphs appearing in Proposition \ref{prop: classification in the regular case} are edge-transitive. In fact, for all examples of quadratic monomial $2$-representation-finite algebras $\Lambda=\K S_{(r,s)}/\I$ that we know of, the graph $\graph{B_{\Lambda}}$ is edge-transitive. This motivates the classification given in the following proposition.

\begin{prop}\label{prop: classification in the edge-transitive case}
Assume that $\charac(\K)=0$. Let $\Lambda=\K S_{(r,s)}/\I$ be a $2$-representation-finite $(r,s)$-star algebra. If $B_{\Lambda}$ is edge-transitive, then the underlying graph $\graph{B_{\Lambda}}$ is one of graphs appearing in Proposition \ref{prop: classification in the regular case}(a)---(c) or: 
\begin{enumerate}
    \item[(d)] the graph G-9P730 in the Encyclopedia of Graphs or its bipartite complement;
    \item[(e)] the graph G-9P731 in the Encyclopedia of Graphs or its bipartite complement.
\end{enumerate}
\end{prop}

Note that, with the exception of case (a) in Proposition \ref{prop: classification in the regular case}, the graphs appearing in Proposition \ref{prop: classification in the edge-transitive case} all appear in pairs with their bipartite complements as the bipartite complement of $C_8$ is $C_8$. This fails in the case of Proposition \ref{prop: classification in the regular case}(a), essentially because in that case the algebra $\Lambda$ is not balanced.

Moreover, if $G$ is one of the graphs in Proposition \ref{prop: classification in the regular case} or Proposition \ref{prop: classification in the edge-transitive case}, then there exists a unique $2$-representation-finite $(r,s)$-star algebra $\Lambda$ with $\graph{B_{\Lambda}}=G$. The graphs and the associated $2$-representation-finite quadratic monomial algebras described in these classification are as follows.

\begin{enumerate}
    \item[(a)] If $\graph{B_{\Lambda}}$ is the bipartite complement of $P_2$, then $r=s=1$, $\Sigma_1=\Sigma_2=1$, and 
    \begin{equation}\label{eq:the A3 example}
    S_{(r,s)}=\begin{tikzpicture}[scale=0.8, transform shape, baseline={(current bounding box.center)}]
    
    \node (x1) at (-0.5,1) {$x_1$};             
    \node (z) at (2.5,1) {$z$};
    \node (y1) at (5.5, 1) {$y_1$};
    
    \node (u) at (0,0.4) { };
    \draw[->] (x1) to node[above] {$a_1$}  (z);
    \draw[->] (z) to node[above] {$b_1$}  (y1);
\end{tikzpicture}
    \quad
    B_{\Lambda}=\begin{tikzpicture}[scale=0.8, transform shape, baseline={(current bounding box.center)}]
    \node (u) at (0,-0.3) { };
    \node (y1) at (0,0) {$y_1$};
    \node (x1) at (2.5,0) {$x_1$.};
\end{tikzpicture}
    \end{equation}
    That is, in this case $\Lambda=\K S_{(1,1)}/(a_1 b_1)$.
    
    \item[(b)] If $\graph{B_{\Lambda}}=C_8$, then $r=s=4$, $\Sigma_1=\Sigma_2=2$, and
    \begin{equation}\label{eq:the C_8 example}
    S_{(r,s)} = 
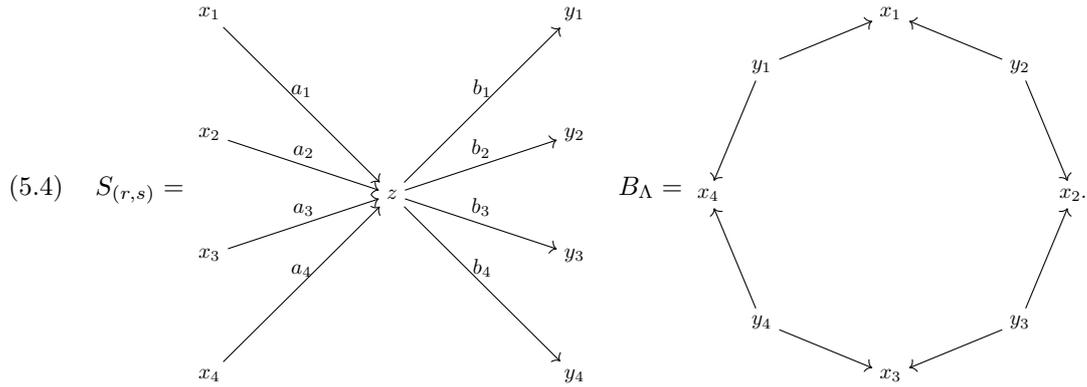
\begin{tikzpicture}[scale=0.8, transform shape, baseline={(current bounding box.center)}]
    \node (x1) at (-0.5,4) {$x_1$};
    \node (x2) at (-0.5,2) {$x_2$};
    \node (x3) at (-0.5,0) {$x_3$};
    \node (x4) at (-0.5,-2) {$x_4$};    
    \node (y1) at (5.5,4) {$y_1$};  
    \node (y2) at (5.5,2) {$y_2$};
    \node (y3) at (5.5,0) {$y_3$};
    \node (y4) at (5.5,-2) {$y_4$};               
    \node (z) at (2.5,1) {$z$};

    \draw[->] (x1) to node[above] {$a_1$}  (z);
    \draw[->] (x2) to node[above] {$a_2$}  (z);
    \draw[->] (x3) to node[above] {$a_3$}  (z);
    \draw[->] (x4) to node[above] {$a_4$}  (z);
    \draw[->] (z) to node[above] {$b_1$}  (y1);
    \draw[->] (z) to node[above] {$b_2$}  (y2);
    \draw[->] (z) to node[above] {$b_3$}  (y3);
    \draw[->] (z) to node[above] {$b_4$}  (y4);
\end{tikzpicture}\quad
B_{\Lambda}=\begin{tikzpicture}[scale=0.8, transform shape, baseline={(current bounding box.center)}]

\def\n{8}

\def\radius{3}

    \node (x2) at ({360/\n * (1 - 1)}:\radius) {$x_2$.};
    \node (y2) at ({360/\n * (2 - 1)}:\radius) {$y_2$};
    \node (x1) at ({360/\n * (3 - 1)}:\radius) {$x_1$};
    \node (y1) at ({360/\n * (4 - 1)}:\radius) {$y_1$};
    \node (x4) at ({360/\n * (5 - 1)}:\radius) {$x_4$};
    \node (y4) at ({360/\n * (6 - 1)}:\radius) {$y_4$};
    \node (x3) at ({360/\n * (7 - 1)}:\radius) {$x_3$};
    \node (y3) at ({360/\n * (8 - 1)}:\radius) {$y_3$};

    \draw[->] (y2) to (x1);
    \draw[->] (y2) to (x2);
    \draw[->] (y3) to (x2);
    \draw[->] (y3) to (x3);
    \draw[->] (y4) to (x3);
    \draw[->] (y4) to (x4);
    \draw[->] (y1) to (x4);
    \draw[->] (y1) to (x1);
\end{tikzpicture}
\end{equation}
That is, in this case $\Lambda=\K S_{(4,4)}/(a_1 b_3, a_1 b_4, a_2 b_1, a_2 b_4, a_3 b_1, a_3 b_2, a_4 b_2, a_4 b_3)$.

\item[(c)] If $\graph{B_{\Lambda}}$ is the Heawood graph, then $r=s=7$, $\Sigma_1=\Sigma_2=3$, and
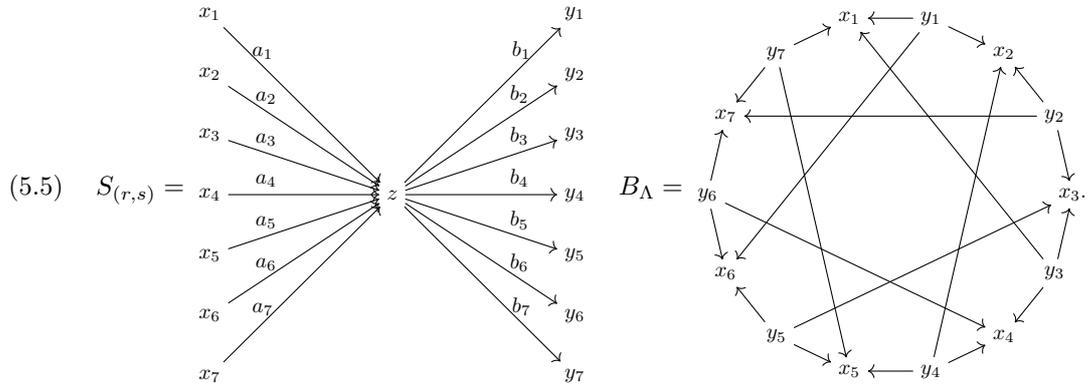
\begin{equation}\label{eq:the Heawood graph example}
    S_{(r,s)} = \begin{tikzpicture}[scale=0.8, transform shape, baseline={(current bounding box.center)}]
    \node (x1) at (-0.5,4) {$x_1$};
    \node (x2) at (-0.5,3) {$x_2$};
    \node (x3) at (-0.5,2) {$x_3$};
    \node (x4) at (-0.5,1) {$x_4$};
    \node (x5) at (-0.5,0) {$x_5$};
    \node (x6) at (-0.5,-1) {$x_6$};
    \node (x7) at (-0.5,-2) {$x_7$};    
    \node (y1) at (5.5,4) {$y_1$};  
    \node (y2) at (5.5,3) {$y_2$};
    \node (y3) at (5.5,2) {$y_3$};
    \node (y4) at (5.5,1) {$y_4$};    
    \node (y5) at (5.5,0) {$y_5$};  
    \node (y6) at (5.5,-1) {$y_6$};
    \node (y7) at (5.5,-2) {$y_7$};                   
    \node (z) at (2.5,1) {$z$};

    \draw[->] (x1) to node[near start, above] {$a_1$}  (z);
    \draw[->] (x2) to node[near start, above] {$a_2$}  (z);
    \draw[->] (x3) to node[near start, above] {$a_3$}  (z);
    \draw[->] (x4) to node[near start, above] {$a_4$}  (z);
    \draw[->] (x5) to node[near start, above] {$a_5$}  (z);
    \draw[->] (x6) to node[near start, above] {$a_6$}  (z);
    \draw[->] (x7) to node[near start, above] {$a_7$}  (z);
 
    \draw[->] (z) to node[near end, above] {$b_1$}  (y1);
    \draw[->] (z) to node[near end, above] {$b_2$}  (y2);
    \draw[->] (z) to node[near end, above] {$b_3$}  (y3);
    \draw[->] (z) to node[near end, above] {$b_4$}  (y4);    
    \draw[->] (z) to node[near end, above] {$b_5$}  (y5);
    \draw[->] (z) to node[near end, above] {$b_6$}  (y6);
    \draw[->] (z) to node[near end, above] {$b_7$}  (y7);
\end{tikzpicture}\quad
B_{\Lambda}=\begin{tikzpicture}[scale=0.8, transform shape, baseline={(current bounding box.center)}]

\def\n{14}

\def\radius{3}

    \node (x3) at ({360/\n * (1 - 1)}:\radius) {$x_3$.};
    \node (y2) at ({360/\n * (2 - 1)}:\radius) {$y_2$};
    \node (x2) at ({360/\n * (3 - 1)}:\radius) {$x_2$};
    \node (y1) at ({360/\n * (4 - 1)}:\radius) {$y_1$};
    \node (x1) at ({360/\n * (5 - 1)}:\radius) {$x_1$};
    \node (y7) at ({360/\n * (6 - 1)}:\radius) {$y_7$};
    \node (x7) at ({360/\n * (7 - 1)}:\radius) {$x_7$};
    \node (y6) at ({360/\n * (8 - 1)}:\radius) {$y_6$};
    \node (x6) at ({360/\n * (9 - 1)}:\radius) {$x_6$};
    \node (y5) at ({360/\n * (10 - 1)}:\radius) {$y_5$};
    \node (x5) at ({360/\n * (11 - 1)}:\radius) {$x_5$};
    \node (y4) at ({360/\n * (12 - 1)}:\radius) {$y_4$};
    \node (x4) at ({360/\n * (13 - 1)}:\radius) {$x_4$};
    \node (y3) at ({360/\n * (14 - 1)}:\radius) {$y_3$};

    \draw[->] (y1) to (x1);
    \draw[->] (y1) to (x2);
    \draw[->] (y1) to (x6);
    \draw[->] (y2) to (x2);
    \draw[->] (y2) to (x3);
    \draw[->] (y2) to (x7);
    \draw[->] (y3) to (x3);
    \draw[->] (y3) to (x4);
    \draw[->] (y3) to (x1);
    \draw[->] (y4) to (x4);
    \draw[->] (y4) to (x5);
    \draw[->] (y4) to (x2);
    \draw[->] (y5) to (x5);
    \draw[->] (y5) to (x6);
    \draw[->] (y5) to (x3);
    \draw[->] (y6) to (x6);
    \draw[->] (y6) to (x7);
    \draw[->] (y6) to (x4);
    \draw[->] (y7) to (x1);
    \draw[->] (y7) to (x7);
    \draw[->] (y7) to (x5);
\end{tikzpicture}
\end{equation}
That is, in this case $\Lambda=\K S_{(7,7)}/(a_1 b_3, a_1 b_4, a_1 b_5, a_1 b_7, a_2 b_1, a_2 b_4, a_2 b_5, a_2 b_6, a_3 b_2, a_3 b_5, a_3 b_6, \\ a_3 b_7, a_4 b_1, a_4 b_3, a_4 b_6, a_4 b_7, a_5 b_1, a_5 b_2, a_5 b_4, a_5 b_7, a_6 b_1, a_6 b_2, a_6 b_3, a_6 b_5, a_7 b_2, a_7 b_3, a_7 b_4, a_7 b_6)$.

\item[(c')] If $\graph{B_{\Lambda}}$ is the bipartite complement of the Heawood graph, then $r=s=7$, $\Sigma_1=\Sigma_2=4$, and
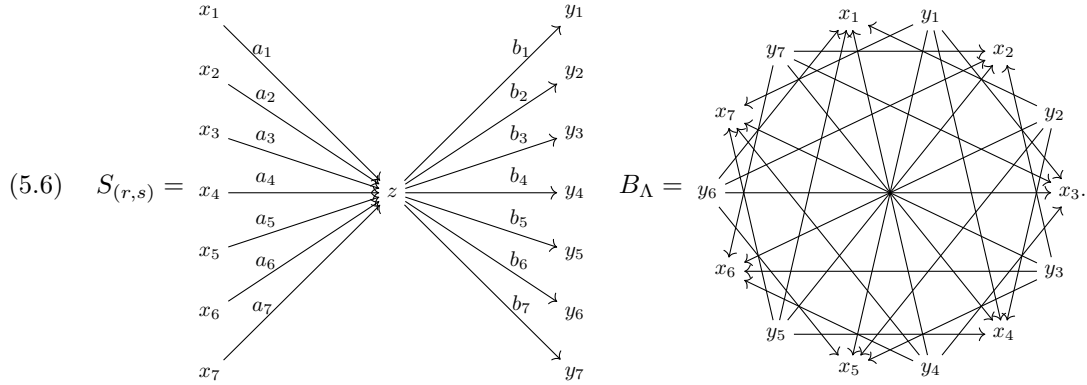
\begin{equation}\label{eq:the bipartite complement of the Heawood graph example}
    S_{(r,s)} = \begin{tikzpicture}[scale=0.8, transform shape, baseline={(current bounding box.center)}]
    \node (x1) at (-0.5,4) {$x_1$};
    \node (x2) at (-0.5,3) {$x_2$};
    \node (x3) at (-0.5,2) {$x_3$};
    \node (x4) at (-0.5,1) {$x_4$};
    \node (x5) at (-0.5,0) {$x_5$};
    \node (x6) at (-0.5,-1) {$x_6$};
    \node (x7) at (-0.5,-2) {$x_7$};
    \node (y1) at (5.5,4) {$y_1$};  
    \node (y2) at (5.5,3) {$y_2$};
    \node (y3) at (5.5,2) {$y_3$};
    \node (y4) at (5.5,1) {$y_4$};    
    \node (y5) at (5.5,0) {$y_5$};  
    \node (y6) at (5.5,-1) {$y_6$};
    \node (y7) at (5.5,-2) {$y_7$};
    \node (z) at (2.5,1) {$z$};

    \draw[->] (x1) to node[near start, above] {$a_1$}  (z);
    \draw[->] (x2) to node[near start, above] {$a_2$}  (z);
    \draw[->] (x3) to node[near start, above] {$a_3$}  (z);
    \draw[->] (x4) to node[near start, above] {$a_4$}  (z);
    \draw[->] (x5) to node[near start, above] {$a_5$}  (z);
    \draw[->] (x6) to node[near start, above] {$a_6$}  (z);
    \draw[->] (x7) to node[near start, above] {$a_7$}  (z);
    (z);    
    \draw[->] (z) to node[near end, above] {$b_1$}  (y1);
    \draw[->] (z) to node[near end, above] {$b_2$}  (y2);
    \draw[->] (z) to node[near end, above] {$b_3$}  (y3);
    \draw[->] (z) to node[near end, above] {$b_4$}  (y4);    
    \draw[->] (z) to node[near end, above] {$b_5$}  (y5);
    \draw[->] (z) to node[near end, above] {$b_6$}  (y6);
    \draw[->] (z) to node[near end, above] {$b_7$}  (y7);
\end{tikzpicture}\quad
B_{\Lambda}=\begin{tikzpicture}[scale=0.8, transform shape, baseline={(current bounding box.center)}]

\def\n{14}

\def\radius{3}

    \node (x3) at ({360/\n * (1 - 1)}:\radius) {$x_3$.};
    \node (y2) at ({360/\n * (2 - 1)}:\radius) {$y_2$};
    \node (x2) at ({360/\n * (3 - 1)}:\radius) {$x_2$};
    \node (y1) at ({360/\n * (4 - 1)}:\radius) {$y_1$};
    \node (x1) at ({360/\n * (5 - 1)}:\radius) {$x_1$};
    \node (y7) at ({360/\n * (6 - 1)}:\radius) {$y_7$};
    \node (x7) at ({360/\n * (7 - 1)}:\radius) {$x_7$};
    \node (y6) at ({360/\n * (8 - 1)}:\radius) {$y_6$};
    \node (x6) at ({360/\n * (9 - 1)}:\radius) {$x_6$};
    \node (y5) at ({360/\n * (10 - 1)}:\radius) {$y_5$};
    \node (x5) at ({360/\n * (11 - 1)}:\radius) {$x_5$};
    \node (y4) at ({360/\n * (12 - 1)}:\radius) {$y_4$};
    \node (x4) at ({360/\n * (13 - 1)}:\radius) {$x_4$};
    \node (y3) at ({360/\n * (14 - 1)}:\radius) {$y_3$};

    \draw[->] (y1) to (x3);
    \draw[->] (y1) to (x4);
    \draw[->] (y1) to (x5);
    \draw[->] (y1) to (x7);
    \draw[->] (y2) to (x1);
    \draw[->] (y2) to (x4);
    \draw[->] (y2) to (x5);
    \draw[->] (y2) to (x6);
    \draw[->] (y3) to (x2);
    \draw[->] (y3) to (x5);
    \draw[->] (y3) to (x6);
    \draw[->] (y3) to (x7);
    \draw[->] (y4) to (x1);
    \draw[->] (y4) to (x3);
    \draw[->] (y4) to (x6);
    \draw[->] (y4) to (x7);
    \draw[->] (y5) to (x1);
    \draw[->] (y5) to (x2);
    \draw[->] (y5) to (x4);
    \draw[->] (y5) to (x7);
    \draw[->] (y6) to (x1);
    \draw[->] (y6) to (x2);
    \draw[->] (y6) to (x3);
    \draw[->] (y6) to (x5);
    \draw[->] (y7) to (x2);
    \draw[->] (y7) to (x3);
    \draw[->] (y7) to (x4);
    \draw[->] (y7) to (x6);
\end{tikzpicture}
\end{equation}
That is, in this case $\Lambda=\K S_{(7,7)}/(a_1 b_1, a_1 b_2, a_1 b_6, a_2 b_2, a_2 b_3, a_2 b_7, a_3 b_1, a_3 b_3, a_3 b_4, a_4 b_2, a_4 b_4, \\ a_4 b_5, a_5 b_3, a_5 b_5, a_5 b_6, a_6 b_4, a_6 b_6, a_6 b_7, a_7 b_1, a_7 b_5, a_7 b_7)$.

\item[(d)] If $\graph{B_{\Lambda}}$ is the graph G-9P730, then $r=6$, $\Sigma_1=3$, $s=9$, $\Sigma_2=2$, and
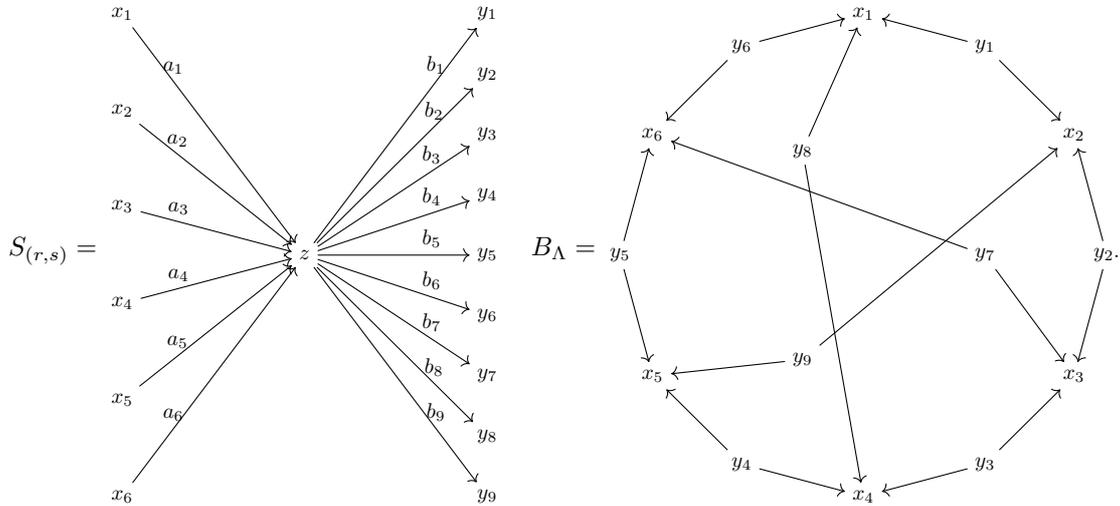
\begin{equation}\label{eq:the graph G-9P730 example}
    S_{(r,s)} = \begin{tikzpicture}[scale=0.8, transform shape, baseline={(current bounding box.center)}]
    \node (x1) at (-0.5,4) {$x_1$};
    \node (x2) at (-0.5,2.4) {$x_2$};
    \node (x3) at (-0.5,0.8) {$x_3$};
    \node (x4) at (-0.5,-0.8) {$x_4$};
    \node (x5) at (-0.5,-2.4) {$x_5$};
    \node (x6) at (-0.5,-4) {$x_6$};    
    \node (y1) at (5.5,4) {$y_1$};  
    \node (y2) at (5.5,3) {$y_2$};
    \node (y3) at (5.5,2) {$y_3$};
    \node (y4) at (5.5,1) {$y_4$};    
    \node (y5) at (5.5,0) {$y_5$};  
    \node (y6) at (5.5,-1) {$y_6$};
    \node (y7) at (5.5,-2) {$y_7$};
    \node (y8) at (5.5,-3) {$y_8$};  
    \node (y9) at (5.5,-4) {$y_9$};                    
    \node (z) at (2.5,0) {$z$};

    \draw[->] (x1) to node[near start, above] {$a_1$}  (z);
    \draw[->] (x2) to node[near start, above] {$a_2$}  (z);
    \draw[->] (x3) to node[near start, above] {$a_3$}  (z);
    \draw[->] (x4) to node[near start, above] {$a_4$}  (z);
    \draw[->] (x5) to node[near start, above] {$a_5$}  (z);
    \draw[->] (x6) to node[near start, above] {$a_6$}  (z);
    
    \draw[->] (z) to node[near end, above] {$b_1$}  (y1);
    \draw[->] (z) to node[near end, above] {$b_2$}  (y2);
    \draw[->] (z) to node[near end, above] {$b_3$}  (y3);
    \draw[->] (z) to node[near end, above] {$b_4$}  (y4);    
    \draw[->] (z) to node[near end, above] {$b_5$}  (y5);
    \draw[->] (z) to node[near end, above] {$b_6$}  (y6);
    \draw[->] (z) to node[near end, above] {$b_7$}  (y7);
    \draw[->] (z) to node[near end, above] {$b_8$}  (y8);
    \draw[->] (z) to node[near end, above] {$b_9$}  (y9);
\end{tikzpicture}\quad
B_{\Lambda}=\begin{tikzpicture}[scale=0.8, transform shape, baseline={(current bounding box.center)}]

\def\n{12}
\def\nb{3}

\def\radius{4}
\def\radiusb{2}

    \node (y2) at ({360/\n * (1 - 1)}:\radius) {$y_2$.};
    \node (x2) at ({360/\n * (2 - 1)}:\radius) {$x_2$};
    \node (y1) at ({360/\n * (3 - 1)}:\radius) {$y_1$};
    \node (x1) at ({360/\n * (4 - 1)}:\radius) {$x_1$};
    \node (y6) at ({360/\n * (5 - 1)}:\radius) {$y_6$};
    \node (x6) at ({360/\n * (6 - 1)}:\radius) {$x_6$};
    \node (y5) at ({360/\n * (7 - 1)}:\radius) {$y_5$};
    \node (x5) at ({360/\n * (8 - 1)}:\radius) {$x_5$};
    \node (y4) at ({360/\n * (9 - 1)}:\radius) {$y_4$};
    \node (x4) at ({360/\n * (10 - 1)}:\radius) {$x_4$};
    \node (y3) at ({360/\n * (11 - 1)}:\radius) {$y_3$};
    \node (x3) at ({360/\n * (12 - 1)}:\radius) {$x_3$};
    \node (y7) at ({360/\nb * (1-1)}:\radiusb) {$y_7$};
    \node (y8) at ({360/\nb * (2-1)}:\radiusb) {$y_8$};
    \node (y9) at ({360/\nb * (3-1)}:\radiusb) {$y_9$};

    \draw[->] (y8) to (x4);
    \draw[->] (y8) to (x1);
    \draw[->] (y7) to (x3);
    \draw[->] (y7) to (x6);
    \draw[->] (y9) to (x2);
    \draw[->] (y9) to (x5);

    \draw[->] (y1) to (x1);
    \draw[->] (y1) to (x2);
    \draw[->] (y2) to (x2);
    \draw[->] (y2) to (x3);
    \draw[->] (y3) to (x3);
    \draw[->] (y3) to (x4);
    \draw[->] (y4) to (x4);
    \draw[->] (y4) to (x5);
    \draw[->] (y5) to (x5);
    \draw[->] (y5) to (x6);
    \draw[->] (y6) to (x1);
    \draw[->] (y6) to (x6);
\end{tikzpicture}
\end{equation}
That is, in this case $\Lambda=\K S_{(6,9)}/(a_1 b_2, a_1 b_3, a_1 b_4, a_1 b_5, a_1 b_7, a_1 b_9, a_2 b_3, a_2 b_4, a_2 b_5, a_2 b_6, \\ a_2 b_7, a_2 b_8, a_3 b_1, a_3 b_4, a_3 b_5, a_3 b_6, a_3 b_8, a_3 b_9, a_4 b_1, a_4 b_2, a_4 b_5, a_4 b_6, a_4 b_7, a_4 b_9, a_5 b_1, a_5 b_2, a_5 b_3, \\ a_5 b_6, a_5 b_7, a_5 b_8, a_6 b_1, a_6 b_2, a_6 b_3, a_6 b_4, a_6 b_8, a_6 b_9)$.

\item[(d')] If $\graph{B_{\Lambda}}$ is the bipartite complement of the graph G-9P730, then $r=6$, $\Sigma_1=6$, $s=9$, $\Sigma_2=4$, and
\begin{equation}\label{eq:the bipartite complement of the graph G-9P730 example}
    S_{(r,s)} = \begin{tikzpicture}[scale=0.8, transform shape, baseline={(current bounding box.center)}]
    \node (x1) at (-0.5,4) {$x_1$};
    \node (x2) at (-0.5,2.4) {$x_2$};
    \node (x3) at (-0.5,0.8) {$x_3$};
    \node (x4) at (-0.5,-0.8) {$x_4$};
    \node (x5) at (-0.5,-2.4) {$x_5$};
    \node (x6) at (-0.5,-4) {$x_6$};    
    \node (y1) at (5.5,4) {$y_1$};  
    \node (y2) at (5.5,3) {$y_2$};
    \node (y3) at (5.5,2) {$y_3$};
    \node (y4) at (5.5,1) {$y_4$};    
    \node (y5) at (5.5,0) {$y_5$};  
    \node (y6) at (5.5,-1) {$y_6$};
    \node (y7) at (5.5,-2) {$y_7$};
    \node (y8) at (5.5,-3) {$y_8$};  
    \node (y9) at (5.5,-4) {$y_9$};                    
    \node (z) at (2.5,0) {$z$};

    \draw[->] (x1) to node[near start, above] {$a_1$}  (z);
    \draw[->] (x2) to node[near start, above] {$a_2$}  (z);
    \draw[->] (x3) to node[near start, above] {$a_3$}  (z);
    \draw[->] (x4) to node[near start, above] {$a_4$}  (z);
    \draw[->] (x5) to node[near start, above] {$a_5$}  (z);
    \draw[->] (x6) to node[near start, above] {$a_6$}  (z);
    
    \draw[->] (z) to node[near end, above] {$b_1$}  (y1);
    \draw[->] (z) to node[near end, above] {$b_2$}  (y2);
    \draw[->] (z) to node[near end, above] {$b_3$}  (y3);
    \draw[->] (z) to node[near end, above] {$b_4$}  (y4);    
    \draw[->] (z) to node[near end, above] {$b_5$}  (y5);
    \draw[->] (z) to node[near end, above] {$b_6$}  (y6);
    \draw[->] (z) to node[near end, above] {$b_7$}  (y7);
    \draw[->] (z) to node[near end, above] {$b_8$}  (y8);
    \draw[->] (z) to node[near end, above] {$b_9$}  (y9);
\end{tikzpicture}\quad
B_{\Lambda}=\begin{tikzpicture}[scale=0.8, transform shape, baseline={(current bounding box.center)}]

\def\n{12}
\def\nb{3}

\def\radius{5}
\def\radiusb{3}

    \node (y2) at ({360/\n * (1 - 1)}:\radius) {$y_2$.};
    \node (x2) at ({360/\n * (2 - 1)}:\radius) {$x_2$};
    \node (y1) at ({360/\n * (3 - 1)}:\radius) {$y_1$};
    \node (x1) at ({360/\n * (4 - 1)}:\radius) {$x_1$};
    \node (y6) at ({360/\n * (5 - 1)}:\radius) {$y_6$};
    \node (x6) at ({360/\n * (6 - 1)}:\radius) {$x_6$};
    \node (y5) at ({360/\n * (7 - 1)}:\radius) {$y_5$};
    \node (x5) at ({360/\n * (8 - 1)}:\radius) {$x_5$};
    \node (y4) at ({360/\n * (9 - 1)}:\radius) {$y_4$};
    \node (x4) at ({360/\n * (10 - 1)}:\radius) {$x_4$};
    \node (y3) at ({360/\n * (11 - 1)}:\radius) {$y_3$};
    \node (x3) at ({360/\n * (12 - 1)}:\radius) {$x_3$};
    \node (y7) at ({360/\nb * (1-1)}:\radiusb) {$y_7$};
    \node (y8) at ({360/\nb * (2-1)}:\radiusb) {$y_8$};
    \node (y9) at ({360/\nb * (3-1)}:\radiusb) {$y_9$};

    \draw[->] (y8) to (x2);
    \draw[->] (y8) to (x3);
    \draw[->] (y8) to (x5);
    \draw[->] (y8) to (x6);
    \draw[->] (y7) to (x1);
    \draw[->] (y7) to (x2);
    \draw[->] (y7) to (x4);
    \draw[->] (y7) to (x5);
    \draw[->] (y9) to (x1);
    \draw[->] (y9) to (x3);
    \draw[->] (y9) to (x4);
    \draw[->] (y9) to (x6);
    \draw[->] (y1) to (x3);
    \draw[->] (y1) to (x4);
    \draw[->] (y1) to (x5);
    \draw[->] (y1) to (x6);
    \draw[->] (y2) to (x1);
    \draw[->] (y2) to (x4);
    \draw[->] (y2) to (x5);
    \draw[->] (y2) to (x6);
    \draw[->] (y3) to (x1);
    \draw[->] (y3) to (x2);
    \draw[->] (y3) to (x5);
    \draw[->] (y3) to (x6);
    \draw[->] (y4) to (x1);
    \draw[->] (y4) to (x2);
    \draw[->] (y4) to (x3);
    \draw[->] (y4) to (x6);
    \draw[->] (y5) to (x1);
    \draw[->] (y5) to (x2);
    \draw[->] (y5) to (x3);
    \draw[->] (y5) to (x4);
    \draw[->] (y6) to (x2);
    \draw[->] (y6) to (x3);
    \draw[->] (y6) to (x4);
    \draw[->] (y6) to (x5);
\end{tikzpicture}
\end{equation}
That is, in this case $\Lambda=\K S_{(6,9)}/(a_1 b_1, a_1 b_6, a_1 b_8, a_2 b_1, a_2 b_2, a_2 b_9, a_3 b_2, a_3 b_3, a_3 b_7, a_4 b_3, a_4 b_4, \\ a_4 b_8, a_5 b_4, a_5 b_5, a_5 b_9, a_6 b_5, a_6 b_6, a_6 b_7)$.

\item[(e)] If $\graph{B_{\Lambda}}$ is the graph G-9P731, then $r=5$, $\Sigma_1=4$, $s=10$, $\Sigma_2=2$, and
\begin{equation}\label{eq:the graph G-9P731 example}
    S_{(r,s)} = \begin{tikzpicture}[scale=0.8, transform shape, baseline={(current bounding box.center)}]
    \node (x1) at (-0.5,4) {$x_1$};
    \node (x2) at (-0.5,1.75) {$x_2$};
    \node (x3) at (-0.5,-0.5) {$x_3$};
    \node (x4) at (-0.5,-2.75) {$x_4$};
    \node (x5) at (-0.5,-5) {$x_5$};
    \node (y1) at (5.5,4) {$y_1$};  
    \node (y2) at (5.5,3) {$y_2$};
    \node (y3) at (5.5,2) {$y_3$};
    \node (y4) at (5.5,1) {$y_4$};    
    \node (y5) at (5.5,0) {$y_5$};  
    \node (y6) at (5.5,-1) {$y_6$};
    \node (y7) at (5.5,-2) {$y_7$};
    \node (y8) at (5.5,-3) {$y_8$};  
    \node (y9) at (5.5,-4) {$y_9$};  
    \node (y10) at (5.5,-5) {$y_{10}$};                    
    \node (z) at (2.5,-0.5) {$z$};

    \draw[->] (x1) to node[near start, above] {$a_1$}  (z);
    \draw[->] (x2) to node[near start, above] {$a_2$}  (z);
    \draw[->] (x3) to node[near start, above] {$a_3$}  (z);
    \draw[->] (x4) to node[near start, above] {$a_4$}  (z);
    \draw[->] (x5) to node[near start, above] {$a_5$}  (z);
    
    \draw[->] (z) to node[near end, above] {$b_1$}  (y1);
    \draw[->] (z) to node[near end, above] {$b_2$}  (y2);
    \draw[->] (z) to node[near end, above] {$b_3$}  (y3);
    \draw[->] (z) to node[near end, above] {$b_4$}  (y4);    
    \draw[->] (z) to node[near end, above] {$b_5$}  (y5);
    \draw[->] (z) to node[near end, above] {$b_6$}  (y6);
    \draw[->] (z) to node[near end, above] {$b_7$}  (y7);
    \draw[->] (z) to node[near end, above] {$b_8$}  (y8);
    \draw[->] (z) to node[near end, above] {$b_9$}  (y9);
    \draw[->] (z) to node[near end, above] {$b_{10}$}  (y10);
\end{tikzpicture}\quad
B_{\Lambda}=\begin{tikzpicture}[scale=0.8, transform shape, baseline={(current bounding box.center)}]

\def\n{10}
\def\nb{5}

\def\radius{4}
\def\radiusb{2}

    \node (y2) at ({360/\n * (1 - 1)}:\radius) {$y_2$.};
    \node (x2) at ({360/\n * (2 - 1)}:\radius) {$x_2$};
    \node (y1) at ({360/\n * (3 - 1)}:\radius) {$y_1$};
    \node (x1) at ({360/\n * (4 - 1)}:\radius) {$x_1$};
    \node (y5) at ({360/\n * (5 - 1)}:\radius) {$y_5$};
    \node (x5) at ({360/\n * (6 - 1)}:\radius) {$x_5$};
    \node (y4) at ({360/\n * (7 - 1)}:\radius) {$y_4$};
    \node (x4) at ({360/\n * (8 - 1)}:\radius) {$x_4$};
    \node (y3) at ({360/\n * (9 - 1)}:\radius) {$y_3$};
    \node (x3) at ({360/\n * (10 - 1)}:\radius) {$x_3$};
    
    \node (y6) at ({360/\nb * (1-1)}:\radiusb) {$y_6$};
    \node (y7) at ({360/\nb * (2-1)}:\radiusb) {$y_7$};
    \node (y8) at ({360/\nb * (3-1)}:\radiusb) {$y_8$};
    \node (y9) at ({360/\nb * (4-1)}:\radiusb) {$y_9$};
    \node (y10) at ({360/\nb * (5-1)}:\radiusb) {$y_{10}$};

    \draw[->] (y6) to (x1);
    \draw[->] (y6) to (x3);
    \draw[->] (y7) to (x2);
    \draw[->] (y7) to (x4);
    \draw[->] (y8) to (x3);
    \draw[->] (y8) to (x5);
    \draw[->] (y9) to (x4);
    \draw[->] (y9) to (x1);
    \draw[->] (y10) to (x5);
    \draw[->] (y10) to (x2);

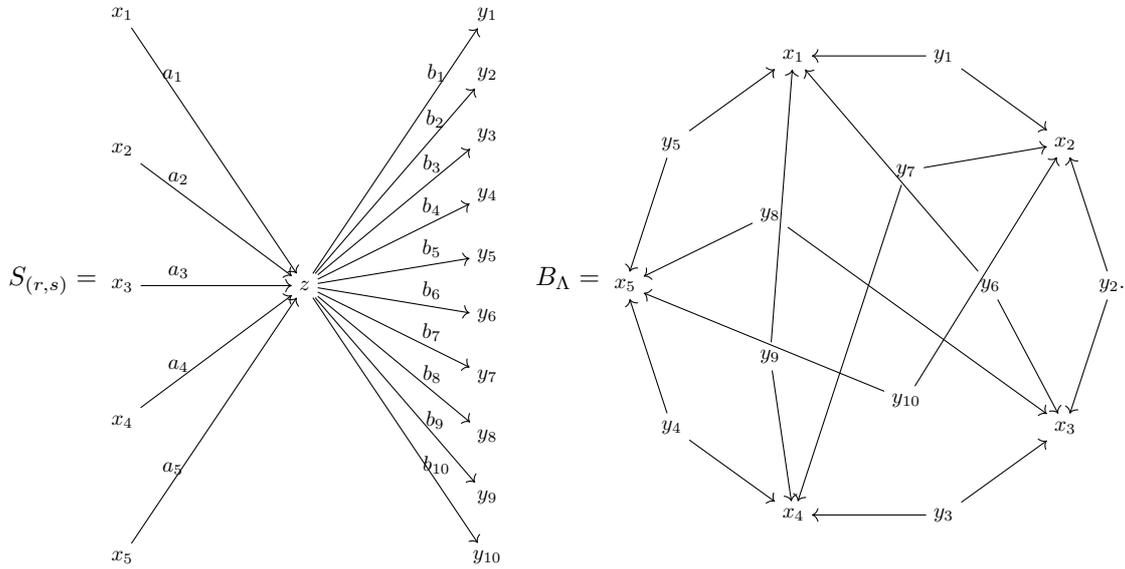
    \draw[->] (y1) to (x1);
    \draw[->] (y1) to (x2);
    \draw[->] (y2) to (x2);
    \draw[->] (y2) to (x3);
    \draw[->] (y3) to (x3);
    \draw[->] (y3) to (x4);
    \draw[->] (y4) to (x4);
    \draw[->] (y4) to (x5);
    \draw[->] (y5) to (x5);
    \draw[->] (y5) to (x1);
\end{tikzpicture}
\end{equation}
That is, in this case $\Lambda=\K S_{(6,9)}/(a_1 b_2, a_1 b_3, a_1 b_4, a_1 b_7, a_1 b_8, a_1 b_{10}, a_2 b_3, a_2 b_4, a_2 b_5, a_2 b_6, \\ a_2 b_8, a_2 b_9, a_3 b_1, a_3 b_4, a_3 b_5, a_3 b_7, a_3 b_9, a_3 b_{10}, a_4 b_1, a_4 b_2, a_4 b_5, a_4 b_6, a_4 b_8, a_4 b_{10}, a_5 b_1, a_5 b_2,\\ a_5 b_3, a_5 b_6, a_5 b_7, a_5 b_9)$.

\item[(e')] If $\graph{B_{\Lambda}}$ is the bipartite complement of the graph G-9P731, then $r=5$, $\Sigma_1=6$, $s=10$, $\Sigma_2=3$, and
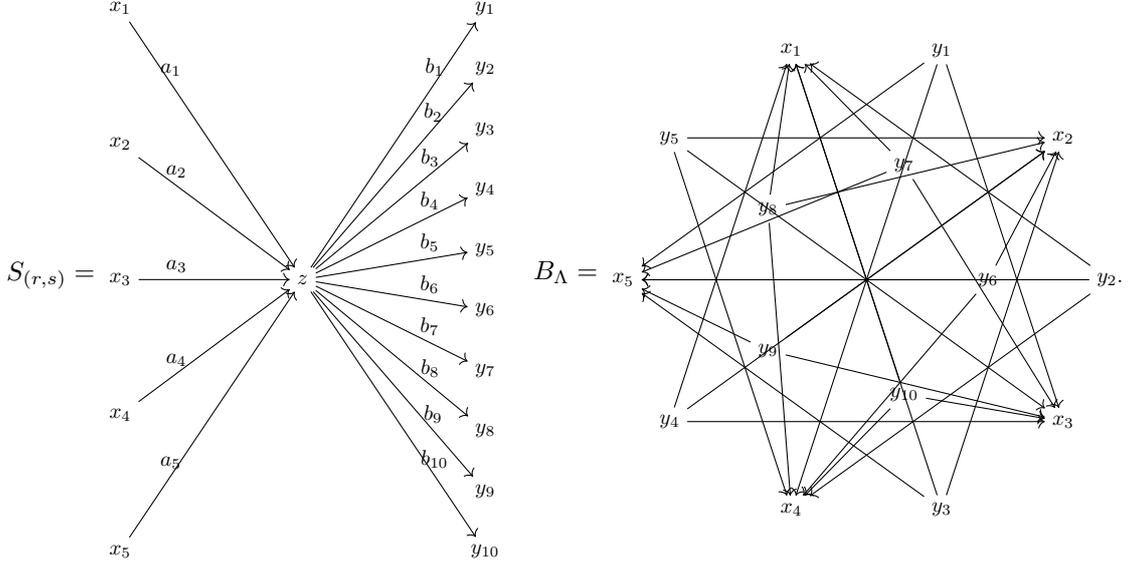
\begin{equation}\label{eq:the bipartite complement of the graph G-9P731 example}
    S_{(r,s)} = \begin{tikzpicture}[scale=0.8, transform shape, baseline={(current bounding box.center)}]
    \node (x1) at (-0.5,4) {$x_1$};
    \node (x2) at (-0.5,1.75) {$x_2$};
    \node (x3) at (-0.5,-0.5) {$x_3$};
    \node (x4) at (-0.5,-2.75) {$x_4$};
    \node (x5) at (-0.5,-5) {$x_5$};
    \node (y1) at (5.5,4) {$y_1$};  
    \node (y2) at (5.5,3) {$y_2$};
    \node (y3) at (5.5,2) {$y_3$};
    \node (y4) at (5.5,1) {$y_4$};    
    \node (y5) at (5.5,0) {$y_5$};  
    \node (y6) at (5.5,-1) {$y_6$};
    \node (y7) at (5.5,-2) {$y_7$};
    \node (y8) at (5.5,-3) {$y_8$};  
    \node (y9) at (5.5,-4) {$y_9$};  
    \node (y10) at (5.5,-5) {$y_{10}$};                    
    \node (z) at (2.5,-0.5) {$z$};

    \draw[->] (x1) to node[near start, above] {$a_1$}  (z);
    \draw[->] (x2) to node[near start, above] {$a_2$}  (z);
    \draw[->] (x3) to node[near start, above] {$a_3$}  (z);
    \draw[->] (x4) to node[near start, above] {$a_4$}  (z);
    \draw[->] (x5) to node[near start, above] {$a_5$}  (z);
    
    \draw[->] (z) to node[near end, above] {$b_1$}  (y1);
    \draw[->] (z) to node[near end, above] {$b_2$}  (y2);
    \draw[->] (z) to node[near end, above] {$b_3$}  (y3);
    \draw[->] (z) to node[near end, above] {$b_4$}  (y4);    
    \draw[->] (z) to node[near end, above] {$b_5$}  (y5);
    \draw[->] (z) to node[near end, above] {$b_6$}  (y6);
    \draw[->] (z) to node[near end, above] {$b_7$}  (y7);
    \draw[->] (z) to node[near end, above] {$b_8$}  (y8);
    \draw[->] (z) to node[near end, above] {$b_9$}  (y9);
    \draw[->] (z) to node[near end, above] {$b_{10}$}  (y10);
\end{tikzpicture}\quad
B_{\Lambda}=\begin{tikzpicture}[scale=0.8, transform shape, baseline={(current bounding box.center)}]

\def\n{10}
\def\nb{5}

\def\radius{4}
\def\radiusb{2}

    \node (y2) at ({360/\n * (1 - 1)}:\radius) {$y_2$.};
    \node (x2) at ({360/\n * (2 - 1)}:\radius) {$x_2$};
    \node (y1) at ({360/\n * (3 - 1)}:\radius) {$y_1$};
    \node (x1) at ({360/\n * (4 - 1)}:\radius) {$x_1$};
    \node (y5) at ({360/\n * (5 - 1)}:\radius) {$y_5$};
    \node (x5) at ({360/\n * (6 - 1)}:\radius) {$x_5$};
    \node (y4) at ({360/\n * (7 - 1)}:\radius) {$y_4$};
    \node (x4) at ({360/\n * (8 - 1)}:\radius) {$x_4$};
    \node (y3) at ({360/\n * (9 - 1)}:\radius) {$y_3$};
    \node (x3) at ({360/\n * (10 - 1)}:\radius) {$x_3$};
    
    \node (y6) at ({360/\nb * (1-1)}:\radiusb) {$y_6$};
    \node (y7) at ({360/\nb * (2-1)}:\radiusb) {$y_7$};
    \node (y8) at ({360/\nb * (3-1)}:\radiusb) {$y_8$};
    \node (y9) at ({360/\nb * (4-1)}:\radiusb) {$y_9$};
    \node (y10) at ({360/\nb * (5-1)}:\radiusb) {$y_{10}$};

    \draw[->] (y1) to (x3);
    \draw[->] (y1) to (x4);
    \draw[->] (y1) to (x5);
    \draw[->] (y2) to (x1);
    \draw[->] (y2) to (x4);
    \draw[->] (y2) to (x5);
    \draw[->] (y3) to (x1);
    \draw[->] (y3) to (x2);
    \draw[->] (y3) to (x5);
    \draw[->] (y4) to (x1);
    \draw[->] (y4) to (x2);
    \draw[->] (y4) to (x3);
    \draw[->] (y5) to (x2);
    \draw[->] (y5) to (x3);
    \draw[->] (y5) to (x4);
    \draw[->] (y6) to (x2);
    \draw[->] (y6) to (x4);
    \draw[->] (y6) to (x5);
    \draw[->] (y7) to (x1);
    \draw[->] (y7) to (x3);
    \draw[->] (y7) to (x5);
    \draw[->] (y8) to (x1);
    \draw[->] (y8) to (x2);
    \draw[->] (y8) to (x4);
    \draw[->] (y9) to (x2);
    \draw[->] (y9) to (x3);
    \draw[->] (y9) to (x5);
    \draw[->] (y10) to (x1);
    \draw[->] (y10) to (x3);
    \draw[->] (y10) to (x4);
\end{tikzpicture}
\end{equation}
That is, in this case $\Lambda=\K S_{(6,9)}/(a_1 b_1, a_1 b_5, a_1 b_6, a_1 b_9, a_2 b_1, a_2 b_2, a_2 b_7, a_2 b_{10}, a_3 b_2, \\ a_3 b_3, a_3 b_6, a_3 b_8, a_4 b_3, a_4 b_4, a_4 b_7, a_4 b_9, a_5 b_4, a_5 b_5, a_5 b_8, a_5 b_{10})$.
\end{enumerate}

In the following example, we indicate one way of obtaining the graph G-9P731, which has connections to incidence geometry.

\begin{example}
The graph G-9P731 corresponding to the solution $(5,4,10,2)$ can be constructed by taking its set of vertices to be $V = X \cup Y$ with $X$ the set $\{1,2,3,4,5\}$ and as $Y$ the set of two element subsets of $X$. One then lets there be an edge between $x \in X$ and $y \in Y$ if and only if $x \in y$. 
\end{example}

\appendix

\section{Solutions to the system of linear Diophantine equations}
\label{Section:Solutions to the system of linear Diophantine equations}
The aim of this appendix is to find integers $\Sigma_1$, $\Sigma_2$, $\lvert d_x\rvert^2$ and $\lvert d_y\rvert^2$ that satisfy the conditions of Proposition \ref{prop:properties when one-point extension is fractionally Calabi-Yau}. The results in this appendix are based on the answer of Max Alekseyev in a Mathoverflow post \cite{Ale2022}.

To simplify the notation, we consider the system of equations 
\begin{align}
  kl + a + b - ak &= p  \tag{9.1} \label{eq:1}\\
    ak &= bl \tag{9.2}\label{eq:2}
\end{align} 
where $ p\in\{0,1,2,3,4\}$ and $a,b,k,l\in\ZZ_{>0}$. Denote  
\[
\mathbf{S} \coloneqq \{(p,a,k,b,l)\in\ZZ^{5}_{\geq 1} \mid (p,a,k,b,l) \text{ satisfy (\ref{eq:1}) and (\ref{eq:2})}\}.
\]
Then an element $(p,a,k,b,l)\in\mathbf{S}$ gives integers satisfying Proposition \ref{prop:properties when one-point extension is fractionally Calabi-Yau} by setting $\Sigma_1\coloneqq k$, $\Sigma_2\coloneqq l$, $\lvert d_x\rvert^2\coloneqq a$, $\lvert d_y\rvert^2\coloneqq b$.

If $(p,x,y,z,w)\in\ZZ^5$, then we define $(p,x,y,z,w)^{\ast}\coloneqq (p,z,w,x,y)$. In particular, we obtain that $(p,a,k,b,l)\in\mathbf{S}$ if and only if $(p,a,k,b,l)^{\ast}\in\mathbf{S}$. We now proceed with finding all elements of $\mathbf{S}$.

First, using (\ref{eq:2}), there exist integers $u,v,w,t\in\ZZ_{>0}$ such that
\begin{equation}\tag{10} \label{eq:change of variables}
a = uv,\;\; k=wt,\;\; b=uw,\;\; l=vt.
\end{equation}
For example, we may set $u=\gcd(a,b)$, so that $a=uv$ and $b=uw$ and that $v$ divides $l$. Then (\ref{eq:1}) becomes
\begin{equation}
    vwt^2+uv+uw-uvwt=p \tag{9.3} \label{eq:3}
\end{equation}
We consider the cases $u\leq t$ and $u>t$ separately.

Case $u\leq t$. In this case (\ref{eq:3}) gives
\[
uv+uw = p+uvwt-vwt^2 \leq p+vwt^2-vwt^2=p.
\]
Hence 
\begin{equation}
    uv +vw \leq p \tag{9.4} \label{eq:4}
\end{equation}
In particular, since 
\[
2=1+1 \leq uv+vw\leq p\leq 4,
\]
we have that $1\leq u,v,w\leq 3$. Hence there are finitely many solutions of the system (\ref{eq:3}) and (\ref{eq:4}), and a check of the different possibilities gives that the only solutions in this case are:
\[
\begin{array}{llllll}
    p=2,& u=1,& v=1,& w=1,& t=1, \\
    p=3,& u=1,& v=1,& w=2,& t=1, \\
    p=3,& u=1,& v=2,& w=1,& t=1, \\
    p=4,& u=1,& v=1,& w=1,& t=2, \\
    p=4,& u=1,& v=1,& w=3,& t=1, \\
    p=4,& u=1,& v=3,& w=1,& t=1, \\
    p=4,& u=1,& v=2,& w=2,& t=1, \\
    p=4,& u=2,& v=1,& w=1,& t=2.
\end{array}
\]
Using (\ref{eq:change of variables}), we obtain the following solutions of the system (\ref{eq:1})--(\ref{eq:2}):
\[
\begin{array}{llllll}
    p=2,& a=1,& k=1,& b=1,& l=1, \\
    p=3,& a=1,& k=2,& b=2,& l=1, \\
    p=3,& a=2,& k=1,& b=1,& l=2, \\
    p=4,& a=1,& k=2,& b=1,& l=2, \\
    p=4,& a=1,& k=3,& b=3,& l=1, \\
    p=4,& a=3,& k=1,& b=1,& l=3, \\
    p=4,& a=2,& k=2,& b=2,& l=2. \\
\end{array}
\]
Hence by setting
\[
S_1 \coloneqq \{(2,1,1,1,1),(3,1,2,2,1),(4,1,2,1,2),(4,1,3,3,1),(4,2,2,2,2)\}
\]
we obtain $S_1\cup S_1^{\ast}\subseteq \mathbf{S}$.

Case $u>t$. Set $d\coloneqq u-t>0$. Replacing $u=d+t$ in (\ref{eq:3}) we obtain
\[
vwt^2 + (d+t)v + (d+t)w - (d+t)vwt = p.
\]
We then have
\begin{align*}
    vwt^2 + (d+t)v + (d+t)w - (d+t)vwt = p &\iff vwt^2 + dv + tv + dw + tw - dvwt - vwt^2 = p \\
    &\iff dv + tv + dw + tw - dvwt = p \\
    &\iff 2dvwt - 2dv -2tv -2dw -2tw + 8 = 8-2p,     
\end{align*}
from which we equivalently obtain
\begin{equation}
    (tv-2)(dw-2) + (dv-2)(tw-2) = 8-2p. \tag{9.5} \label{eq:5}
\end{equation}
In particular, since $0\leq p\leq 4$, we have that $8-2p\in\{0,2,4,6,8\}$. We consider two subcases.

Subcase $tv\geq 3$ and $dw\geq 3$ and $dv\geq 3$ and $tw\geq 3$. Then $(dv-2)(tw-2)\geq 1$ and so (\ref{eq:5}) gives
\[
(tv-2)(dw-2) \leq 8-2p-1 \leq 7.
\]
Hence $tv-2\leq 7$ and $dw-2\leq 7$, which gives $1\leq t,v,d,w\leq 9$. Hence in this case the equation (\ref{eq:5}) has finitely many solutions, which we can obtain by checking the different possibilities:
\[
\begin{array}{lllll}
    p=0,& d=1,& v=3,& w=3,& t=2, \\
    p=0,& d=1,& v=3,& w=6,& t=1, \\
    p=0,& d=1,& v=4,& w=4,& t=1, \\
    p=0,& d=1,& v=6,& w=3,& t=1, \\
    p=0,& d=2,& v=2,& w=2,& t=2, \\
    p=0,& d=2,& v=3,& w=3,& t=1, \\
    p=0,& d=3,& v=1,& w=1,& t=6, \\
    p=0,& d=3,& v=1,& w=2,& t=3, \\
    p=0,& d=3,& v=2,& w=1,& t=3, \\
    p=0,& d=4,& v=1,& w=1,& t=4, \\
    p=0,& d=6,& v=1,& w=1,& t=3, \\
    p=1,& d=1,& v=3,& w=5,& t=1, \\
    p=1,& d=1,& v=5,& w=3,& t=1, \\
    p=1,& d=3,& v=1,& w=1,& t=5, \\
    p=1,& d=5,& v=1,& w=1,& t=3, \\
    p=2,& d=1,& v=3,& w=4,& t=1, \\
    p=2,& d=1,& v=4,& w=3,& t=1, \\
    p=2,& d=3,& v=1,& w=1,& t=4, \\
    p=2,& d=4,& v=1,& w=1,& t=3, \\
    p=3,& d=1,& v=3,& w=3,& t=1, \\
    p=3,& d=3,& v=1,& w=1,& t=3.
\end{array}
\]
Using $d=u-t$ and (\ref{eq:change of variables}), we obtain the following solutions of the system (\ref{eq:1})--(\ref{eq:2}):
\[
\begin{array}{lllll}
    p=0,& a=6,& k=6,& b=12,& l=3, \\
    p=0,& a=8,& k=4,& b=8,& l=4, \\
    p=0,& a=9,& k=3,& b=9,& l=3, \\
    p=0,& a=9,& k=6,& b=9,& l=6, \\
    p=0,& a=12,& k=3,& b=6,& l=6, \\
    p=1,& a=6,& k=5,& b=10,& l=3, \\
    p=1,& a=8,& k=3,& b=8,& l=3, \\
    p=1,& a=8,& k=5,& b=8,& l=5, \\
    p=1,& a=10,& k=3,& b=6,& l=5, \\
    p=2,& a=6,& k=4,& b=8,& l=3, \\
    p=2,& a=7,& k=3,& b=7,& l=3, \\
    p=2,& a=7,& k=4,& b=7,& l=4, \\
    p=2,& a=8,& k=3,& b=6,& l=4, \\
    p=3,& a=6,& k=3,& b=6,& l=3. 
\end{array}
\]
Hence by setting
\begin{align*}
S_2 \coloneqq \{&(0,6,6,12,3),(0,8,4,8,4),(0,9,3,9,3),(0,9,6,9,6),(0,12,3,6,6),(1,6,5,10,3), \\
&(1,8,3,8,3),(1,8,5,8,5),(2,6,4,8,3),(2,7,3,7,3),(2,7,4,7,4),(3,6,3,6,3)\}
\end{align*}
we obtain $S_2\cup S_2^{\ast}\subseteq \mathbf{S}$.

Subcase $tv\leq 2$ or $dw\leq 2$ or $dv\leq 2$ or $tw\leq 2$. Let us consider the case $tv\leq 2$ first. Then either $(t,v)=(1,1)$ or $(t,v)=(1,2)$ or $(t,v)=(2,1)$. Hence (\ref{eq:5}) gives
\begin{align*}
    &(t,v)=(1,1)\;\; \text{ and } \;\;(1-2)(dw-2)+(d-2)(w-2) = 8-2p,\;\; \text{ or } \\
    &(t,v)=(1,2)\;\; \text{ and } \;\;(2d-2)(w-2) = 8-2p,\;\; \text{ or } \\
    &(t,v)=(2,1)\;\; \text{ and } \;\;(d-2)(2w-2) = 8-2p.
\end{align*}
Equivalently, we obtain
\begin{align}
    &(t,v)=(1,1)\;\; \text{ and } \;\;d+w = p-1,\;\; \text{ or } \tag{9.6}\label{eq:tv=11} \\
    &(t,v)=(1,2)\;\; \text{ and } \;\;(d-1)(w-2) = 4-p,\;\; \text{ or } \tag{9.7}\label{eq:tv=12} \\
    &(t,v)=(2,1)\;\; \text{ and } \;\;(d-2)(w-1) = 4-p. \tag{9.8}\label{eq:tv=21}
\end{align}
In (\ref{eq:tv=11}) we have $d+w=p-1\leq 3$, and so $1\leq d,w\leq 2$. Hence we obtain the solutions
\[
\begin{array}{lllll}
    p=3,& d=1,& v=1,& w=1,& t=1, \\
    p=4,& d=1,& v=1,& w=2,& t=1, \\
    p=4,& d=2,& v=1,& w=1,& t=1.
\end{array}
\]
In (\ref{eq:tv=12}), if $p<4$, then $(d-1)(w-2)\in\{1,2,3,4\}$. Hence we obtain the solutions
\[
\begin{array}{lllll}
    p=0,& d=2,& v=2,& w=6,& t=1, \\
    p=0,& d=3,& v=2,& w=4,& t=1, \\
    p=0,& d=5,& v=2,& w=3,& t=1,
    \\
    p=1,& d=2,& v=2,& w=5,& t=1, \\
    p=1,& d=4,& v=2,& w=3,& t=1, \\
    p=2,& d=2,& v=2,& w=4,& t=1, \\
    p=2,& d=3,& v=2,& w=3,& t=1, \\
    p=3,& d=2,& v=2,& w=3,& t=1.
\end{array}
\]
In (\ref{eq:tv=12}), if $p=4$, then $(d-1)(w-2)=0$. Hence we obtain the solutions
\[
\begin{array}{lllll}
    p=4,& d=1,& v=2,& w\geq 1,& t=1, \\
    p=4,& d\geq 1,& v=2,& w=2,& t=1.
\end{array}
\]
Finally, (\ref{eq:tv=21}) is symmetric to (\ref{eq:tv=12}). Hence in this case we obtain the solutions
\[
\begin{array}{llllll}
    p=0, &d=3, &v=1, &w=5, &t=2, \\
    p=0, &d=4, &v=1, &w=3, &t=2, \\
    p=0, &d=6, &v=1, &w=2, &t=2, \\
    p=1, &d=3, &v=1, &w=4, &t=2, \\
    p=1, &d=5, &v=1, &w=2, &t=2, \\
    p=2, &d=3, &v=1, &w=3, &t=2, \\
    p=2, &d=4, &v=1, &w=2, &t=2, \\
    p=3, &d=3, &v=1, &w=2, &t=2, \\
    p=4, &d=2, &v=1, &w\geq 1, &t=2, \\
    p=4, &d\geq 1, &v=1, &w=1, &t=2.
\end{array}
\]
Now using $d=u-t$ and (\ref{eq:change of variables}), we may translate all of the solutions in the case $tv\leq 2$ to the following solutions of the system (\ref{eq:1})--(\ref{eq:2}):
\[
\begin{array}{lllll}
    p=0,& a=5,& k=10,& b=25,& l=2, \\
    p=0,& a=6,& k=6,& b=18,& l=2, \\
    p=0,& a=8,& k=4,& b=16,& l=2, \\
    p=0,& a=12,& k=3,& b=18,& l=2, \\
    p=1,& a=5,& k=8,& b=20,& l=2, \\
    p=1,& a=6,& k=5,& b=15,& l=2, \\
    p=1,& a=7,& k=4,& b=14,& l=2, \\
    p=1,& a=10,& k=3,& b=15,& l=2, \\
    p=2,& a=5,& k=6,& b=15,& l=2, \\
    p=2,& a=6,& k=4,& b=12,& l=2, \\
    p=2,& a=8,& k=3,& b=12,& l=2, \\
    p=3,& a=2,& k=1,& b=2,& l=1, \\
    p=3,& a=5,& k=4,& b=10,& l=2, \\
    p=3,& a=6,& k=3,& b=9,& l=2, \\    
    p=4,& a=2,& k=2,& b=4,& l=1, \\
    p=4,& a=3,& k=1,& b=3,& l=1, \\
    p=4,& a=4,& k=w,& b=2w,& l=2, \\
    p=4,& a=d+2,& k=2,& b=d+2,& l=2.
\end{array}
\]
Hence by setting
\begin{align*}
    S_3\coloneqq \{&(0,5,10,25,2),(0,6,6,18,2),(0,8,4,16,2),(0,12,3,18,2),(1,5,8,20,2),(1,6,5,15,2),\\
    &(1,7,4,14,2),(1,10,3,15,2),(2,5,6,15,2),(2,6,4,12,2),(2,8,3,12,2),(3,2,1,2,1),\\
    &(3,5,4,10,2),(3,6,3,9,2),(4,2,2,4,1),(4,3,1,3,1),(4,4,x,2x,2)_{x\geq 1},(4,x+2,2,x+2,2)_{x\geq 1}\}
\end{align*}
we obtain $S_3\cup S_3^{\ast}\subseteq \mathbf{S}$. 

Next notice that by symmetry, the solutions in the case $tv\leq 2$ are given by the set $S_3^{\ast}$. 

Let us now consider the case $dw\leq 2$. Then either $(d,w)=(1,1)$ or $(d,w)=(1,2)$ or $(d,w)=(2,1)$.  Hence (\ref{eq:5}) gives
\begin{align*}
    &(d,w)=(1,1)\;\; \text{ and } \;\;(tv-2)(1-2)+(v-2)(t-2) = 8-2p,\;\; \text{ or } \\
    &(d,w)=(1,2)\;\; \text{ and } \;\;(v-2)(2t-2) = 8-2p,\;\; \text{ or } \\
    &(d,w)=(2,1)\;\; \text{ and } \;\;(2v-2)(t-2) = 8-2p.
\end{align*}
Equivalently, we obtain
\begin{align}
    &(d,w)=(1,1)\;\; \text{ and } \;\;v+t = p-1,\;\; \text{ or } \tag{9.9}\label{eq:dw=11} \\
    &(d,w)=(1,2)\;\; \text{ and } \;\;(v-2)(t-1) = 4-p,\;\; \text{ or } \tag{9.10}\label{eq:dw=12} \\
    &(d,w)=(2,1)\;\; \text{ and } \;\;(v-1)(t-2) = 4-p. \tag{9.11}\label{eq:dw=21}
\end{align}
In (\ref{eq:dw=11}) we have $v+t=p-1\leq 3$ and so $1\leq v,t\leq 2$. Hence we obtain the solutions
\[
\begin{array}{lllll}
    p=3,& d=1,& v=1,& w=1,& t=1, \\
    p=4,& d=1,& v=1,& w=1,& t=2, \\
    p=4,& d=1,& v=2,& w=1,& t=1.
\end{array}
\]
In (\ref{eq:dw=12}), if $p<4$, then $(v-2)(t-1)\in\{1,2,3,4\}$. Hence we obtain the solutions
\[
\begin{array}{lllll}
    p=0,& d=1,& v=3,& w=2,& t=5, \\
    p=0,& d=1,& v=4,& w=2,& t=3, \\
    p=0,& d=1,& v=6,& w=2,& t=2,
    \\
    p=1,& d=1,& v=3,& w=2,& t=4, \\
    p=1,& d=1,& v=5,& w=2,& t=2, \\
    p=2,& d=1,& v=3,& w=2,& t=3, \\
    p=2,& d=1,& v=4,& w=2,& t=2, \\
    p=3,& d=1,& v=3,& w=2,& t=2.
\end{array}
\]
In (\ref{eq:dw=12}), if $p=4$, then $(v-2)(t-1)=0$. Hence we obtain the solutions
\[
\begin{array}{lllll}
    p=4,& d=1,& v=2,& w=2,& t\geq 1, \\
    p=4,& d=1,& v\geq 1,& w=2,& t=1.
\end{array}
\]
Finally, (\ref{eq:dw=21}) is symmetric to (\ref{eq:dw=12}). Hence in this case we obtain the solutions
\[
\begin{array}{lllll}
    p=0,& d=2,& v=2,& w=1,& t=6,
    \\
    p=0,& d=2,& v=3,& w=1,& t=4, \\
    p=0,& d=2,& v=5,& w=1,& t=3, \\
    p=1,& d=2,& v=2,& w=1,& t=5, \\    
    p=1,& d=2,& v=4,& w=1,& t=3, \\
    p=2,& d=2,& v=2,& w=1,& t=4, \\    
    p=2,& d=2,& v=3,& w=1,& t=3, \\
    p=3,& d=2,& v=2,& w=1,& t=3, \\
    p=4,& d=2,& v\geq 1,& w=1,& t=2, \\
    p=4,& d=2,& v=1,& w=1,& t\geq 1.
\end{array}
\]
Now using $d=u-t$ and (\ref{eq:change of variables}), we may translate all of the solutions in the case $dw\leq 2$ to the following solutions of the system (\ref{eq:1})--(\ref{eq:2}):
\[
\begin{array}{lllll}
    p=0,& a=16,& k=6,& b=8,& l=12, \\
    p=0,& a=18,& k=4,& b=6,& l=12, \\
    p=0,& a=18,& k=10,& b=12,& l=15, \\
    p=0,& a=25,& k=3,& b=5,& l=15, \\
    p=1,& a=14,& k=5,& b=7,& l=10, \\
    p=1,& a=15,& k=4,& b=6,& l=10, \\
    p=1,& a=15,& k=8,& b=10,& l=12, \\
    p=1,& a=20,& k=3,& b=5,& l=12, \\
    p=2,& a=12,& k=4,& b=6,& l=8, \\
    p=2,& a=12,& k=6,& b=8,& l=9, \\
    p=2,& a=15,& k=3,& b=5,& l=9, \\
    p=3,& a=2,& k=1,& b=2,& l=1, \\
    p=3,& a=9,& k=4,& b=6,& l=6, \\
    p=3,& a=10,& k=3,& b=5,& l=6, \\    
    p=4,& a=3,& k=2,& b=3,& l=2, \\
    p=4,& a=4,& k=1,& b=2,& l=2, \\
    p=4,& a=t+2,& k=t,& b=t+2,& l=t, \\
    p=4,& a=2v,& k=2,& b=4,& l=v.
\end{array}
\]
Hence by setting
\begin{align*}
    S_4\coloneqq \{&(0,16,6,8,12), (0,18,4,6,12), (0,18,10,12,15), (0,25,3,5,15), (1,14,5,7,10),(1,15,4,6,10),\\
    &(1,15,8,10,12), (1,20,3,5,12), (2,12,4,6,8), (2,12,6,8,9), (2,15,3,5,9), (3,2,1,2,1),\\
    &(3,9,4,6,6), (3,10,3,5,6), (4,3,2,3,2), (4,4,1,2,2), (4,x+2,x,x+2,x)_{x\geq 1},(4,2x,2,4,x)_{x\geq 1}\}
\end{align*}
we obtain $S_4\cup S_4^{\ast}\subseteq \mathbf{S}$. 

Finally notice that the solutions in the case $dv\leq 2$ are given by the set $S_4^{\ast}$. This completes all the cases and hence we have
\[
\mathbf{S} = \left(\bigcup_{i=1}^4 S_i\right)\cup \left(\bigcup_{i=1}^4 S_i^{\ast}\right).
\]

\section{Classification in the regular and the edge-transitive case}
\label{Section:Classification in regular and edge-transitive case}
The aim of this appendix is to prove the results of Section \ref{sec:classification results}. Note that in this appendix, as suggested by our stated goal, we have as a standing assumption that the base field $k$ has characteristic zero. Some of the claims do not need this assumption as it is mainly utilised when using \cite{QPA} to check whether certain higher preprojective algebras are selfinjective or to compute certain extension groups. 

Assume that $\Lambda=\K S_{(r,s)}/\I$ is a semi-regular $(r,s)$-star algebra of bidegree $(\Sigma_1,\Sigma_2)$. Then we know that the quiver $B_{\Lambda}$ and the unique indecomposable $B_{\Lambda}$ module $M$ with dimension vector $(1,\ldots,1)^T$ satisfy the conditions of Setup \ref{setup:graph, quiver, one-point extension} by Theorem \ref{thrm:twisted fractionally Calabi-Yau semi-regular star algebra}. Assume moreover that $\Lambda$ is $2$-representation-finite and balanced; recall that the only non-balanced $2$-representation-finite $(r,s)$-star algebra appears when $(r,s)=(1,1)$ by Proposition \ref{prop:star algebras which are 2-hereditary}. Then Corollary \ref{cor:2-RF semi-regular star algebra} gives that
\begin{equation}\label{eq:the conditions on sigmas and r s}
    r\Sigma_1 = s\Sigma_2 \quad\text{ and }\quad   \Sigma_1\Sigma_2+r+s-r\Sigma_1=p\in\{1,2,3,4\}.
\end{equation}
On the other hand, we know that $\Lambda$ is twisted fractionally Calabi--Yau by Theorem \ref{thm:twisted fractionally CY and n-RF}(a). Hence the algebra $\K\Gamma_{\Lambda}[M]$ is twisted fractionally Calabi--Yau by Lemma \ref{lem:twisted fractionally Calabi-Yau of trivial extension}. Therefore, we may apply Corollary \ref{cor:proposition about order of matrix holds for Gamma twisted fractionally Calabi-Yau} to the algebra $\Gamma_{\Lambda}$ and the graph $\graph{B_{\Lambda}}$. In particular, since by (\ref{eq:the conditions on sigmas and r s}) we have that $p$ is not equal to $0$, we obtain the following.

\begin{cor}\label{cor:all the possible cases}
    \begin{enumerate}
        \item[(a)] If $p=1$, then the order of $-\Phi_{\Gamma_{\Lambda}}$ is $6$.
        \item[(b)] If $p=2$, then the order of $-\Phi_{\Gamma_{\Lambda}}$ is $4$.
        \item[(c)] If $p=3$, then the order of $-\Phi_{\Gamma_{\Lambda}}$ is $3$.
        \item[(d)] If $p=4$, then $(r,\Sigma_1,s,\Sigma_2)=(4,2,4,2)$.
    \end{enumerate}
\end{cor}

Based on our claims in Section \ref{sec:classification results} we are interested in two specific cases: in the case where $B_{\Lambda}$ is regular and in the case where $B_{\Lambda}$ is edge-transitive. We start with the regular case. Our strategy for the classification is to go through the cases $p=1$, $p=2$, $p=3$ and $p=4$ separately in each of the regular and edge-transitive cases. 

\subsection{Regular case}

In this section, our goal is to present a classification of quadratic monomial $2$-representation-finite algebras $\Lambda$ in characteristic zero with associated graphs $\graph{B_{\Lambda}}$ being regular. Recall that by Corollary \ref{cor:the graph B_Lambda is reflexive and almost always Salem} the graph $\graph{B_{\Lambda}}$ is reflexive. Regular bipartite graphs which are also reflexive have been classified by Koledin and Stani\' c in \cite{KS14}. 
They show that any regular bipartite graph of degree at most $2$ is reflexive, and so is its bipartite complement. 
Moreover, if the degree of the graph exceeds $3$, there are only finitely many such graphs which are reflexive: in fact only $70$ cases need be considered, namely those in the set $\mathcal{R}^*$ defined in \cite{KS14}. 
Consulting our list of solutions, we see that we need only consider regular graphs with $18$ or fewer vertices.

\subsection*{Solutions with \texorpdfstring{$p=1$}{p=1}}
In this case, the only possible regular solutions are the Koszul dual pair $(8,3,8,3)$ and $(8,5,8,5)$. Hence, if we consider the former, we need to look at Tables 2 and 3 of \cite{KS13}, of which only $I_2$, $J_6$ and $J_7$ have the right number of vertices. 
The Coxeter matrices associated to the latter two are not of finite order, and hence they cannot correspond to examples. 

For $I_2$, however, the matrix does have finite order and satisfies $\Ext^1_A(D\Lambda,\Lambda) = 0$. 
Note that $I_2$ is also known as the M\"obius-Kantor graph. Using QPA, we can check that the preprojective algebra corresponding to $I_2$ is not finite dimensional. 
Alternatively, we can check that $\Ext^1_\Lambda(\tau_2 D\Lambda,\Lambda) \neq 0$. 

\subsection*{Solutions with \texorpdfstring{$p=2$}{p=2}}
Note that here we find the solution $(1,1,1,1)$ which corresponds to $A_3$ modulo the radical square. 
This is the first of the two examples in the planar case shown in \cite{ST24}. 
In fact, it is straightforward to check that this is the only possible quadratic monomial $2$-representation-finite algebra corresponding to this solution and that $B_{\Lambda}$ is the bipartite complement of the graph $P_2$. This gives the example (\ref{eq:the A3 example}).

Now, aside from the solution $(1,1,1,1)$, we only have the Koszul dual pair $(7,3,7,3)$ and $(7,4,7,4)$. 
Since the latter does not lie in $\mathcal{R}^*$, we must instead consult Theorem 8 and Tables 2 and 3 of \cite{KS13}.
The only options are the Heawood graph and $J_4$ and $J_5$ of Table 3. 
However, for $J_5$ one can check using QPA that $\Ext^1_{\Lambda}(\tau_2(D\Lambda), \Lambda) \neq 0$, or, alternatively, that its preprojective algebra is not finite-dimensional. 
Moreover, the Coxeter matrix of the algebra defined by $J_4$ does not have finite order. 

Using QPA, we can see that the algebras defined in (\ref{eq:the Heawood graph example}) and (\ref{eq:the bipartite complement of the Heawood graph example}) are $2$-representation-finite by checking that their preprojective algebras are selfinjective.

\subsection*{Solutions with \texorpdfstring{$p=3$}{p=3}}
By \cite[Lemma 4.24]{ST24}, the solution $(2,1,2,1)$ must have $\Ext^1_{\Lambda}(D\Lambda,\Lambda) \neq 0$.
Hence, we thus only need to check the Koszul self-dual solution $(6,3,6,3)$. 
This can only correspond to the graph $I_1$ from Table 2 of \cite{KS13}. While the algebra associated to this has a Coxeter matrix of finite order, one can check that $\Ext^1_{\Lambda}(D\Lambda,\Lambda) \neq 0$.

\subsection*{Solutions with \texorpdfstring{$p=4$}{p=4}}
While $p = 4$ contains an infinite family of solutions, by Proposition \ref{prop: p and the order of the Coxeter matrix when bi-eigenvector}, only the solution $(4,2,4,2)$ can correspond to a $2$-representation-finite algebra. 
For these parameters, one can check that the only possible graph is the cycle graph $C_8$. 
This corresponds to the second of the examples in the planar case shown in \cite{ST24} and to (\ref{eq:the C_8 example}) in this article.
\newline

\noindent By collecting all of the cases above, we see that we have proved Proposition \ref{prop: classification in the regular case}.

\subsection{Edge-transitive case}

Now assume that $\graph{B_{\Lambda}}$ is edge-transitive. Since $\graph{B_{\Lambda}}$ is a bipartite graph, this implies that there are at most two orbits of vertices for such a graph $G$ under the action of its automorphism group. Hence we deduce that $\graph{B_{\Lambda}}$ must be either semi-regular or regular. 
Using known classifications of edge-transitive graphs on up to $47$ vertices, we find that up to Koszul duality, there are only two graphs yielding $2$-representation-finite algebras: 
in the list available at Encyclopedia of Graphs, these are the graphs G-9P731 and G-9P730, with parameters corresponding to the solutions $(5,4,10,2)$ and $(6,3,9,2)$. 

In the classification of edge-transitive graphs at Encyclopedia of Graphs, for each solution $(r, \Sigma_1, s, \Sigma_2)$, we search for connected bipartite graphs with $r + s$ vertices and $r\Sigma_1$ edges. 
For each solution, we indicate some condition it fails to satisfy if the associated algebra is not $2$-representation-finite, and otherwise we indicate that QPA verifies that the higher preprojective algebra is selfinjective. 

\subsection*{Solutions with \texorpdfstring{$p=1$}{p=1}}
\begin{itemize}
    \item $(5,8,20,2)$: The only graphs matching these parameters are G-9PB76 and G-9PB77, and both of these have at least two vertices that have the same neighbours. Hence, by \cite{ST24}, the associated algebras have non-zero $\Ext^1_{\Lambda}(D\Lambda, \Lambda)$.
    \item $(6,5,10,3)$: Only G-9P77R matches these parameters, and the associated algebra has a Coxeter matrix of finite order, satisfies $\Ext^1_{\Lambda}(D\Lambda, \Lambda) = 0$, but has $\Ext^1_{\Lambda}(\tau_2 D\Lambda, \Lambda) \neq 0$.
    \item $(6,5,10,3)$: Only G-9P8H9 matches these parameters, but the Coxeter matrix of its associated algebra is not of finite order.
    \item $(7,4,14,2)$: Only G-9P8H6 matches these parameters, but it has at least two vertices that have the same neighbours. Hence, by \cite{ST24}, the associated algebra satisfies $\Ext^1_{\Lambda}(D\Lambda, \Lambda)\neq 0$. 
    \item $(10,3,15,2)$: Only G-9PB72 matches these parameters, but the Coxeter matrix of its associated algebra is not of finite order.
\end{itemize}

\subsection*{Solutions with \texorpdfstring{$p=2$}{p=2}}
\begin{itemize}
    \item $(5,6,15,2)$: The graphs G-9 and G-9P85G are the only ones that match these parameters, but the former is regular and hence we know it cannot yield an example by our classification in the regular case. Considering the latter, it is easy to see that at least two vertices have the same neighbours, and as before, we deduce that the associated algebra cannot be $2$-representation-finite. 
    \item $(6,4,8,3)$: Only the graphs G-9P6Z6, G-9P6Z7 and G-2HJMC match the parameters, but of these, all but G-9P6Z6 have at least two vertices that have the same neighbours. G-9P6Z6 has associated algebra satisfying $\Ext^1_{\Lambda}(D\Lambda, \Lambda) = 0$, but its Coxeter matrix is not of finite order. 
    \item $(6,4,12,2)$: Only G-9P7N0 and G-9P7N1 match the parameters here. The latter can be excluded as it has at least two vertices that have the same neighbours. The former has associated algebra satisfying $\Ext^1_{\Lambda}(D\Lambda, \Lambda) = 0$, but its Coxeter matrix is not of finite order.
    \item $(8,3,12,2)$: The only graph to consider in this case is G-9P85F, and for the associated algebra, we have both that $\Ext^1_{\Lambda}(D\Lambda, \Lambda) \neq 0$ and that its Coxeter matrix is not of finite order. 
\end{itemize}

\subsection*{Solutions with \texorpdfstring{$p=3$}{p=3}}
\begin{itemize}
    \item $(5,4,10,2)$: Here, only G-9P731 satisfies the parameters and QPA verifies that its higher preprojective algebra is selfinjective. This corresponds to Example \ref{eq:the graph G-9P731 example} and the bipartite complement corresponds to Example \ref{eq:the bipartite complement of the graph G-9P731 example} via Koszul duality.
    \item $(6,3,9,2)$: In this case, only G-9P730 satisfies the parameters and QPA verifies that its higher preprojective algebra is selfinjective. This corresponds to Example \ref{eq:the graph G-9P730 example} and the bipartite complement corresponds to Example \ref{eq:the bipartite complement of the graph G-9P730 example} via Koszul duality. 
\end{itemize}

\noindent By collecting all of the cases above, we see that we have proved Proposition \ref{prop: classification in the edge-transitive case}.

\bibliography{bib.bib}
\bibliographystyle{alpha} 

\end{document}